\documentclass[a4paper,11pt,leqno]{article}
\usepackage{amsfonts,amssymb,amsmath,amsthm}
\usepackage{enumerate}
\usepackage{a4wide}
\usepackage{color}
\usepackage{eucal}
\usepackage{tikz-cd}

\newtheorem{theorem}{Theorem}[section]
\newtheorem{proposition}[theorem]{Proposition}
\newtheorem{corollary}[theorem]{Corollary}
\newtheorem{lemma}[theorem]{Lemma}
\theoremstyle{remark}
\theoremstyle{definition}
\newtheorem{remark}[theorem]{Remark}
\newtheorem{example}[theorem]{Example}
\newtheorem{definition}[theorem]{Definition}

\numberwithin{equation}{section}

\newcommand{\R}{{\mathbb R}}

\newcommand{\beq}{\begin{equation}}
\newcommand{\eeq}{\end{equation}}
\newcommand{\red}{\textcolor{red}}
\begin{document}

\title{Criticality theory for Schr\"odinger operators \\ 
        with singular potential 
\thefootnote\relax\footnote{
This paper was partially supported by an Indo-US project under the VIMMS and 
the grant MTM2014-52402-C3-1-P (Spain).
The first author was supported by the Simons Foundation Collaboration Grant for Mathematicians 210368.}
}

\author{
{\sc Marcello Lucia} \\
{\sc\footnotesize  The City University of New York, CSI}\\
{\sc\footnotesize Mathematics Department}\\
{\sc\footnotesize 2800 Victory Boulevard}\\
{\sc\footnotesize Staten Island, New York 10314, USA}\\
{\tt\footnotesize marcello.lucia@math.csi.cuny.edu}
 \and 
 {\sc              S. Prashanth} \\
 {\sc\footnotesize TIFR-Centre For Applicable Mathematics}\\
 {\sc\footnotesize Post Bag No. 6503, Sharada Nagar,}\\
 {\sc \footnotesize GKVK Post Office,}\\
 {\sc\footnotesize  Bangalore 560065, India} \\
 {\tt\footnotesize pras@math.tifrbng.res.in}\\
}

\date{\today}

\maketitle

\begin{abstract}

%
In a seminal work, B. Simon provided a classification of nonnegative Schr\"odinger operators $-\Delta+V$ 
into subcritical and critical operators based on the long-term behaviour of the associated heat kernel. 
Later works by others developed an alternative subcritical/critical classification based on whether or not the operator admits a weighted spectral gap. All these works dealt only with  potentials that ensured the validity of Harnack's inequality, typically potentials in $L^p_{loc}$ for some $p > \frac{N}{2}$.

 This paper extends  such a weighted spectral gap classification to a large class of locally integrable potentials $V$  called balanced potentials here.  Harnack's inequality will not hold in general for such potentials.  In addition to the standard potentials in  the space $L^p_{loc}$ for some $p > \frac{N}{2}$,  this class contains  potentials which are locally bounded above   or more generally,  those for which the  positive and negative singularities are  separated as well as  those for which the interaction of the positive and negative singularities is not ``too strong". 

We also establish for such potentials, under the additional assumption  that $V^+ \in L^p_{loc}$ for some $p > \frac{N}{2}$, versions of the Agmon-Allegretto-Piepenbrink Principle characterising the subcritical/critical operators  in terms of the cone of  positive  distributional super-solutions (or solutions).\\

{\bf A compressed version of this work has appeared in \cite{MPJDE}. We also remove some restrictions imposed in \S 10.3 of \cite{MPJDE}}

\bigskip


\medskip


\end{abstract}

\section{Introduction}\label{s1}
\subsection{ A brief review of the literature}\label{s1.1}

Understanding the nature of critical points of functionals has been one of the principal aims in the calculus of variations.
Jacobi~\cite{Jacobi} was probably the first one to realise that the non-negativity of the second variation of a functional is closely related to the existence of nontrivial  nonnegative  solutions to the associated linear equation. This  led to his well-known  ``non-conjugacy" condition valid for one-dimensional functionals of the type $J(u) = \int_0^1 f(t, u, u') dt$. 
Since then, analogous results  in higher dimensions for linear second order symmetric elliptic operators with a potential term  have been investigated by many authors. When the potential is well behaved, it has been shown  that the  oscillatory properties of solutions to such operators  are closely related to the Morse index of the associated quadratic form (see for instance \cite{Allegretto:1974}, \cite{Pip} and \cite{Piepenbrink:1974}). 

\medskip

In this work, we consider  Schr\"odinger operators of the type $-\Delta + V$ in a domain 
$\Omega$ of $\R^N$, $N \geq 1$ with a potential $V \in L^1_{loc}(\Omega)$. 
The properties of this operator will be related to the following associated quadratic form:
$$
   Q_V (\xi) = \int_{\Omega}\left\{  |\nabla \xi |^2 + V \xi^2 \right\} ,
   \quad
   \xi \in W_c^{1,\infty} (\Omega),
$$
where $W^{1,\infty}_c(\Omega)$ stands for the space of Lipschitz continuous functions with compact support in $\Omega$. 
To avoid any confusions, we introduce some basic definitions.

\begin{definition}\label{darr}
We shall call $-\Delta+V$ a  {\em nonnegative operator} in $\Omega$ if $Q_V(\xi) \geq 0$ for all  $\xi \in W_c^{1,\infty} (\Omega)$.
\end{definition}

\begin{definition} \label{def:DistributionalSolution}
We say that $u \in L^1_{loc}(\Omega)$ is a {\em distributional supersolution} to $-\Delta+V$  in $\Omega$ if
\begin{equation}\label{distr}
Vu \in  L^1_{loc}(\Omega) \; \text{ and } \; 
\int_\Omega u \, \big\{ - \Delta \xi \,+ \,  V \xi \big\}  \,dx
 \, \ge \, 
0,
\qquad\forall\:0\le\xi\in C^2_c(\Omega).
\end{equation}
A  {\em distributional solution} is a function $u \in L^1_{loc}(\Omega)$  such that \eqref{distr} holds for all $\xi \in C^2_c(\Omega)$. 

We shall say that $u \in L^1_{loc} (\Omega)$ is a {\it strict} distributional supersolution to $-\Delta+V$ in $\Omega$  if  it is a distributional supersolution  but not a distributional solution. 
\end{definition}

We define the following cone of all non-negative super-solutions to $-\Delta+V$:
\begin{definition}\label{infinitea}
Let
$${\cal C}_V (\Omega):=\big \{ u \in L^1_{loc}(\Omega): u \not \equiv 0, u \geq 0 \text{ and } u 
\text{ a distributional supersolution to } -\Delta +V \text{ in } \Omega \big \}.$$
\end{definition}

It is easy to see that if $u \in {\cal C}_V(\Omega)$, then $-\Delta u +Vu$ 
is a non-negative Radon measure $\mu$.
\begin{definition}\label{reisz}
Given $u \in {\cal C}_V(\Omega)$, the  measure $\mu:=-\Delta u +Vu$ will be called the Riesz measure associated to $u$. We let ${\cal M}_V (\Omega):= \{ \mu: \mu \hbox{ is the Riesz measure of some } u \in {\cal C}_V(\Omega)\}$.
\end{definition}

In this work, we shall label as the ``AAP Principle" (``Agmon-Allegretto-Piepenbrink Principle") any statement asserting the equivalence between the  strength of ``non-negativity" of the 
Schr\"odinger operator $-\Delta + V$ and the cone ${\cal C}_V (\Omega)$  (or the Riesz measure space ${\cal M}_V (\Omega)$).

The contributions of Agmon \cite{A}, Allegretto~\cite{Allegretto:1974}, Gesztesy-Zhao \cite{GeZhao}, Piepenbrink~\cite{Piepenbrink:1974}, Pinchover-Tintaraev \cite{PinTint}, and the regularity relaxations on $V$  brought 
by them  have led to the following version of the ``AAP principle" (see Simon \cite[Thm. C.8.1]{Simon}):

\medskip

{\it  
Assume that 
\begin{equation} \label{eq:SimonAssumptionAAP}
  V^{+} \hbox{ is in the local Kato class  and }  \;
  V^-   \hbox{  in  Kato class in } \Omega.
\end{equation} 
Then, the quadratic form $Q_V$ is non-negative in $\Omega$  if and only if the equation  $(-\Delta + V)u =0$ admits a positive distributional solution in $\Omega$ i.e., iff  $0 \in {\cal M}_V (\Omega)$.}

\medskip

Using the well-known Picone's identity, one can easily show that the existence of a positive distributional solution 
implies that $Q_V$ must be non-negative. The converse is slightly more difficult, and relies on the Harnack's principle which holds if $V$ satisfies \eqref{eq:SimonAssumptionAAP}. 
As pointed out by Fischer-Colbrie and Schoen, this principle extends to a complete manifold 
(see~\cite[Thm.~1]{Fischer-ColbrieSchoen:1980}), and they used it to classify stable minimal surfaces in $3$-manifolds of non-negative scalar curvature. For analytical applications, AAP type principles can be used to show various integral inequalities in Sobolev spaces.

\medskip

Simons \cite{Simon:1981} also introduced a classification of non-negative operators $-\Delta+V$ into subcritical and critical operators according to the 
long-time behavior of the associated heat kernel. In~\cite{Murata:1986}, Murata
rephrased Simon's idea from a different perspective.  
He considered potentials 
\begin{equation} \label{eq:MurataPotential}
V \in   L^p_{loc} (\Omega) 
   \quad \hbox{ with } 
   \left\{ 
   \begin{array}{ll} 
   p > N/2 \, &\hbox{ for } N \geq 2 , \\
   p=1 \, &\hbox{ for } N=1 ,
   \end{array}
   \right.
\end{equation} 
and classified  non-negative operators $-\Delta+V$ by  the sign of their fundamental solutions. He defined 
any such operator to be {\it subcritical} in $\Omega$ if its fundamental solution at any point in $\Omega$ is positive in $\Omega$, and defined it to be 
{\it critical} in $\Omega$  otherwise. 
The Laplacian in $\mathbb R^N$ provides a very simple example to illustrate this classification: 
it is critical for $N=1,2$ (all fundamental solutions are of the type $-|x|$ and $- \log |x|$ which are not positive), whereas it is subcritical for $N \geq 3$ (since it has a positive fundamental solution of the type $|x - y|^{2-N}$). 
On Riemannian manifolds  this classification has led to the notions of {\it parabolic manifold} 
(versus {\it non-parabolic manifold}), 
depending on whether the Laplace-Beltrami operator 
 admits a positive fundamental solution at every point or not (see~\cite{Grigoryan:1987}).

\medskip

A key feature of the classification of Murata is that subcritical/critical operator can be characterized in terms of the quadratic form $Q_V$ as follows (see \cite[Thm 2.4]{Murata:1986}): 

\medskip

{\it A nonnegative  operator $-\Delta+V$ is subcritical in $\Omega$ if and only if
for any nonnegative function $w \not \equiv 0$ with compact support satisfying~\eqref{eq:MurataPotential}, there exists $\varepsilon >0$ such that 
$Q_V (\xi) \geq \varepsilon \int_{\Omega}  w \xi^2$ for all $\xi \in C_c^2 (\Omega)$.
}

\medskip

Furthermore, Murata also showed that this dichotomy is reflected in a very interesting way in the structure of positive solutions:

{\it A nonnegative operator $-\Delta+V$ is critical in $\Omega$ if and only if the set of all positive solutions to $(-\Delta+V)u =0$ in $\Omega$ is given by  positive scalar multiples of a fixed positive solution. }

\medskip

Namely, in this case  the equation $(-\Delta+V)u = 0$ admits up to scalar multiplication a unique positive solution which  is usually called the ``ground state".
In other words, for subcritical operators  the associated quadratic form  satisfies a 
weighted Poincar\'e inequality, or possesses a ``spectral gap", whereas critical operators are the ones for which the non-negativity of the quadratic form  cannot be improved at all. In particular, to show that $Q_V$  can be improved, it is enough to find two positive linearly independent solutions to $(-\Delta+V)u =0$ (see for instance lemma \ref{min}). This yields a strategy to see how far
some well-known inequalities, for example classical ``Hardy inequality" on a bounded domain $\Omega_*$ of $\mathbb R^N$, $N \geq 3$,
$$
   \int_{\Omega_*}  |\nabla \xi|^2 - H_N |x|^{-2} \xi^2  \geq 0  \; \forall \; \xi \in C_c^2(\Omega_*)
$$
where $H_N=  (\frac{N-2}{2})^2$, can be improved (a question raised by Brezis and Vazquez). 

\medskip

Murata's classification has been extended to non-symmetric elliptic operators with H\"older continuous coefficients by Pinchover \cite{Pin:1988}. Extension to the $p$-Laplace operators have been studied by 
Pinchover-Tintarev~\cite{PT} and \cite{PP}. 

\medskip

We remark that in one dimension, a rather exhaustive analysis  of nonnegative operators $-d^2/dx^2+V$ has been done  by Gesztesy and Zhao \cite{GeZhao} 
 under the only assumption that $V$ is locally integrable; see theorems 3.1 and 3.6 there. Therefore, in this work we will mainly restrict our attention to dimensions bigger than one.

\medskip

All the above works have one common feature: they impose regularity assumptions on the potential $V$ (and the coefficients of the operator) that allow the application of   Harnack's inequality, which in particular guarantees that the solutions to $(-\Delta+V)u = 0$ are positive and continuous. 

\medskip

One maybe interested in frameworks that allow 
a study of  operators which fail to satisfy the Harnack's inequality or do not admit a fundamental solution at every point. More specifically, it will be interesting to extend the subritical/critical classification and obtain the corresponding AAP Principle to the operator $-\Delta+V$ for  a  larger class of locally integrable functions than previously considered. 

In this context we summarise a part of  recent work of Jaye, Mazya and Verbitsy \cite{JayeMazyaVerb} which treats  potentials that are in the space $H^{-1}_{loc}$.
Given a distribution $\sigma$ on $\Omega$, we can define the following quadratic form
$$
Q_\sigma(\xi) = \int_\Omega |\nabla \xi|^2 + \langle \sigma, \xi^2 \rangle \;\; \forall \xi \in C_c^{\infty}(\Omega).
$$
\begin{definition}\label{maz}
We say that a distribution $\sigma \in {\cal D}^{\prime} (\Omega)$ is \\
(i) form-bounded if there exist constants $\beta_1,\beta_2>0$ such that
$$
-\beta_1 \int_\Omega |\nabla \xi|^2 \leq  \langle \sigma, \xi^2 \rangle \leq \beta_2 \int_\Omega |\nabla \xi|^2 \;\; \;\; \forall \xi \in C_c^{\infty}(\Omega);
$$
(ii) upper (lower) bounded if  the upper (lower) bound above holds.
\end{definition}
We straightaway remark that if  $\sigma$ is form bounded, then in fact $\sigma \in H^{-1}_{loc}(\Omega)$ 
(see \cite[Lemma 2.6]{JayeMazyaVerb}).
We note that the quadratic form $Q_\sigma$  is non-negative if and only if $\sigma$ is lower bounded with the constant $\beta_1\leq1$. 
The work \cite{JayeMazyaVerb} deals with the problem of finding equivalent and easier to verify conditions for the form-boundedness of a distributional potential $\sigma$.  They prove the  following AAP-Principle  :
\begin{theorem}\label{swiss}
If $\sigma \in  {\cal D^\prime}(\Omega)$ is form-bounded with $\beta_1<1$, then there exists an a.e. positive solution $u \in H^1_{loc}(\Omega)$ to the following problem
$$-\Delta u + \sigma u =0 \hbox{ in } {\cal D^\prime}(\Omega).$$
If  $\sigma$ is a negative Borel measure, then $\sigma$ is lower bounded (and hence form-bounded) with $\beta_1=1$ iff there exists a positive superharmonic supersolution to the above equation.
\end{theorem}
 More generally, they prove the equivalence of form boundedness of $\sigma$  to the existence of a solution to the 
 Riccati-type equation  $-\Delta u = |\nabla u|^2 - \sigma \hbox{ in } {\cal D^\prime}(\Omega)$ which additionally satisfies a Caccioppoli-type inequality (see \cite[Thm. 3.14]{JayeMazyaVerb}). But the subcritical/critical classification of nonnegative operators $-\Delta+\sigma$ is not treated in this work.

We remark that in theorem \ref{swiss}, the assumption $\beta_1<1$ implies a strong form of positivity, namely, that the pre-Hilbert space $(C_c^{\infty}(\Omega), \sqrt{Q_\sigma})$ is continuously imbedded in $H^1_{loc}(\Omega)$. 

\medskip

In this work, as we said before,  our aim is to   provide a subcritical/critical classification for nonnegative operators   with locally integrable potentials  $V$ as well as show  the corresponding AAP-Principle. Using the above terminology, our  main assumption is that  $V$ is a ``balanced" potential in $\Omega$ (see definition \ref{tri}, \ref{tre}) which in particular requires that $V$ is lower bounded with $\beta_1 \leq1$ but not necessarily that it is upper bounded. In the latter half of the work  we assume additionally that $V^+$ satisfies the \; stronger local integrability property : $V^+ \in L^p_{loc}(\Omega)$ for some $p > \frac{N}{2}$; that is,  $V$ is ``locally" upper bounded  in $\Omega$. In fact, in our setting $V$ will be form-bounded in dimensions bigger than 2, and hence will belong to $H^{-1}_{loc}(\Omega)$, if for instance $V^+ \in L^{\frac{N}{2}}(\Omega)$. More importantly, in contrast to \cite{JayeMazyaVerb}, for subcritical operators we only demand that the pre-Hilbert space $(W_c^{1,\infty}(\Omega), \sqrt{Q_V})$ imbeds continuously into  $L^1_{loc}(\Omega)$. A typical example is  the critical Hardy operator
$$
-\Delta -H_N|x|^{-2}
$$
 in a ball  $B \subset {\mathbb R}^N, \; N \geq 3,$ containing the origin. In fact, we base all our development on this weak assumption about the pre-Hilbert space which necessitates a restriction to balanced potentials  as mentioned above. Nevertheless, we will find that almost all the well studied potentials fall into this category.

\medskip

Another  motivation to consider potentials $V$ which are at least locally integrable (and not more irregular) is the fact that in dimensions $N \geq 3$,  the study of  form bounded  distributional potentials can be reduced to  that of  nonnegative quadratic forms $Q_V$ with $V$ nonpositive and  locally integrable; see theorem 4.1 in \cite{JayeMazyaVerb}.

\medskip 

{\bf Notations:}
\vspace{-5pt}
\begin{enumerate}
\item[(i)] 
For a set $A$, we write $A \Subset \Omega$ if $\bar{A}$ is a compact subset of $\Omega$,
\vspace{-5pt}
\item[(ii)] 
The space of distributions on $\Omega$ will be denoted as ${\cal D^\prime}(\Omega)$,
\vspace{-5pt}
\item[(iii)] 
For a function  $w \in L^1_{loc}(\Omega)$, we will write ``$Q_V \succeq w$ in $\Omega$" to express the condition: 
$$
    Q_V (\xi) \geq \int_{\Omega} w \xi^2, \; \forall  \xi \in W_c^{1,\infty} (\Omega).
$$
\vspace{-5pt}
\item[(iv)] 
For a function $f$ its support $\overline{\{f \neq 0\}}$ will be denoted as ${\rm supp}(f)$, and 
given a space $X$ of functions on $\Omega$, we let $X_c$ denote its subspace consisting of those with compact support in $\Omega$.
\vspace{-5pt}
\item[(v)] $``\inf", ``\sup"$ will denote respectively the essential infimum and supremum.
\vspace{-5pt}
\item[(vi)]  $\|\cdot\|_{p,A}$ will denote the $L^p$-norm on a set $A$.
\vspace{-5pt}
\item[(vii)] For $N \geq 3$ and $1 \leq p < \infty$, the space ${\mathcal D}_0^{1,p}(\Omega)$ is the completion of $C_c^{\infty}(\Omega)$ with respect to the norm:
$\displaystyle  \bigg(\int_\Omega |\nabla \cdot |^p\bigg)^{\frac{1}{p}}$.
\item[(viii)] Given two topological spaces $X,Y$, we write $X \hookrightarrow Y$ to mean that $X \subset Y$ and the inclusion map is continuous.
\end{enumerate}

\subsection{Contents of the paper}\label{s1.2}
In \S \ref{s2} we recall some preliminary results required in the rest of the paper.

\smallskip 

In \S \ref{s3} we show how the energy space ${\cal H}_V(\Omega)$ can be constructed from a given nonnegative quadratic form $Q_V$ on $\Omega$ under the only assumption that $V \in L^1_{loc} (\Omega)$. 
Some results  relating the singularities of $V$ to the ``size" of the energy space are proved, and we also introduce the concept of an ``energy solution in ${\cal H}_V(\Omega)$" to the equation $-\Delta u+Vu =f$.

\smallskip 

We introduce in  \S \ref{s4}  two types of subcritical/critical classification for  nonnegative operators $-\Delta +V$ on $\Omega$  based on whether the energy space imbeds into $L^1_{loc}$ or $L^2_{loc}$ space. The former imbedding leads to the $L^1$-subcritical/critical set of operators and the latter into feebly and globally $L^2$- subcritical/critical set of operators. Given that we work with $V$ whose negative part $V^-$ is especially unrestricted, it is nontrivial to show that these three notions of criticalities coincide (see \S \ref{s9}).

\smallskip 

In \S \ref{s5} we explore the various imbeddings of the energy space ${\cal H}_V(\Omega)$ when $-\Delta+ V$ is $L^1$-subcritical in $\Omega$. In this a special reference can be made to \eqref{eq:sloka} which resembles a statement of continuity for the map $\xi \mapsto \int_K V \xi, \; K \Subset \Omega$ on the energy space ${\cal H}_V(\Omega)$ with a correction term. 

We introduce the notion of ``balanced" and ``tame" potentials $V$ in \S \ref{s6}  which are central to the paper.  We show that the concept of balanced potentials covers well-known ones as well as many that are not usually considered. We show that for balanced potentials $V$ that are also $L^1$-subcritical, the concept of energy solutions is equivalent to that of distributional solutions.

\smallskip 

\S \ref{s7} deals with AAP principle for a general nonnegative operator $-\Delta+V$. Assuming $V \in L^1_{loc}(\Omega)$, we show that existence of a nonnegative  (distributional) super solution to this operator implies that the associated quadratic form $Q_V$ is nonnegative. The converse is proved under the additional condition that $V^+ \in L^p_{loc}(\Omega)$ for some $p>\frac{N}{2}\;(N\geq 2)$,  and relies on the $L^1$-theory developed by Stampacchia for the operator $-\Delta+V^+$.

\smallskip 

We exploit the  result on equivalence of energy and distribution solutions to develop an AAP principle in \S \ref{s8} for $L^1$-subcritical operators with balanced potentials.

\smallskip 

Using the results from the previous sections, in subsection \S \ref{s9.1} we show that the  notions of  $L^1$- and globally $L^2$-criticalities  (introduced in \S \ref{s4}) coincide for nonnegative operators with balanced potentials. In subsection \S \ref{s9.2} we show that notions of  $L^1$- and feebly $L^2$-criticalities too coincide  for the class of tame potentials. 
\smallskip 

In  \S \ref{s10}, we exploit this equivalence between the two notions of criticality for tame potentials and provide characterisations (in the spirit of the AAP principle)  of nonnegative subcritical operators in terms of their  Riesz measures (see theorems \ref{fen} and \ref{fenu}) as well as the dimension of the cone of nonnegative  (distributional) super solutions (see theorem~\ref{rine}).

\smallskip 

Finally in \S \ref{s11} we give applications and examples of our results.

\section{Some Preliminaries}\label{s2}
 
We first recall the concept of harmonic capacity (called here as $``Cap "$):
\begin{definition}\label{def:harcap}
For a compact set $K \subset \Omega$, the harmonic capacity of $K$ with respect to $\Omega$ is defined as:
$$
Cap(K;\Omega):= \inf \Big \{ \int_\Omega |\nabla \xi|^2: \xi \in C_c^{\infty}(\Omega),  \xi \geq 1 \hbox{ on } K \Big \}.
$$
\end{definition}

The above definition can be used to define the capacity $Cap(A;\Omega)$ of any subset $A$ of $\Omega$. In particular, the following fact is well known:\\
If $Cap(A ; \Omega)=0$ for any bounded subset $A$ of $\Omega$ then $Cap(A \cap U ; U)=0$ for any open  set $U \subset \R^N$.
In such a situation, we will simply say $Cap(A)=0$.

We have then the following characterisation for the validity of $Q_V \succeq 0$ when $V$ is non-positive (see chapter 2 in \cite{MazBook}):

\begin{theorem}\label{mazya}
Let $V \in L^1_{loc}(\Omega)$ with  $V^+ \equiv 0$. Then $Q_V \succeq 0$ in $\Omega$ if and only if the following inequality holds:
$$\frac{1}{4}  \leq \sup_{K \in {\mathcal N}}  \frac{1}{Cap(K;\Omega)} \int_K V^- \leq 1 $$
where
$\displaystyle  {\mathcal N} := \{K: K \text{ is a compact subset of }\;  \Omega \text{ with } Cap(K;\Omega)>0\} $.
\end{theorem}

The main trouble in using the above characterisation is that the calculation of capacity of sets is hard. In this context, we cite the following more concrete characterisation:

\medskip

\begin{theorem}\cite[theorem 5.1]{JayeMazyaVerb}\label{vat}
$$
\Big \{ V \in L^1_{loc}(\Omega) :  Q_{V} \succeq 0 \;\text{ in }\; \Omega \Big \}= \Big \{ V \in L^1_{loc}(\Omega) : V \geq  div \; \vec{{\bf f}}+ |\vec{{\bf f}}|^2 \;\text{ for some }\; \vec{{\bf f}} \in (L^2_{loc}(\Omega))^N \Big \}.
$$
The inequality above should be understood in the sense of distributions.
\end{theorem}

\medskip

\begin{definition}\label{dusera}
Let $V \in L^1_{loc}(\Omega)$. Given $f \in L^1_{loc}(\Omega)$, we say that a function $u \in L^1_{loc}(\Omega)$ is a distribution solution of $-\Delta u + Vu = f$ in $\Omega$ if 
$$
Vu \in L^1_{loc}(\Omega) \;\text{ and }\; \int_{\Omega} u \left(-\Delta \xi + V \xi \right )= \int_{\Omega} f \xi, \quad \hbox{ for all } \xi \in C^{\infty}_c(\Omega).
$$
\end{definition}
\medskip
In the following lemma we recall results on the solvability for a Schrodinger operator with a non-negative potential and the right-hand side data given in  $H^{-1}$.

\begin{lemma} \label{lem:GdeGde}
Let $\Omega_*$ be a bounded open set, $W \in L^{1}_{loc} (\Omega_*)$, $W \geq 0$, and $f \in H^{-1}(\Omega_*)$. Then the problem
\begin{equation} \label{eq:CiaoBella}
  - \Delta u + Wu = f
\end{equation}
admits a unique distribution solution $u \in H^1_0 (\Omega_*)$. Furthermore,  for any compact set $K \subset \Omega_*$, 
\begin{equation}\label{clean}
\|Wu\|_{L^1(K)} \leq 2 \|W\|_{L^1(K)}^{\frac{1}{2}} \| f \|_{H^{-1}(\Omega_*)}.
\end{equation}
If in addition $f \geq 0$, the solution is nonnegative.
\end{lemma}
\begin{proof}
Let $\langle \cdot ,\cdot \rangle$ denote the duality pairing between $H^{-1}(\Omega_*)$ and $H^1_0(\Omega_*).$  Consider the functional
$$
  J: H_0^{1} (\Omega_*) \to \mathbb R \cup \{ \infty \},
  \qquad 
  J(w)= \frac{1}{2} \int_{\Omega_*} \Big\{ |\nabla w|^2 + W(x) w^2 \Big\} dx
  -  \langle f , w \rangle .
$$
Since this functional is coercive and strictly convex, it admits a unique minimizer $u$ in the subspace $X := \big\{ u \in H^1_0 (\Omega_*) \, \colon \, J(u) < \infty \big\}$.
For each $\xi \in C_c^1 (\Omega_*)$, and $t \in \mathbb R$ we have $ u + t \xi \in X$ and hence
$$
  \frac{d}{dt} J( u + t \xi) \Big|_{t = 0} = 0 .
$$
This implies 
$$
   \int_{\Omega_*} \{ \nabla  u \cdot \nabla \xi + W u \xi \} =  \langle f , \xi \rangle  ,
   \quad \forall \xi \in C_c^1(\Omega_*).
$$
Furthermore, since $J( u) \leq J(0) = 0$ we deduce
$$
    \int_{\Omega_*} \Big\{ |\nabla  u|^2 + W(x)  u^2 \Big\} dx  \leq 2  \| f \|_{H^{-1}(\Omega_*)} \| u \|_{H^1_0(\Omega_*)} .
$$
Therefore, $ \| u \|_{H^1_0(\Omega_*)} \leq 2  \| f \|_{H^{-1}(\Omega_*)}$ and hence, $\|Wu^2 \|_{L^1(\Omega_*)} \leq  4  \| f \|_{H^{-1}(\Omega_*)}^2$.
So, on each $K \Subset \Omega_*$ we get, 
$$
    \int_{K} W | u| \leq   \left( \int_{K} W \right)^{1/2}  \left( \int_{K} W | u|^2 \right)^{1/2}  
    \leq 2 \|W\|_{L^1(K)}^{\frac{1}{2}} \| f \|_{H^{-1}(\Omega_*)} .
$$
Under the further assumption that $ \langle f , \xi \rangle \geq 0$ for all nonnegative $\xi \in H^1_0 (\Omega_*)$, 
we have $\langle f , |u| - u \rangle \geq 0$. In this case we deduce $J(u) \geq J(|u|)$, which proves that the unique minimizer is a nonnegative function.
\end{proof}

\medskip

Using the imbedding $L^{\infty}(\Omega_*) \hookrightarrow H^{-1} (\Omega_*)$, as a particular case of 
Lemma~\ref{lem:GdeGde} we have 
\begin{corollary} \label{cor:OdeOde}
Let $\Omega_*$ be a bounded open set, $W \in L^{1}_{loc} (\Omega_*)$, $W \geq 0$, and $f \in L^{\infty}(\Omega_*)$. Then the problem \eqref{eq:CiaoBella}
admits a unique  distribution solution $u \in H^1_0 (\Omega_*)$. Furthermore,  for any compact set $K \subset \Omega_*$, 
\begin{equation}\label{lean}
\|Wu\|_{L^1(K)} \leq 2 \|W\|_{L^1(K)}^{\frac{1}{2}} \| f\|_{L^{\infty}(\Omega_*)}.
\end{equation}
If in addition $f \geq 0$, the solution is nonnegative.
\end{corollary}

In order to solve with  a general measure as the right-hand side data, we need to go out of the space $H^1_0(\Omega)$. Indeed, we require the following concept of ``very weak" solutions:
\begin{definition}\label{VWS}
Let $V \in L^1_{loc}(\Omega)$ and $\mu$ be a finite Borel measure on $\Omega$. Then, $u \in L^1(\Omega)$ is said to be a very weak solution of the problem $-\Delta u + Vu = \mu$ in $\Omega$, $u=0$ on $\partial \Omega$ if $Vu \in  L^1(\Omega)$ and
$$
\int_{\Omega} u (-\Delta \xi + V  \xi )= \int_{\Omega} \xi d \mu \;\;\hbox{ for all } \; \xi \in C^2(\overline{\Omega}) \cap C_0(\overline{\Omega}).
$$
\end{definition}

\begin{remark}\label{phat}
This notion of solution is equivalent to asking that $u  \in W^{1,1}_0(\Omega)$  solve the above equation in the sense of distributions (see  \cite[Corollary 4.5]{PoCo}). 
\end{remark}

We then have the following fundamental result due to Stampacchia based on the duality method
(see \cite[th\'eor\`eme 9.1]{Stampacchia}):

\begin{proposition} \label{prop:train}
Let $\mu$ be a finite Borel measure on a bounded smooth domain $\Omega_*$ and  $W \in L^{p}(\Omega_*)$ for some $p >N/2$ be a nonnegative function. Then there exists a very weak solution  $u \in L^1(\Omega_*)$ of the problem $-\Delta u + W u = \mu$ in $\Omega_*$, $u=0$ on $\partial \Omega_*$. Furthermore, $u \in W^{1,p}_{0} (\Omega_*)$ for any $p \in [1, \frac{N}{N-1})$ and the following estimate holds:
\begin{equation}\label{eq:diwali}
\|u\|_{W^{1,p}_{0} (\Omega_*)} \leq C(p,\Omega_*) |\mu|(\Omega_*), \; \forall \; 1\leq p < \frac{N}{N-1}.
\end{equation}
A very weak solution satisfying the above properties will be identical to $u$.
\end{proposition}



The following is a local version of regularity for solutions with measure data :
\begin{proposition} \label{prop:RanRan}
Let $u$ be a distributional solution of $ -\Delta u = \mu$ for some Radon measure $\mu$  in $\Omega$.Then $u \in W^{1,q}_{loc} (\Omega)$ for any $q \in [1, \frac{N}{N-1})$.
\end{proposition}
\begin{proof} Fix any ball $B \Subset \Omega$.  Let $ B_1, B_2$ be  balls  with $B \Subset B_1 \Subset B_2 \Subset \Omega$. Define $\mu_1 := \mu \llcorner B_1$. Consider the problem 
$$
    - \Delta v = \mu_1 \hbox{ in } B_2, \; v=0 \hbox{ on } \partial B_2. 
$$
Then by  proposition \ref{prop:train} above, there exists a unique very weak solution $v$ of the above problem with $v \in W_0^{1,q}(B_2)$ for all $1 \leq q < \frac{N}{N-1}.$ Let $\psi:=u-v$. Then $\psi$ is harmonic in $B_1$ and hence is locally smooth there. Thus $u \in W^{1,q} (B)$ for any $q \in [1, \frac{N}{N-1})$.\end{proof}

The following weak maximum principle for distributional supersolutions is well known 
(see \cite[proposition B.1]{BPAnn}):
\begin{lemma}\label{lem:power}
Let $\Omega_*$ be a bounded smooth open set. If $u \in W^{1,1}_0(\Omega_*)$ satisfies $-\Delta u \geq 0$ in ${\cal D^\prime}(\Omega_*)$, then $u\geq 0$ a.e. in $\Omega_*$.
\end{lemma}

The following is a version of the strong maximum principle 
due to Ancona~\cite{Ancona} and Brezis, Ponce  \cite{BrezisPonce}(see theorem 1), but in a distributional setting:
\begin{theorem} \label{thm:StrongMax} {\bf (Strong Maximum Principle)}\\
Let $U \subset \R^N $ be a non-empty open connected bounded set and $W \in L^1_{loc} (U)$ with $W \geq 0$. Assume
$u\in L^1_{loc}(U)$ is an a.e. nonnegative quasi-continuous function such that $\Delta u$ is a Radon measure on $U$ and  $W u \, \in \, L_{loc}^{1} (U)$. If $u$ satisfies the differential inequality
$$
  - \Delta u + W u \, \geq \, 0 \mbox{ in } {\cal D}^ {\prime}(U),
$$
then either $u \equiv 0$ a.e., or else $Cap(\{u=0\})=0$. If the latter holds, we say $u>0$ q.a.e. (quasi almost everywhere).
\end{theorem}
We note that, since  $W u \, \in \, L_{loc}^{1} (U)$, from Remark 3 of \cite{BrezisPonce}  indeed we have $\Delta u \leq Wu$ in the sense of measures and theorem 1 of  \cite{BrezisPonce}  applies.

\section{Energy space ${\cal H}_V(\Omega)$} \label{s3}

Throughout this section we only assume $V \in L^1_{loc}(\Omega)$.

\begin{definition}\label{try}
Let $W^{1,\infty}_c(\Omega)$ be the vector space of Lipschitz continuous functions with compact support in $\Omega$. We define the quadratic form
$$Q_V(u):= \int_\Omega \Big\{ |\nabla u|^2 + V u^2 \Big\}, \quad u \in W^{1,\infty}_c(\Omega).$$
\end{definition}
\begin{lemma} \label{lem:Lip}
\begin{enumerate}
\item[{\rm \bf  (i)}]
 $Q_V $ is non-negative in $W^{1,\infty}_c(\Omega)$ whenever $Q_V$ is non-negative in 
$C_c^1 (\Omega)$.
\item[{\rm \bf (ii)}] 
Given $u \in W^{1,\infty}_c(\Omega)$, we have $|u| \in W^{1,\infty}_c(\Omega)$ and $Q_V (u) = Q_V (|u|)$.
\item[{\rm \bf  (iii)}]
Assume $Q_V $ is non-negative in   $ C_c^1(\Omega)$. 
Then, $Q_V$ is positive definite on $W^{1,\infty}_c(\Omega)$ and in particular on $C_c^1(\Omega)$.
\end{enumerate}
\end{lemma}
\begin{proof}
{\bf  (i)} 
Let $u \in W^{1,\infty}_c(\Omega)$. Let $\rho_{\varepsilon} \in C_c^{\infty} (\Omega)$ be an approximation of the unity and define 
$u_{\varepsilon}  := u \star \rho_{\varepsilon}$. 
Since $u$ has compact support, we have $u_{\varepsilon} \in C_c^{\infty}(\Omega)$ with all their supports contained in a fixed compact set in $\Omega$ for all $\varepsilon>0$ sufficiently small. Furthermore,
$u_{\varepsilon}$ converges uniformly to $u$ in $\Omega$ and $\nabla u_{\varepsilon} \to \nabla u$
 in $(L^2 (\Omega))^N$. We therefore conclude that $Q_V(u_{\varepsilon}) \to Q_V(u )$. Non-negativity of $Q_V$ in $W^{1,\infty}_c(\Omega)$ follows from its nonnegativity in $C_c^1(\Omega)$.
From the identity, 
$$Q_V(u-w) = Q_V(u)+Q_V(w) -2 \int_\Omega \Big\{ \nabla u \cdot \nabla w +V u w\Big\}$$
we infact obtain that $Q_V(u-u_\varepsilon) \to 0$ as $\varepsilon \to 0$.


\medskip

{\bf (ii)} Well known.

\medskip

{\bf (iii)}
Assume that $Q_V (u_0) = 0$ for some $u_0 \in W^{1,\infty}_c(\Omega)$. 
Then, $u_0$ is a minimizer for $Q_V$ on $ W^{1,\infty}_c(\Omega) $. Since 
$|u_0| \in W^{1,\infty}_c(\Omega)$ and $Q_V(u_0) = Q_V (|u_0|)$, we may assume that  
$u_0 \geq 0$  and  the following Euler-Lagrange equation holds:
$$
   \int_{\Omega} \nabla u_0 \cdot \nabla \xi + V u_0 \xi = 0,
   \qquad
   \forall  \xi \in C_c^{1} (\Omega).
$$ 
Therefore, 
$$
  \int_{\Omega} \nabla u_0 \cdot \nabla \xi + V ^+ u_0 \xi \geq 0,
   \qquad
   \forall  \xi \in C_c^{1} (\Omega), \; \xi \geq 0.
$$
By applying the strong maximum principle as stated in Theorem~\ref{thm:StrongMax}, we have the alternative: either 
$u_0 >0  \hbox{ a.e. }$ or $u_0 \equiv 0.$ 
Since $u_0$ has compact support, we conclude that indeed $u_0 \equiv 0$.

\end{proof}

\medskip
\begin{definition}\label{space}
Suppose $Q_V \succeq 0$ in $\Omega$, i.e., $Q_V$ is non-negative on $W^{1,\infty}_c(\Omega)$. From the above lemma we deduce that
$$
   x_V(\xi,\phi) := \int_{\Omega}\nabla \xi \cdot \nabla \phi + V \xi \phi
$$
defines an inner product on  $W^{1,\infty}_c(\Omega) \times W^{1,\infty}_c(\Omega)$ with the associated norm 
$ \sqrt{Q_V}$.

\smallskip

The pre-Hilbert space $(W^{1,\infty}_c(\Omega), \sqrt{Q_V})$ will be denoted ${\cal L}_V(\Omega)$ 
and its closure under the $\sqrt{Q_V}$ norm will be called 
the energy space ${\cal H}_V(\Omega)$ associated to the Schr\"odinger 
operator $-\Delta + V$. We denote the resulting  inner-product  on ${\cal H}_V(\Omega)$ as $a_{V,\Omega}$, and the associated norm  as $\| \cdot \|_{V,\Omega}$.   
\end{definition}

\begin{lemma} \label{lem:Dense}
$C_c^\infty (\Omega)$ is a dense subspace of ${\cal H}_V(\Omega) $.
\end{lemma}
\begin{proof}
From (i) in the proof of lemma \ref{lem:Lip} we obtain that $C_c^\infty(\Omega)$ is dense in ${\cal L}_V(\Omega)$. Density  in ${\cal H}_V(\Omega)$ now follows from the fact that ${\cal L}_V(\Omega)$ is a dense subspace of ${\cal H}_V(\Omega)$.
\end{proof}

In the following proposition, we state some relations between the singularities of $V$ and the ``size" of the corresponding ${\cal H}_V(\Omega)$ space.
\begin{proposition}\label{cong}
Let $U$ be an open set in $\R^N$, $N \geq 2$ which is restricted to be bounded if $N=2$. Assume  $V \in L^1_{loc}(U)$ to be such that $Q_V \succeq 0$ in $U$.
\begin{enumerate}
\item[{\rm \bf (i)}]
Let $V^+ \in L^{\frac{N}{2}}(U)$ if $N \geq 3$ and $V^+ \in L^q(U)$ for some $q>1$ when $N=2$.  Then 
\begin{equation}\label{dona}
(W^{1,\infty}_c(U) , \|\cdot\|_{{\mathcal D}^{1,2}_0(U)} )\hookrightarrow {\cal L}_V(U).
\end{equation}
\item[{\rm \bf (ii)}]
 If for some $1 \leq p <  2$,
\begin{equation} \label{eq:dona2}
   (W_c^{1,\infty}(U), \| \cdot \|_{{\mathcal D}_0^{1,p}(U)})  \hookrightarrow {\cal L}_V(U)
\end{equation}
then $V^- \not \in L^{\frac{N}{2}}(B)$ for  any ball $B \Subset U$. In particular, if $N=2$, the above imbedding is impossible.
\end{enumerate}
\end{proposition}
\begin{proof}
{\bf (i)}
Let $N \geq 3$. Using Sobolev imbedding we note that for any $\xi \in W_c^{1,\infty}(U)$,
$$
   0 \, \leq \, Q_V(\xi) 
      \, \leq  \, \int_U \Big\{ |\nabla \xi |^2 + V^+ \xi^2 \Big\}  
      \, \leq \, C \Big (1+\|V^+\|_{L^{\frac{N}{2}}(U)} \Big )\| \xi \|^2_{{\mathcal D}^{1,2}_0(U)}.
$$
A similar proof holds when $N=2$.
\medskip

{\bf (ii)}
We prove the contrapositive statement.
Assume $V^- \in L^{\frac{N}{2}}(B)$ for some ball $B \Subset U$ with center $x_0$.
Choose the sequence $\{\xi_k\}:$
$$
    \xi_k (x) := (1- k |x-x_0| ) \chi_{\{|x-x_0| \leq \frac{1}{k}\}},
$$
and note that $\{\xi_k\} \subset W_c^{1,\infty}(B)$ for $k$ large. 
A straightforward computation and H\"older's inequality give 
$$
   \int_{U} |\nabla \xi_k|^p \sim  k^{p-N}
   \qquad
   \int_{U} V^- \xi_k^2 \leq O_k(1) k^{2-N} 
   \| V^- \|_{L^{N/2} ( {\{|x-x_0| \leq \frac{1}{k}\}}  )} .
$$
Hence,  for some constant $c>0$,
\begin{eqnarray*}
   \frac{Q_V (\xi_k)}{\| \xi_k \|^2_{{\cal D}_{0}^{1,p} (U)  }  }
   &\geq& 
   \frac{c}{k^{2(1-N/p)}}\int_{U}  \Big\{ |\nabla \xi_k|^2 - V^{-} \xi_k^2 \Big\} 
   \, \geq \,  
   c k^{N(2-p)/p} .
\end{eqnarray*}
Since $p \in [1,2)$, 
it follows that the imbedding~\eqref{eq:dona2} cannot hold.
\end{proof}

\begin{corollary}\label{aba}
Let $V^+ \in L_{loc}^{\frac{N}{2}}(\Omega)$ if $N \geq 3$ and $V^+ \in L_{loc}^q(\Omega)$ for some $q>1$ 
when~$N=2$. Then, for each open set $U \Subset \Omega$ the imbedding~\eqref{dona}  holds and denoting  
by $T:{\mathcal D}^{1,2}_0(U) \to {\cal H}_V(U)$  the continuous linear extension of this imbedding, we have  
$$
\|T(u)\|^2_{V,U}=  \int_U \Big\{ |\nabla u|^2 + V u^2 \Big\}  \;\;\;\;  \forall  \; u \in H^1_c(U) \cap C_c(U).
$$
\end{corollary}
\begin{proof}  Consider the approximating sequence $\{u_\epsilon:= u \star \rho_\epsilon\}$ considered in (i) of lemma \ref{lem:Lip} above. Clearly, $u_\epsilon \to u$ in $H^1_0(U)$. Hence, we have that $\{u _\epsilon\}$ is a Cauchy sequence in ${\cal L}_V(U)$. Necessarily,  $u_\epsilon \to T(u)$ in ${\cal H}_V(U)$; that is, $Q_V(u_\epsilon) \to \|T(u)\|_{V,U}$. Since $u_\epsilon \to u$ in $C(\overline{U})$ and $ \nabla u_\epsilon \to \nabla u$ in $(L^2(U))^N$, we get
$$
Q_V(u_\epsilon) \to  \int_U \Big\{ |\nabla u|^2 + V u^2 \Big\}. 
$$
\end{proof}

In the energy space, we can consider the notions of solution and supersolutions in the following sense:
\begin{definition} \label{def:EnergySolution}
Let $V \in L^1_{loc}(\Omega)$ and $f \in  L^1_{loc}(\Omega)$.
We say that $u_* \in {\cal H}_V(\Omega)$ is an  ``energy supersolution" to $-\Delta u +V u =f$  if 
$$
    a_{V,\Omega} (u_{*}, \xi) \geq \int_\Omega  f \xi ,  
    \quad
    \forall \; \xi \in  W_c^{1,\infty}(\Omega),\; \xi \geq 0.
$$ 
($u_{*} \in {\cal H}_V(\Omega)$ will be called an  ``energy subsolution" if the inequality above is reversed.)\\
We say that $u_{*} \in {\cal H}_V(\Omega)$ is an  ``energy solution" to $-\Delta u +V u =f$  if
$$
    a_{V,\Omega} (u_{*}, \xi) = \int_\Omega f \xi,
    \quad
    \forall \xi \in  W_c^{1,\infty} (\Omega).
$$ 
\end{definition}

\section{Classification of Schr\"odinger operators : Two notions of Criticality} \label{s4}


\begin{definition}\label{nongu}
We call the quadratic form $Q_V$ to be supercritical in $\Omega$, if $Q_V  \not \succeq 0$ in $\Omega$; that is, $Q_V(\xi)<0$ for some $\xi \in W_c^{1,\infty} (\Omega)$.
\end{definition}

We introduce the following classification of criticality to Schr\"odinger operators with  only locally integrable potential $V.$ 
We can use the pre-energy space ${\cal L}_V(\Omega) $ to define the following $L^1$-version of criticality:
\begin{definition}[$L^1-$ criticality] \label{def:frozen}
Let $V \in L^1_{loc}(\Omega)$. We say that the operator $-\Delta +V$  is:
\begin{enumerate}
\item[{\rm \bf (i)}]
{\it $L^1$-subcritical in $\Omega$}, if   $Q_V \succeq 0$ in $\Omega$ and  ${\cal L}_V(\Omega) \hookrightarrow L^1_{loc} (\Omega)$;
\item[{\rm \bf  (ii)}]
{\it $L^1$-critical in $\Omega$,} if $Q_V \succeq 0$  in $\Omega$ but $-\Delta+V$ is not $L^1$-subcritical in $\Omega$;
\end{enumerate}
\end{definition}

\begin{definition}\label{amar}
{\it If ${\cal L}_V(\Omega) \hookrightarrow L^1_{loc}(\Omega)$, we can uniquely extend this inclusion map to a continuous linear map
$$
{\bf  J}: {\cal H}_V(\Omega) \to L^1_{loc}(\Omega).
$$
In particular, for any compact set $K \subset \Omega$, we have
$$
      \int_K |{\bf J}(u)|  \, \leq \, C_K \|u\|_{V,\Omega} 
      \quad \text{ for all } u \in  {\cal H}_V(\Omega).
$$
}
\end{definition}

The map  ${\bf J}$ may not be injective in general. For a large class of potentials called balanced potentials (see definition  \ref{tri}) we can show the injectivity; see corollary \ref{bru}.

We now define two notions of criticality based on weighted $L^2-$spaces, one notion global in $\Omega$ and the other local.
\begin{definition}[Global $L^2$- criticality] \label{def:Critical}
Let $V \in L^1_{loc}(\Omega)$. The operator $-\Delta +V$  is said:
\begin{enumerate}
\item[{\rm \bf (i)}]
{\it  globally $L^2$-subcritical in $\Omega$}, if   $Q_V \succeq 0$ in $\Omega$ and  
${\cal L}_V(\Omega) \hookrightarrow L^2 (\Omega; w \, dx)$  for some nonnegative weight $w \in L^1_{loc}(\Omega)$  satisfying $1/w \in L^1_{loc}(\Omega)$.  
\item[{\rm \bf  (ii)}]
{\it globally $L^2$-critical in $\Omega$} if $Q_V \succeq 0$ in $\Omega$, but it is not  globally $L^2$-subcritical in $\Omega$.
\end{enumerate}
\end{definition}

\medskip

\begin{definition} [Feeble $L^2$- criticality] \label{def:weak Critical}
Let $V \in L^1_{loc}(\Omega)$. The operator $-\Delta +V$  is said:
\begin{enumerate}
\item[{\rm \bf (i)}]
{\it $feebly \; L^2$-subcritical in $\Omega$}, if  $Q_V \succeq 0$ in $\Omega$ and  ${\cal L}_V(\Omega) \hookrightarrow L^2 (K)$ for some compact set $K \subset \Omega$ with positive measure;
\item[{\rm \bf  (ii)}]
{\it feebly\; $L^2$-critical in $\Omega$} if $Q_V \succeq 0$ in $\Omega$, but it is not feebly $L^2$-subcritical in $\Omega$.
\end{enumerate}
\end{definition}

The following implications are easy to show:
\begin{center}
 global $L^2$-subcriticality $\Longrightarrow$ feeble $L^2$-subcriticality \\
 \vspace{5pt} 
\hspace{-30pt}  global $L^2$-subcriticality $\Longrightarrow$  $L^1$-subcriticality.
\end{center}

\medskip 

We refer to proposition \ref{prop:Sonata} (see definition \ref{chalka})  for a large class of potentials $V$ for which  all the three definitions of  criticality are equivalent.

Note that our definitions differ from the ones  given in the literature for more ``regular" potentials $V$ (see for instance, ~\cite{Murata:1986} and \cite{PinTint}) in two respects.
Firstly, in our setting the operator $- \Delta + V$ does not necessarily admit a Green's function. Secondly, in the global $L^2$-subcriticality definition, we allow weights $w$ that may vanish in a null set and in the feeble  $L^2$-criticality definition, the quadratic form is assumed  to be ``coercive" only on a compact set with positive measure. 
 
\medskip

The following proposition  justifies the use of  weights like $w$ and $\chi_K$  in our definitions of  $L^2$ -subcriticality. 
\begin{proposition}\label{weird}
Let $N \geq 3$ and $V \in L^1_{loc}(\Omega)$. Assume that $Q_V \succeq 0$ in $\Omega$ and 
$$(W^{1,\infty}_c(\Omega), \| \cdot\|_{W_0^{1,p}(\Omega)} ) \hookrightarrow {\cal L}_V(\Omega) \text{ for some } 1 \leq p <  \frac{2N}{N+2}.$$ Then any nonnegative weight $w \in L^1_{loc}(\Omega)$ such that 
$$
{\cal L}_V(\Omega) \hookrightarrow L^2 (\Omega; w \;dx)
$$
must satisfy $\inf_B w =0$ on any ball $B \Subset \Omega$. In particular, if $w \in C(\Omega)$ 
then $w \equiv 0$.
\end{proposition}
\begin{proof}
 Let $B \Subset \Omega$ be a ball. With our assumption on $p$, we note that $ W^{1,p}_0(B) \not \subset L^2(B)$. Choose a sequence $\{\xi_n\} \subset C_c^{\infty}(B)$ such that 
$$
  \sup_n \;\|\xi_n\|_{ W^{1,p}_0(B)} < \infty 
  \quad \hbox{ and } \quad 
  \| \xi_n\|_{L^2(B)} \to \infty  \hbox{ as } n \to \infty.
$$
By the assumed imbedding, we note that $\{\xi_n\}$ is a bounded sequence in ${\cal L}_V(\Omega)$. Therefore,
$$
\int_\Omega w |\xi_n|^2 \leq Q_V(\xi_n) \leq C < \infty.
$$
This shows that
$\displaystyle \inf_B w =0$.
\end{proof}

\section{Various embeddings of the  space ${\cal L}_V(\Omega)$} \label{s5}
Recall from definition \ref{space} that ${\cal L}_V(\Omega):= (W^{1,\infty}_c(\Omega), \sqrt{Q_V}).$
\begin{proposition} \label{prop:BaseEnergy}
Let $V_i \in L^1_{loc}(\Omega)$ be such that 
$Q_{V_i} \succeq 0$ in $\Omega, i=1,2$ and $V_1 \geq V_2$. 
\begin{enumerate}
\item[{\rm \bf (i)}]
Then, the identity map $id: {\cal L}_{V_1}(\Omega) \to {\cal L}_{V_2}(\Omega)$ is continuous and 
it admits a unique continuous linear extension $id^* : {\cal H}_{V_1}(\Omega) \to  {\cal H}_{V_2}(\Omega)$;
\item[{\rm \bf (ii)}]
If furthermore $-\Delta+V_i$ is an $L^1$-subcritical operator in $\Omega$ ($i=1,2$) with the corresponding extensions ${\bf J}_i$, 
then $J_1 = J_2 \circ {id^*}$ and  $id^*$ is injective whenever  $J_1$ is injective. 
\end{enumerate}
\end{proposition}

\begin{proof}
{\bf (i)} 
Follows from the fact : $0 \leq Q_{V_2}(\xi) \leq Q_{V_1}(\xi)$ for all $\xi \in W^{1,\infty}_c(\Omega)$.

\medskip

{\bf (ii)}
Let $u \in {\cal H}_{V_1}(\Omega)$  and $\{\xi_n\} \subset C_c^{\infty}(\Omega)$ be such that $\xi_n \to u$ in ${\cal H}_{V_1}(\Omega)$. Then,  $\xi_n \to {\bf J}_1(u)$ in $L^1_{loc}(\Omega)$. Note that $\{\xi_n\}$ is  a Cauchy sequence  in ${\cal H}_{V_2}(\Omega)$. Let $\tilde{u}$ be the  limit of $\{\xi_n\}$ in ${\cal H}_{V_2}(\Omega)$. This means that $id^*(u) = \tilde{u}$ and $\xi_n \to {\bf J}_2(\tilde{u})$   in $L^1_{loc}(\Omega)$. Therefore, ${\bf J}_2(id^*(u))= {\bf J}_1(u)$. Injectivity of $id^*$ follows easily.
\end{proof}

\begin{lemma} \label{lem:Damm}
Assume 
$Q_V \succeq 0$ in $\Omega$.
Then, the following properties hold:

\begin{enumerate}
\item[{\bf (i)}]
 ${\cal L}_{V^+}(\Omega) \hookrightarrow L^1_{loc}(\Omega) $,
${\cal L}_{V^+}(\Omega) \hookrightarrow L^2 (\Omega, V^{-} dx)$
and  ${\cal L}_{V^+}(\Omega)  \hookrightarrow L^2 (\Omega, V^{+} dx)$.
\item[{\bf(ii)}] Fix $\varphi  \in C^2_c (\Omega)$. Then, there exists a constant $C:=  \sqrt{Q_V(\varphi)} +  \| \Delta \varphi \|_{L^{\infty}(supp(\varphi))} $ such that
\begin{equation}\label{eq:sloka}
\left|  \int_{\Omega} V \varphi \xi \right| \leq   C  \bigg( \sqrt{Q_V(\xi )} + \int_{supp(\varphi)} |\xi| \bigg),
     \quad
     \forall \xi \in W^{1,\infty}_c (\Omega) .
\end{equation}

\item[{\rm \bf (iii)}] Suppose ${\cal L}_V(\Omega) \hookrightarrow L^1_{loc} (\Omega)$. Fix $\varphi  \in C^2_c (\Omega), \varphi \geq 0$. Consider the linear maps
\begin{equation} \label{eq:Fennel}
   \Phi^{\pm}:
   {\cal L}_{V}(\Omega) \to L^1 (\Omega),
   \quad \hbox{ given by } \quad 
  \Phi^{\pm}( \xi ) =  V^{\pm}  \varphi \xi .
\end{equation}
Then $\Phi^+$ is continuous iff $\Phi^-$ is continuous.
\item[{\rm \bf  (iv)}] Suppose ${\cal L}_V(\Omega) \hookrightarrow L^1_{loc} (\Omega)$. Then,
$$  {\cal L}_V(\Omega) \hookrightarrow L^1_{loc} (\Omega, V^+ dx)
   \quad \hbox{ iff } \quad
   {\cal L}_V(\Omega) \hookrightarrow L^1_{loc} (\Omega, V^- dx).
$$
\end{enumerate}
\end{lemma}
\begin{proof}
{\bf (i)}
Note that $Q_{V^+} \succeq 0$ in $\Omega$. Then, the embedding ${\cal L}_{V^+}(\Omega)  \hookrightarrow {\cal L}_{V}(\Omega) $ is a special case of proposition~\ref{prop:BaseEnergy}.
The other two imbeddings follow from the inequalities 
$$
 \int_{\Omega} V^{-} \xi^2 \leq  Q_{V^{+}} (\xi),
 \quad
 \int_{\Omega}V^{+} \xi^2 \leq \int_{\Omega} \{ |\nabla \xi|^2 + V^{+} \xi^2 \} \text{ for all } \xi  \in W^{1,\infty}_c (\Omega).
 $$

\medskip
{\bf(ii)} 
Noting that $a_{V,\Omega}(\varphi, \xi) \leq \sqrt{Q_V( \varphi )} \sqrt {Q_V(\xi )} $  we deduce that 
\begin{eqnarray*}
     \int_{\Omega} V \varphi \xi &=& a_{V,\Omega}(\varphi,\xi)-\int_{\Omega} \nabla \varphi \cdot \nabla \xi  \nonumber \\
     &\leq& 
    \sqrt{Q_V( \varphi )} \sqrt{Q_V(\xi )} + \int_{\Omega} ( \Delta \varphi) \xi  
     \nonumber \\
     &\leq&
     C(V,\varphi) \bigg(\sqrt{Q_V(\xi )} + \int_{supp(\varphi)} |\xi|\bigg)
     \quad
     \forall \xi \in W^{1,\infty}_c (\Omega) ,
\end{eqnarray*}
where  $C(V, \varphi ) := \sqrt{Q_V(\varphi)} +  \| \Delta \varphi \|_{L^{\infty}(supp(\varphi))}.$ 
Replacing $\xi$ by $-\xi$ in the above calculation, we obtain  \eqref{eq:sloka}.

\medskip

{\bf (iii)} 
Using ${\cal L}_V \hookrightarrow L^1_{loc} (\Omega)$, we rewrite \eqref{eq:sloka} as the following (replacing $\xi$ by $|\xi|$):
$$
    \Big| 
    \int_{\Omega} V^+ \varphi |\xi| -   \int_{\Omega} V^- \varphi |\xi|
    \Big|   
    \, \leq \, 
    C_{\varphi} \sqrt{Q_V(\xi )} ,
     \quad
     \forall \xi \in W^{1,\infty}_c (\Omega) .
$$
We conclude that
$$
    \int_{\Omega} V ^{\pm}\varphi |\xi | 
     \, \leq \, 
    C_{\varphi}  \sqrt{Q_V(\xi )} +  \int_{\Omega} V ^{\mp}\varphi |\xi| .
$$
It follows now that $\Phi^+$ is continuous iff $\Phi^-$ is continuous.

\medskip

{\bf (iv)} follows from (iii) by noting that $\Phi^{\pm}$ is continuous for any $\varphi  \in C^2_c (\Omega), \varphi \geq 0,$ iff $ {\cal L}_V(\Omega) \hookrightarrow L^1_{loc} (\Omega, V^{\pm} dx)
$.
 \end{proof}

\begin{lemma} \label{lem:BADAMILK}
Assume 
$-\Delta+V$ is an $L^1-$ subcritical operator in $\Omega$.
Then
\begin{enumerate}
\item[{\rm \bf  (i)}]
$Q_V(\xi^{\pm}) \leq Q_V(\xi)$ for each $\xi \in W^{1,\infty}_c (\Omega)$;
\item[{\rm\bf   (ii)}]
Given  $u \in {\cal H}_V(\Omega)$, we can find $u_{\pm} \in {\cal H}_V(\Omega)$ such that
$$
u=u_+-u_- \; \text{ and } \; {\bf J}(u_{\pm}) = ({\bf J}(u))^\pm \geq 0.
$$

\item[{\rm \bf  (iii)}]
Given  $u \in {\cal H}_V(\Omega)$, choose $u_{\pm} \in {\cal H}_V(\Omega)$ as in (ii) above. Then,
$$\|\;u_{\pm}\;\|_{V,\Omega} \leq \|u\|_{V,\Omega}, \quad \|\;u_{+}+u_-\;\|_{V,\Omega} \leq \|u\|_{V,\Omega}.$$
\item[{\rm \bf  (iv)}]
Given  $u \in {\cal H}_V(\Omega)$, choose $u_{\pm} \in {\cal H}_V(\Omega)$ as in (ii) above. Then, 
$$a_{V,\Omega}(u_+, u_{-}) \leq 0.$$
\end{enumerate}
\end{lemma}
\begin{proof}
{\bf (i)}
For each $\xi \in W^{1,\infty}_c (\Omega)$ we have  
$$
   \xi^{+}, \xi ^{-} \in W^{1,\infty}_c (\Omega)
   \quad \hbox{ and } \quad 
   Q_V (\xi) = Q_V (\xi^{+} ) + Q_V (\xi^{-}) \geq Q_V (\xi^{\pm}).
$$ 

{\bf  (ii)} Let $\{\xi_n\} \subset W^{1,\infty}_c (\Omega)$ be such that $\xi_n \to u$ in ${\cal H}_V(\Omega)$. From  part (i) we get 
that $ Q_V (\xi_n^{\pm})$ is also bounded. Therefore, 
there exists $u_{\pm} \in {\cal H}_V(\Omega)$ such that (up to a subsequence)
\begin{equation} \label{eq:WantGoToMALLBis}
    \xi_n^{\pm} \longrightarrow  u_{\pm} 
    \quad 
    \hbox{ weakly in } {\cal H}_V(\Omega).
\end{equation} 
Clearly, $u=u_+-u_-$.
Recalling from definition \ref{amar} that ${\bf J}: {\cal H}_V(\Omega) \to L^1_{loc} (\Omega)$ is a continuous linear map, we have 
\begin{enumerate}
\item[(a)]
$\xi_n  \xrightarrow{\, \, L^1_{loc}(\Omega) \, } {\bf J}(u) $, which implies (up to subsequence) $\xi_n \xrightarrow{\, \, a.e \, \,} {\bf J}(u)$, 
and hence  $\xi_n^{\pm} \xrightarrow{\, \, a.e \, \,} {\bf J}(u)^{\pm}$;
\item[(b)]
$\xi_n^{\pm} \longrightarrow {\bf J}( u_{\pm})$ in the weak topology on $ L^1_{loc}(\Omega)$ (recall that   $L^{\infty}_{c} (\Omega)$ is the dual of $ L^1_{loc}(\Omega)$; refer  Thm. 4, p.182 of Horvath \cite{Horvath}).
\end{enumerate}
Let $B \Subset \Omega$ be a ball. By property (b) above, we have
\begin{equation} \label{eq:WantGoToMALL}
   \xi_n^{\pm} \longrightarrow   {\bf J}(u_{\pm})
   \, \, \hbox{ weakly in }  L^1 (B) .
\end{equation}
This is equivalent to $\{\xi_n^{\pm}\}$ being an equi-integrable family in $L^1 (B)$ (see lemma \ref{lem:L1Compact} in the appendix). By Vitali's convergence Theorem 
($\xi_n^{\pm} \xrightarrow{\, \, a.e \, \,} {\bf J}(u)^{\pm}$ and is equi-integrable) we conclude that 
$\xi_n^{\pm}  \xrightarrow{\, \, L^1 (B) \, } {\bf J}(u)^{\pm}$. Hence in \eqref{eq:WantGoToMALL} we have strong convergence, and furthermore ${\bf J}(u_{\pm}) = {\bf J}(u)^{\pm} \geq 0$. 

\medskip

{\bf (iii)}
Let $u \in {\cal H}_V(\Omega)$, and as in part (ii), consider a sequence $\{\xi_n\}  \subset W_c^{1,\infty} (\Omega)$ such that 
$\xi_n  \to u$ in ${\cal H}_V(\Omega)$ as well as 
$$
   \xi_n^+ \to u_{+}  \;\text{ and }\;
   \quad
   \xi_n^- \to u_{-} \hbox{ weakly in } {\cal H}_V(\Omega).
$$
Hence, $|\xi_n| \to u_+ + u_-$ weakly in ${\cal H}_V(\Omega)$. 
Since a norm is sequentially weakly lower semi-continuous and 
$Q_V (|\xi|) = Q_V (\xi)$ for any $\xi \in W_c^{1,\infty} (\Omega)$ we deduce that
$$
    \|\; u_+ + u_-\;\|_{V,\Omega} \leq \liminf_{n \to \infty} Q_V (|\xi_n|) =  \liminf_{n \to \infty} Q_V (\xi_n) = \|u\|_{V,\Omega} .
$$ 
In a similar manner, using (i), we can show that $\|\;u_{\pm}\;\|_{V,\Omega} \leq \|u\|_{V,\Omega}$.
\medskip

{\bf (iv)}
The parallelogram law and part (iii) imply the following
$$
   4 a_{V,\Omega}(u_{+}, u_{-})  = \|u_{+} + u_{-}\|^2_{V,\Omega} - \|u_{+} - u_{-}\|^2_{V,\Omega} \leq 0.
$$
\end{proof}

\section{Balanced potentials, Energy  and Distribution Solutions} \label{s6}

\subsection{Balanced  potentials: definition and  properties}  \label{s6.1}

Let $V \in L^1_{loc}(\Omega)$ and $Q_V \succeq 0$ in $\Omega$.  
For  $U \Subset \Omega$  a non-empty open set, we consider the 
space 
$$
{\mathcal L}_V(\Omega; U):= \Big(W^{1,\infty}_c(\Omega), \| \cdot\|^{U}_{V,\Omega} \Big)
$$
with the norm
$$
 \|  \xi \|^{U}_{V,\Omega} := \sqrt{Q_V(\xi )} + \int_{U} |\xi|, \quad \xi \in W^{1,\infty}_c(\Omega),
$$
and let ${\mathcal B}_V(\Omega; U)$ denote the completion of this space.  

\begin{remark}
\begin{enumerate}
\item[(i)]
In general, the norms  $\|  \cdot \|^{U}_{V,\Omega}$ are not equivalent as $U$ varies over $\Omega$. For a general nonnegative operator equivalence of these norms seems to require a further restriction on $V^+$ 
(see corollary \ref{knock} and also   \cite[Theorem 1.4]{PinTint}).
\item[(ii)]
When $-\Delta +V$ is an $L^1$-subcritical operator in $\Omega$ or if ${\mathcal L}_V(\Omega; U) \hookrightarrow L^1_{loc}(\Omega)$ for any $U$, it is easy to see that the norms  $\|  \cdot \|^{U}_{V,\Omega}$ are equivalent as $U$ varies over $\Omega$. In the former case  all the spaces ${\mathcal B}_V(\Omega; U)$ are equivalent to ${\mathcal H}_V(\Omega)$.
\end{enumerate}
\end{remark}

The inclusion map ${\mathcal L}_V(\Omega; U) \to L^1(U)$ being continuous, 
it admits a unique continuous extension to ${\mathcal B}_V(\Omega; U)$ that will be denoted as ${\bf J}_{U}$. If $-\Delta +V$ is an $L^1$-subcritical operator in $\Omega$,  we can identify ${\mathcal B}_V(\Omega; U)$ with ${\mathcal H}_V(\Omega)$,  
and we will let ${\bf J}_{U}(u):=  {\bf J}(u)_{\mid U}$.

\medskip

Given a closed ball $B \subset \Omega$, let $X_B$ denote the completion of the space  $C_c^2(B)$ under the $C^2(B)$ norm  and $X_B^*$ denote its dual space. Based on \eqref{eq:sloka}, we have the following result.  
\begin{proposition} \label{pol}
Let $V \in L^1_{loc}(\Omega)$ be such that $Q_V \succeq 0$ in $\Omega$.   Then  for any closed ball $B \subset \Omega$, and any open set $U \Subset \Omega$ containing $B$, the linear map
\begin{equation}\label{rage}
{\mathcal L}_V(\Omega; U) \ni \xi \mapsto  V\xi \in    X_B^* 
\end{equation}
is continuous. In particular, if $-\Delta +V$ is an $L^1$-subcritical operator in $\Omega$, then the map in \eqref{rage} is continuous from ${\mathcal L}_V(\Omega)$ to $X_B^* $ for any such ball $B$.
\end{proposition}
\proof  From \eqref{eq:sloka}, we easily obtain that 
$$
\|V\xi\|_{X_B^*} := \sup_{\| \varphi \|_{X_B} \leq 1} \Big|\int_\Omega V \xi \varphi \Big | \leq C \| \xi\|^U_{V,\Omega}, \quad \forall \xi \in W^{1,\infty}_c(\Omega).
$$
\qed

\medskip

Denote by 
${\mathbb V}_{B,U}$ the unique continuous linear extension of the map given in \eqref{rage} to 
${\mathcal B}_V(\Omega; U)$. Consider the following ``multiplication by $V$" map
$$
{\mathbb M}_B : 
\big\{ w \in  L^1(B): Vw \in L^1(B) \big\} \to L^1(B) \quad \text{ given by } \;\; {\mathbb M}_B (w):=  Vw.
$$
We consider situations where we can recover ${\mathbb V}_{B,U}$ as the composition of ${\bf J}_{U}$ and the map ${\mathbb M}_B$.

\begin{definition}\label{tri}
A potential $V \in L^1_{loc}(\Omega)$  is called {\it balanced }
in $\Omega$ if

\hspace{1mm} {\bf (i)} $Q_V \succeq 0$ in $\Omega$,

\medskip
and for any closed ball  $B \subset \Omega$ there exists an  open set $U \Subset \Omega$ containing $B$ such that
\begin{enumerate}
\item[{\bf (ii)}] 
$V {\bf J}_{U}(u) \in L^1(B)$ for any $u \in {\mathcal B}_V(\Omega; U)$,
\item[{\bf (iii)}] 
${\mathbb V}_{B,U} = {\mathbb M}_B \circ {\bf J}_{U}$ on  ${\mathcal B}_V(\Omega; U)$. 
\end{enumerate}
\end{definition}

That is, for a balanced potential,  the following diagram  commutes:
\begin{equation}
    \begin{tikzcd}
        L^1 (B) \arrow{r}{{\mathbb M}_B}        & L^1 (B)   \arrow{r}{i} & X_B^{*} \\
        \     &                {\mathcal B}_V(\Omega; U)      \arrow{lu}{{\bf J}_U}    \arrow{ru}[swap]{{\mathbb V}_{B,U}}      &   \     \\ 
                 \     &                {\cal L}_V (\Omega; U)    \arrow{u}[swap]{i}                  &  \
    \end{tikzcd}
\end{equation}
In the subsections~\ref{s6.5} and \ref{s6.6} we will give examples of both balanced and non-balanced potentials.

\begin{proposition}\label{nip}
Let (i) and (ii) of definition \ref{tri} be satisfied. Then the following are equivalent:
\begin{enumerate}
\item[{\rm {\bf (i)}}]
$V$ is balanced in $\Omega$,
\item[{\rm {\bf (ii)}}]
the map $\displaystyle
{\mathbb M}_B \circ {\bf J}_{U} : {\mathcal B}_V(\Omega; U) \to L^1(B) \subset X_B^*$
is continuous (in the $X_B^*$ norm),
\item[{\rm {\bf (iii)}}]
for any sequence $\{\xi_n\} \subset W_c^{1,\infty}(\Omega)$  such that
$\xi_n \to u$ in ${\mathcal B}_V(\Omega; U)$, we have 
$$\int_B V \xi_n \varphi \to \int_B V {\bf J}_{U}(u)  \varphi , \;\; \forall \varphi \in C_c^2(B).$$
\end{enumerate}
\end{proposition}
\proof 
{\bf (i)} $\Longrightarrow$ {\bf (ii)}: follows from the continuity of ${\mathbb V}_{B,U}$.

{\bf (ii)} $\Longrightarrow$ {\bf (iii)}: is trivial.

{\bf (iii)} $\Longrightarrow$ {\bf (i)}:
Note that $\{V \xi_n\}$ is a  sequence  converging to ${\mathbb V}_{B,U}(u)$ in $X_B^*$ by proposition \ref{pol}. Since $C_c^2(B)$ is dense in $X_B$, assumption (iii) implies that  ${\mathbb V}_{B,U}(u)= V{\bf J}_{U}(u)= ({\mathbb M}_B \circ {\bf J}_{U})(u)$. Thus, condition (iii) of definition \ref{tri} holds.
\qed

\medskip

\begin{remark} \label{nig} 
\begin{enumerate}
\item[{\rm \bf (i)}] 
If $V$ is balanced in $\Omega$, then Range(${\mathbb V}_{B,U}) \subset L^1(B)$.
\item[{\rm \bf (ii)}] 
One can easily check that the graph of the map ${\mathbb M}_B$ is always closed in $L^1(B) \times L^1(B)$, but not necessarily so in $L^1(B) \times X_B^*$. But, 
if  definition \ref{tri} (i)-(ii) are satisfied and the graph of the map
${\mathbb M}_B$ is closed in $L^1(B) \times X_B^*$, then $V$ is balanced in $\Omega$. This follows from proposition \ref{pol}.
\item[{\rm \bf (iii)}] 
When ${\cal L}_V(\Omega) \hookrightarrow L^1_{loc}(\Omega)$
(i.e. $-\Delta+V$  is  an $L^1$-subcritical operator in $\Omega$), 
we can identify ${\cal B}_V (\Omega;U)$ with  ${\cal H}_V (\Omega)$.  In this case, for each $\varphi \in C_c^2 (\Omega)$ consider the linear form 
\begin{equation} \label{eq:Cuoco}
   {\cal L}_V (\Omega) \ni \xi \mapsto \int_{\Omega}  V \xi \varphi
\end{equation}
which is continuous by \eqref{eq:sloka}, and let ${\bf J}$ be as given in definition \ref{amar}. 
Hence, in the subcritical case, the potential $V$ is balanced iff $V {\bf J} (u) \in L^1_{loc}(\Omega)$  for all $u \in {\cal H}_V (\Omega)$ and the unique continuous extension 
of~\eqref{eq:Cuoco} to 
${\cal H}_V (\Omega)$ is given by 
$$
   {\cal H}_V (\Omega) \ni    u \mapsto \int_{\Omega}  V {\bf J} (u) \varphi .
$$ 
\end{enumerate}
\end{remark}

\begin{proposition}\label{mea}
\begin{enumerate}
\item[{\rm \bf (i)}] 
The class of balanced potentials  form a  convex set in $L^1_{loc}(\Omega)$.
\item[{\rm \bf (ii)}] 
Let $V \in L^1_{loc}(\Omega)$ be   balanced  in $\Omega$ 
and $W \in L^1_{loc}(\Omega)$ be nonnegative. Then $V+W$ is  balanced  in $\Omega$.
\end{enumerate}
\end{proposition}
\proof 
\medskip
{\bf (i)}  
Let $V_1,V_2$ be  balanced  in $\Omega$ and let $V := tV_1+(1-t)V_2, \; t \in (0,1).$ It is easy to see that $Q_{V} \succeq 0$ in $\Omega$. Fix a ball $B \Subset \Omega$ and choose $U_1, U_2$ corresponding to $V_1,V_2$ from definition \ref{tri}. Let   $U:=U_1 \cup U_2$.  To avoid confusion we denote the maps ${\bf J}_{U_i}$ associated to the potentials $V_i$ and the open set $U_i$ by ${\bf J}^i_{U_i}$ for $i=1,2.$ \\
Using the concavity of the square root function, we get for $t \in (0,1)$, $\xi \in W_c^{1,\infty}(\Omega)$,
\begin{eqnarray}
 t \Big(\sqrt{Q_{V_1}(\xi)} +\int_{U_1} |\xi| \Big)+ (1-t) \Big(\sqrt{Q_{V_2}(\xi)} +  \int_{U_2} |\xi| \Big)  
&\leq&   \sqrt{Q_{V}(\xi)} +  \int_U |\xi|  . \notag  
\end{eqnarray}
From the above inequality, it follows that if $\xi_n \to u$ in ${\mathcal B}_{V}(\Omega; U)$, then $\xi_n \to u_i$ (for some $u_i$) in ${\mathcal B}_{V_i}(\Omega; U_i)$ for $i=1,2$.
We note that for such a sequence $\{\xi_n\}$, we obtain also that $\xi_n \to {\bf J}_{U}(u)$ in $L^1(U)$ as well as $\xi_n \to {\bf J}^i_{U_i}(u_i)$ in $L^1(U_i)$ for $i=1,2$. Therefore, 
$${\bf J}_{U}(u) = {\bf J}^1_{U_1}(u_1)={\bf J}^2_{U_2}(u_2) \;\; \text{ on }\;\; B \;\;\text{ for any }\;\; u \in {\mathcal B}_{V}(\Omega; U). $$
It follows that (ii) of definition \ref{tri} holds for this $U$. Similarly, it also follows that (iii) of definition \ref{tri} holds since 
$$
V \xi_n \to tV_1 {\bf J}^1_{U_1}(u_1) + (1-t) V_2 {\bf J}^2_{U_2}(u_2) = V  {\bf J}_{U}(u) \;\;\text{ in } \;\; X_B^*.
$$
\medskip
{\bf (ii)}  Clearly  $Q_{V+W} \succeq 0$ in $\Omega$.  Fix a ball $B \Subset \Omega$ and choose $U$ corresponding to $B$ from definition \ref{tri} applied to $V$.  We note that
$$
 \frac{1}{2} \Big(\sqrt{Q_{V}(\xi)} + \int_{U} |\xi| \Big)+ \frac{1}{2} \Big( \int_\Omega W |\xi|^2 \Big)^{\frac{1}{2} }  
\leq   \sqrt{Q_{V+W}(\xi)} +  \int_U |\xi|  . \notag  
$$
We denote the maps  associated to the potentials $V, V+W$ and the open set $U$ by ${\bf J}_{U}$  and ${\bf \tilde{J}}_{U}$ respectively.
As above,  if $\xi_n \to u$ in ${\mathcal B}_{V+W}(\Omega; U)$, then $\xi_n \to v$ for some $v \in {\mathcal B}_{V}(\Omega; U)$, $ \xi_n \to {\bf \tilde{J}}_{U}(u)$,  $\xi_n \to {\bf J}_{U}(v)$ in $L^1(U)$  and $\xi_n \to w$ in $L^1(U; W dx)$ for some $w$.  
 Arguing as before, we can verify (ii) and (iii) of definition \ref{tri} hold for $V+W$.  \qed

\medskip

\subsection{ Distribution Vs Energy solution } \label{s6.2}

Distributional solutions to $-\Delta u + Vu =h$  (see definition~\ref{dusera}) do not necessarily belong to the energy space ${\cal H}_V(\Omega)$ associated to the 
Schr\"odinger operator $-\Delta + V$. For instance, non-zero constants  are distributional solutions to $\Delta u=0$ in $\Omega$;
however, non-zero constants do not lie in $ {\cal H}_0 (\Omega)= {\cal D}^{1,2}_0(\Omega)$ for dimensions larger than two. 

When ${\cal L}_V(\Omega) \hookrightarrow L^1_{loc}(\Omega)$, it is  natural to ask if  for an element in ${\cal H}_V(\Omega)$, the notions of distributional solutions and energy solutions (see definition \ref{def:EnergySolution}) are equivalent. This equivalence is not obvious since  the 
energy solution $u $ may not have the property that $V u \in L^1_{loc} (\Omega)$.
But this equivalence is true under the restriction that $V $  is balanced as shown by the following proposition.

\begin{proposition} \label{prop:EnergyVsDist}
Assume $V \in L^1_{loc}(\Omega)$ is balanced in $\Omega$ and  
$-\Delta+V$ is an $L^1-$ subcritical operator in $\Omega$.
Given $f \in L^{1}_{loc}(\Omega)$, an element $u \in {\cal H}_V(\Omega)$ is an energy solution (see definition \ref{def:EnergySolution}) to $-\Delta u + Vu =f$  iff {\bf J}(u) is a distributional solution.
\end{proposition}
\begin{proof} 
Assume first $u \in {\cal H}_V(\Omega)$ is an energy solution and consider a sequence $\{ \xi_n\} \subset W_c^{1,\infty} (\Omega)$ such that  
$\xi_n  \to  u$ in  ${\cal H}_V(\Omega)$.  Since
$$
    a_{V,\Omega} (\xi_n, \varphi) = \int_{\Omega} f \varphi + a_{V,\Omega} (\xi_n - u, \varphi)\;\; \;\;\forall \varphi \in W^{1,\infty}_c(\Omega),
$$
 we get

$$
    \int_{\Omega} \xi_n ( - \Delta \varphi)  + \int_{\Omega}  V \xi_n \varphi = \int_{\Omega} f \varphi + o_n(1),
   \quad
   \forall \varphi \in C_c^2 (\Omega)  .
$$
Remark \ref{nig} (iii)  implies,
$$
   \int_{\Omega} {\bf J}(u) ( - \Delta \varphi)  + \int_{\Omega}  V {\bf J}(u) \varphi = \int_{\Omega} f \varphi,
   \quad
   \forall \varphi \in C_c^{2} (\Omega)  .
$$

Conversely, let $u \in {\cal H}_V(\Omega)$ be such that ${\bf J}(u)$ is a distributional solution; in particular, $V{\bf J}(u) \in L^1_{loc}(\Omega)$. Tracing back the steps given above, we conclude that $u$ is an energy solution as well.
\end{proof}

\medskip 

\begin{corollary}\label{bru}
Let $V  \in L^1_{loc} (\Omega)$  be  balanced in $\Omega$ and 
$-\Delta+V$  an $L^1-$ subcritical operator in $\Omega$.
Then ${\bf J}$ is injective.
\end{corollary}
\proof Let  $u_* \in {\cal H}_V(\Omega)$ be such that ${\bf J}(u_*)=0$. We note that, since 0 is a distributional solution to $-\Delta u +Vu=0$, by the above proposition $u_*$ is an energy solution to this same equation. Since an energy solution is unique in ${\cal H}_V(\Omega)$, we get $u_*=0$. \qed

\medskip

\begin{lemma}[Comparison principle for energy solutions] 
\label{lem:BADAMILK2}
Let $V  \in L^1_{loc} (\Omega)$ be such that 
$-\Delta+V$ is an $L^1-$ subcritical operator in $\Omega$.
Let
$f \in  L^1_{loc} (\Omega) $ and $u \in {\cal H}_{V} (\Omega)$
be an energy super solution :
$$
     a_{V,\Omega} (u,\xi)  \geq  \int_{\Omega} f \xi,
     \quad
     \forall \; 0 \leq \xi \in W_c^{1,\infty} (\Omega).
$$
If $ f \geq 0 $,
then ${\bf J}(u) \geq 0$.
\end{lemma}
\begin{proof}

Using the equation satisfied by $u$ we get 
$$
   a_{V,\Omega}  (u, \xi) \geq 0  ,
   \quad
   \, \forall \; 0 \; \leq \xi \in W^{1, \infty}_c (\Omega).
$$ 

By  Lemma~\ref{lem:BADAMILK} (ii), we can find a sequence $\{\xi_n \} \subset W^{1,\infty}_c (\Omega) $ such that
\begin{equation} \label{eq:Headache4}
 \xi_n \to u \text{ in }  {\cal H}_{V}(\Omega), \quad  \xi_n ^{\pm}\to u_{\pm} \hbox{ weakly in } {\cal H}_{V}(\Omega),
  \quad
  \xi_n ^{\pm}\to {\bf J}( u)^{\pm} \hbox{ in }  L^1_{loc}(\Omega).
\end{equation}
Recalling $u=u_+- u_-$,  
we can now argue as follows:
\begin{align*}
   0       &\leq 
      \lim_{n \to \infty} a_{V,\Omega} \big( u, \xi_n^- \big)   
      &\\
      &=
     a_{V,\Omega} \big( u, u_- \big)  
     &\text{(by \eqref{eq:Headache4})}  \\
     &=
     a_{V,\Omega} \big( u_+,  u_-\big) - 
     a_{V,\Omega} \big( u_-,u_- \big) 
     & \\
     &\leq 
     - a_{V,\Omega} \big( u_-,u_-  \big).
     &\text{(Lemma~\ref{lem:BADAMILK} (iv))} 
\end{align*}

Hence $u_-= 0$ and thus,
${\bf J}(u)={\bf J}(u_+) = {\bf J}(u)^+ \geq 0$.
\end{proof}

\begin{proposition}[Strong Maximum principle for energy solution] \label{prop:booth}
Assume $V \in L^1_{loc}(\Omega)$  is balanced in $\Omega$ and  
$-\Delta+V$ is an $L^1-$ subcritical operator in $\Omega$.
Given  a non-negative function $ f \in L^{1}_{loc} (\Omega)$, $f \not \equiv 0$, let $u \in {\cal H}_V(\Omega)$  be the  energy solution of $-\Delta u + Vu=f$. That is,
$$
     a_{V,\Omega} (u, \xi) = \int_{\Omega} f \xi ,
     \quad \forall \,  \xi \in W_c^{1,\infty} (\Omega) .
   $$
Then, ${\bf J}(u)$  (a  distributional solution of $-\Delta u + Vu=f$)   has a   quasi-continuous representative $\tilde u$ which satisfies
\begin{equation} \label{eq:ZeroSet}
   {\rm Cap}\big( \{ \tilde u = 0 \}  \big) =  0.
\end{equation}
\end{proposition}
\begin{proof}
By corollary \ref{bru} and  lemma \ref{lem:BADAMILK2}, we have ${\bf J}(u) \not \equiv 0$ and 
${\bf J}(u)  \geq 0$.  Furthermore, ${\bf J}(u)$ is also a distributional solution of the corresponding equation from proposition \ref{prop:EnergyVsDist}. From proposition  \ref{prop:RanRan}, we conclude that ${\bf J}(u) \in W^{1,1}_{loc}(\Omega)$ and hence admits a  quasicontinuous representative $\tilde{u}$  in this space. Now the strong maximum principle in theorem \ref{thm:StrongMax}
can be applied to $\tilde{u}$ showing that ${\rm Cap} (\{ \tilde u = 0 \})=0$. 
\end{proof}
\medskip

\subsection{Strongly Balanced Potentials} \label{s6.3}
We note that any non-negative operator $-\Delta+V$ leads to an  estimate like~\eqref{eq:sloka} 
but not to similar  estimates involving only the positive or negative parts of $V$.
We are therefore led to introduce the following  more concrete formulation:

\begin{definition} \label{tre}
A function $V \in L^1_{loc}(\Omega)$  is called {\it strongly balanced }
in $\Omega$ if $Q_V \succeq 0$ in $\Omega$ and  
for any compact $K \subset \Omega$ there exists an open set $U \Subset \Omega$ containing  $K$ for which the linear map:
$$
{\mathcal L}_V(\Omega; U) \ni \xi \mapsto  V^+ \xi \in L^1(K)
$$
is continuous. That is, there exists a constant  $ c>0$ (depending on $K,U$) satisfying
\begin{equation} \label{eq:BalancedCondition}
      c \int_{K } V^+   |\xi|  \leq    \|  \xi \|^{U}_{V,\Omega}:=\sqrt{Q_V(\xi)} +  \int_{U }  |\xi|  \;\;\;   \forall  \xi \in W_c^{1,\infty}(\Omega).
\end{equation}
\end{definition}

\begin{remark}\label{hic}
\begin{enumerate}
\item[{\bf (i)}] 
If all the norms  $\|  \cdot \|^{U}_{V,\Omega}$ are equivalent as $U$ varies over $\Omega$, we can easily see that the definition \ref{tre} is the same as requiring  that \eqref{eq:BalancedCondition} holds for any  pair $K,U$.
\item[{\bf (ii)}] 
Indeed, for the class of strongly balanced potentials in ${\mathcal V}(\Omega)$ (see definition \ref{uno} and proof of proposition \ref{paint}) given the compact set $K$  we cannot ensure that \eqref{eq:BalancedCondition} holds for an  arbitrary open set $U$ containing $K$.
\end{enumerate}
\end{remark}

 \begin{proposition} \label{trie}
Let $V \in L^1_{loc}(\Omega)$ be such that $Q_V \succeq 0$ in $\Omega$. 
Then, the following statements are equivalent:
\vspace{-2mm}
\begin{enumerate}
\item[{\rm \bf (i)}] 
$V$  is   strongly balanced in $\Omega$;
\item[{\rm \bf (ii)}] 
for any  compact set $K \subset \Omega$, there exists an open set $U \Subset \Omega$ containing $K$ for which the following 
linear map is continuous
\begin{equation} \label{BIC}
{\mathcal L}_V(\Omega; U) \ni \xi \mapsto  V^- \xi \in L^1(K); 
  \end{equation}
\item[{\rm \bf (iii)}] 
for any  compact set $K \subset \Omega$, there exists an open set $U \Subset \Omega$ containing $K$ for which the following 
linear map is continuous
\begin{equation}\label{ell}
{\mathcal L}_V(\Omega; U) \ni \xi \mapsto  V \xi \in L^1(K) .
\end{equation}
\end{enumerate}
\end{proposition} 
\proof 
{\bf (i)}  $\Longleftrightarrow$ {\bf (ii)}:
Let $V$ be strongly balanced in $\Omega$ and a compact $K$ be given. Corresponding to $K$ choose an open set $U_1$ from definition \ref{tre}. Following  a similar argument  using \eqref{eq:sloka} as in lemma \ref{lem:Damm} (iii) there exists an  open set 
$U \Subset \Omega$ containing $K$ such that~\eqref{BIC} is continuous. Converse is similar. 

\medskip

{\bf (i)}  $\Longleftrightarrow$ {\bf (iii)}: 
Follows from the equivalence (i) and (ii).
\qed

\medskip

\begin{proposition}\label{fra}
Let $V \in L^1_{loc}(\Omega)$ be a strongly balanced potential in $\Omega$. Then, 
\vspace{-2mm}
\begin{enumerate}
\item[{\bf (i)}] 
for  any compact $K \subset \Omega$ we may find an open set $U \Subset \Omega$ containing $K$ such that
$$ 
    V {\bf J}_U ( u) \in L^1(K)  \quad \forall u \in  {\mathcal B}_V(\Omega; U),
$$
\item[{\rm \bf (ii)}] 
for $K,U$ as in (i), it holds
$$
\xi_n \to u \;\;\text{ in }\;\; {\mathcal B}_V(\Omega; U) \Longrightarrow V \xi_n \to V {\bf J}_U ( u) \;\;\text{ in }\;\;L^1(K).
$$
\end{enumerate}
\end{proposition}
\proof 
{\bf (i)} 
 Given compact $K$ choose open set $U$ as in (i) above.  Let $u \in  {\mathcal B}_V(\Omega; U)$. 
 Choose a sequence $\{\xi_n\} \subset {\mathcal L}_V(\Omega; U)$ converging to $u$. Then, it follows from \eqref{BIC}  that $\{V^\pm \xi_n \}$ is Cauchy in $L^1(K)$ and hence converges there. Noting that $\xi_n \to {\bf J}_U(u)$ in $L^1(U)$ and hence pointwise a.e. in $U$ for a subsequence, we get that $V^\pm  {\bf J}_U(u) \in L^1(K)$. 

\medskip

{\bf (ii)} 
From (i) above, we obtain that $\{V\xi_n\}$ is Cauchy in $L^1(K).$ Noting that (up to a subsequence) $\xi_n \to   {\bf J}_U ( u)$ pointwise a.e. in $K$, the result follows. \qed

\medskip

Finally, we can justify our terminology :
\begin{corollary}\label{nes}
$V \in L^1_{loc}(\Omega)$ is a strongly balanced potential in $\Omega$ iff it is  balanced  in $\Omega$ and the map
$$
{\mathbb V}_{B,U} : {\mathcal B}_V(\Omega; U) \to L^1(B)
$$
is continuous in the norm.
\end{corollary}
\proof 
Follows from both statements proposition \ref{fra} (i) and (ii) along with a simple covering argument. 
\qed

\begin{remark}\label{qani}
Let $V \in L^1_{loc}(\Omega)$ such that $Q_V \succeq 0$ in $\Omega$. Then,  the proofs given in proposition above shows that \eqref{eq:BalancedCondition} holds
\begin{enumerate}
\item[{\rm \bf (i)}] for any pair $(K,U)$ iff the linear maps
\begin{equation}\label{ela}
{\mathcal L}_V(\Omega; U) \ni \xi \mapsto  V^\pm \xi \in L^1_{loc}(\Omega)
\end{equation}
are continuous for any $U$.
\item[{\rm \bf (ii)}] 
for any pair $(K,U)$ such that $K \subset U$  iff the linear maps
\begin{equation}\label{eli}
{\mathcal L}_V(\Omega; U) \ni \xi \mapsto  V^\pm \xi \in L^1_{loc}(U).
\end{equation}
are continuous for any $U$.
\end{enumerate}
\end{remark}

\medskip

Next result provides an abstract result for ensuring strong balancedness.
 
\begin{proposition} \label{inz}
Let $V \in L^1_{loc}(\Omega)$ be such that $Q_V \succeq 0$ in $\Omega$. Suppose for any ball $B \Subset \Omega$ there exists an open set $U \Subset \Omega$ containing $B$ and  a reflexive  space
$({\mathcal Z}_B, \|\cdot \|_B)$   consisting of $L^1(B)$ functions closed under the $|\cdot|$  operation such that
\begin{enumerate}
\item[{\rm (i)}]
$\big\| |z| \big\|_B \leq \|z\|_B$  for all $z \in {\mathcal Z_B}$,
\item[{\rm (ii)}]
${\cal L}_{V}(\Omega; U) \hookrightarrow  {\mathcal Z_B} \hookrightarrow L^1(B)$,
\item[{\rm (iii)}]
$V  z \in L^1(B)$ for all $z \in {\mathcal Z_B}$.
\end{enumerate}
Then 
$V$ is  strongly balanced in $\Omega$. 
\end{proposition}

The above proposition uses the following basic result, whose proof is recalled in the appendix.
 
\begin{lemma}\label{brb}
Let $K \subset {\mathbb R}^N$ be a compact set (with positive measure) and 
$({\mathcal Z}, \|\cdot \|)$  a reflexive  space  consisting of $L^1(K)$ functions closed under the $|\cdot|$  operation such that $$
   {\mathcal Z} \hookrightarrow L^1(K) 
   \quad \hbox{ and } \quad
   \big\| |z| \big\| \leq \|z\| \, \, \forall z \in {\mathcal Z}.
$$
Then given any sequence $\{z_n\} \subset {\mathcal Z}$ converging to some $z \in {\mathcal Z}$, we may find a subsequence $\{z_{n_k}\}$ and $z^* \in {\mathcal Z}$ such that
$$
|z_{n_k}| \leq z^* \;\;\; \text { in } K  \;\;\forall k,
  \quad \hbox{ and } \quad 
  z_{n_k} \hbox{ converges pointwise to }  z.
$$
\end{lemma}

{\bf Proof of Proposition~\ref{inz}:}
Take any  ball $B \Subset \Omega$ along with the corresponding $U$ and a sequence $\xi_n \to 0$ in ${\cal L}_{V}(\Omega; U)$. It follows that $\xi_n \to 0$ in ${\mathcal Z}_B$ (and hence in $L^1(B)$). We claim that $V  \xi_n \to 0$ in $L^1(B)$. Consider any subsequence $\{V  \xi_{n_k} \}$. From Lemma~\ref{brb}, we may find a further subsequence $\{\xi_{n_{k_j}} \}$ and a function $z^* \in {\mathcal Z}_B$ such that 
$$
    |\xi_{n_{k_j}} | \leq z^* \;\; \text{ in }  B \;\; \forall j
    \quad \hbox{ and } \quad
    \xi_{n_{k_j}} \hbox{ converges pointwise to } 0.
$$
Since $V  z^* \in L^1(B)$ (by assumption),  we have $V  \xi_{n_{k_j}}  \xrightarrow{ L^1(B) } 0$
(by Lebesgue dominated convergence theorem)  
 and hence the claim follows. 
 Proposition~\ref{trie} (iii) shows $V$ is strongly balanced in $\Omega$.  
 \qed

\medskip

 \goodbreak

\begin{proposition}\label{mia}
\begin{enumerate}
\item[{\rm \bf (i)}] 
The class of strongly balanced potentials  form a  convex set in $L^1_{loc}(\Omega)$.
\item[{\rm \bf (ii)}] 
Let $V \in L^1_{loc}(\Omega)$ be a strongly balanced potential in $\Omega$ 
and $W \in L^1_{loc}(\Omega)$ be nonnegative. Then $V+W$ is strongly balanced  in $\Omega$.
\end{enumerate}
\end{proposition}
\proof 

\medskip

{\bf (i)}  
Let $V_1,V_2$ be strongly balanced potentials in $\Omega$. Fix a compact set  $K \subset \Omega$ and choose $c_1,c_2$ and $U_1, U_2$ corresponding to $V_1,V_2$ from definition \ref{tre}. Let  $c:=\min\{c_1,c_2\}$ and $U:=U_1 \cup U_2$.   Using the convexity of  the map $x \mapsto x^+$  and the concavity of the square root function, we get for $t \in (0,1)$,
\begin{eqnarray}
c \int_{K} (tV_1 +(1-t)V_2)^+  |\xi| &\leq & c_1 t \int_{K } V_1^+  |\xi| 
 + c_2 (1-t)  \int_{K } V_2^+ |\xi| \notag \\
&\leq&  t \sqrt{Q_{V_1}(\xi)} + (1-t) \sqrt{Q_{V_2}(\xi)} +  \int_U |\xi|  \notag \\
&\leq&   \sqrt{Q_{tV_1+ (1-t) V_2}(\xi)} +  \int_U |\xi|  . \notag  
\end{eqnarray}

\medskip

{\bf (ii)}    Fix a compact $K  \subset \Omega$ and choose $c, U$ corresponding to $K$ from definition \ref{tre}.   Using the convexity of the map $x\mapsto x^+$  again, 
\begin{eqnarray}
\frac{c}{2} \int_{K} ( V +  W )^+   |\xi| &\leq & \frac{1}{2} \sqrt{Q_{V}(\xi)} + \frac{1}{2} \int_{U } |\xi|  +  c^*_K \Big( \int_\Omega W |\xi|^2 \Big)^{\frac{1}{2} }\notag \\
&\leq&  c_K \Big(\sqrt{Q_{V}(\xi)} + \sqrt{Q_{V+W}(\xi)} \Big) + \frac{1}{2} \int_{U} |\xi|  \notag \\
&\leq& 2 c_K \sqrt{Q_{V+W}(\xi)}+ \frac{1}{2} \int_{U} |\xi|  . \notag  
\end{eqnarray}
 \qed

\subsection{Balanced condition for $L^1$-subcritical operators} \label{s6.4}
In this subsection we show that for $L^1$-subcritical operators the balanced condition in definition \ref{tri} can be simplified and that the concepts of balanced and strongly balanced potentials coincide.
We first restate some of the previous results for an $L^1$-subcritical operator. Recall that for these operators the space ${\cal L}_V(\Omega; U)$ is equivalent to ${\cal L}_V(\Omega)$ for any open set $U \Subset \Omega$. 
\begin{proposition} \label{katha}
Let $V \in L^1_{loc} (\Omega)$ be such that ${\cal L}_V(\Omega) \hookrightarrow L^1_{loc}(\Omega)$
(i.e. $-\Delta+V$  is  an $L^1$-subcritical operator in $\Omega$). Then, 
\begin{enumerate}
\item[{\rm \bf (i)}] 
$V \hbox{ is strongly balanced in  } \Omega
  \quad \Longleftrightarrow  \quad
   {\cal L}_V(\Omega) \hookrightarrow L^1_{loc}( \Omega, V^{\pm} dx)$.
\item[{\rm \bf (ii)}] 
Let  ${\bf J}$ be as in definition \ref{amar}, and
denote by ${\bf J^{\pm}}$, the unique continuous extensions to  ${\cal H}_V(\Omega)$ of the two  imbeddings given in {\bf (i)}.
Then, ${\bf J^{\pm}(u)}={\bf J(u)}$ a.e. in $supp(V^\pm)$. 
\end{enumerate}
\end{proposition}
\begin{proof}
{\bf (i)} 
follows from proposition \ref{trie} (i).

\medskip

{\bf (ii)} 
Let $\{\xi_n\} \subset  W^{1,\infty}_c(\Omega)$ be a sequence converging to $u \in {\cal H}_V(\Omega)$. Then,  $\{\xi_n\}$ converges to ${\bf J}(u)$ in $L^1_{loc}(\Omega)$ and to ${\bf J}^\pm(u)$ respectively in  $ L^1_{loc}(\Omega, V^{\pm}dx)$  as $n \to \infty$. 
The result follows from a.e. pointwise convergence (up to a subsequence) of $L^1-$ convergent sequences.
\end{proof}

\medskip

For $L^1$-subcritical potentials of the form $V_0:= div \; \vec{{\bf f}} + |\vec{{\bf f}}|^2$ we show  injectivity of  the associated ${\bf J}_0$ map by connecting it to the closability of the ``magnetic operator".

\begin{lemma}\label{tca}
Let $\vec{{\bf f}} \in (L^2_{loc}(\Omega))^N$ and denote  $V_0:= div \; \vec{{\bf f}} + |\vec{{\bf f}}|^2$.  
Assume $V_0 \in L^1_{loc}(\Omega)$ is such that  $-\Delta +V_0$ is an $L^1$-subcritical operator in $\Omega$ with the associated map ${\bf J}_0$. 

Then ${\bf J}_0$ is injective iff  the magnetic operator
\begin{equation}\label{bia}
 T_{\vec{{\bf f}}}(\xi):= \nabla \xi - \vec{{\bf f}} \xi,\;\; \xi \in C_c^{1}(\Omega)
\end{equation}
is  closable with domain $ (C_c^{1}(\Omega), \| \cdot \|_1)$ to $(L^2(\Omega))^N$. 
\end{lemma}
\proof By a straightforward integration by parts, we get
\begin{equation}\label{haha}
Q_{V_0}(\xi) := \int_\Omega |\nabla \xi|^2 + V_0 \xi^2  = \int_\Omega |\nabla \xi - \vec{{\bf f}} \xi|^2 \;\; \forall \xi \in C_c^1(\Omega).
\end{equation}
Let $T_{\vec{{\bf f}}}$  be closable.  Take $\{\xi_n\} \subset C_c^1(\Omega)$ such that $\{ \xi_n \}$ is a Cauchy sequence in ${\mathcal H}_{V}(\Omega)$ and $\xi_n \to 0$ in $L^1_{loc}(\Omega)$. Then, it follows  from \eqref{haha} that $ \{T_{\vec{{\bf f}}}(\xi_n)\}$ is a Cauchy sequence in $(L^2(\Omega))^N$ and hence converges to $0$ by the closability property. From \eqref{haha} again, $\xi_n \to 0$ in ${\mathcal H}_{V_0}(\Omega)$ and hence ${\bf J}_0$ is injective.

For the converse, let $\{\xi_n\} \subset C_c^1(\Omega)$ be such that
$$
\xi_n \to 0 \;\text{ in } L^1_{loc}(\Omega) \;\;\text{ and }\;\; T_{\vec{{\bf f}}}(\xi_n) \to \vec{{\bf g}} \;\text{ in } (L^2(\Omega))^N.
$$
Then from \eqref{haha} we get $\{\xi_n\}$ is a Cauchy sequence in ${\mathcal H}_{V_0}(\Omega)$ and hence by injectivity of ${\bf J}_0$, $\xi_n \to 0$ in ${\mathcal H}_{V_0}(\Omega)$. By \eqref{haha} again, we get $\vec{{\bf g}}=0$. \qed

\medskip

\begin{lemma}\label{u2}
 $T_{\vec{{\bf f}}}$ is a closable operator from $L^1_{loc}(\Omega)$ to $(L^1(\Omega))^N$ for any $\vec{{\bf f}} \in (L^1_{loc}(\Omega))^N$. 
\end{lemma}
\proof Let $\{\xi_n\} \subset C_c^1(\Omega)$ be such that
\begin{equation}\label{u2.1}
\xi_n \to 0 \;\text{ in } L^1_{loc}(\Omega) \;\;\text{ and }\;\; T_{\vec{{\bf f}}}(\xi_n) :=  \vec{{\bf g}}_n \to \vec{{\bf g}} \;\text{ in } (L^1(\Omega))^N.
\end{equation}
Let us denote $\vec{{\bf f}}= (f_1, \cdots , f_N), \vec{{\bf g}}= (g_1, \cdots , g_N)$ and $\vec{{\bf g}}_n = (g^n_1, \cdots , g^n_N)$. Fix a cube $E:= \Pi_{i=1}^N (a_i,b_i)$ and choose any $\varphi \in C_c^1(E)$.  Consider the modified sequence $\{\tilde{\xi}_n:=  \varphi \xi_n\}$. Then, from \eqref{u2.1},
 \begin{equation}\label{u2.4}
\tilde{\xi}_n \to 0 \;\text{ in } L^1_{loc}(\Omega) \;\;\text{ and }\;\; T_{\vec{{\bf f}}}(\tilde{\xi}_n) =  \varphi T_{\vec{{\bf f}}}(\xi_n)+ \nabla \varphi \xi_n := \vec{\tilde{{\bf g}}}_n \to \varphi \vec{{\bf g}} \;\text{ in } (L^1(\Omega))^N.
\end{equation}
We note that $f_i$ is locally integrable on a.e. line parallel to the $i$-th co-ordinate axis. We can hence define for a.e. $(x_1, \cdots,x_{i-1},  x_{i+1},\cdots,  x_N) \in  \Pi_{ j \neq i} (a_j,b_j)$,
 $$ F_i(x) := -\int_{a_i}^{x_i} f_i(x_1, \cdots , x_{i-1}, s, x_{i+1}, \cdots , x_N) \; ds.$$
 We note that $F_i$ is absolutely continuous on $[a_i,b_i]$ for a.e. $(x_1, \cdots,x_{i-1},  x_{i+1},\cdots,  x_N) \in  \Pi_{ j \neq i} (a_j,b_j)$.
Denoting $\vec{\tilde{{\bf g}}}_n = (\tilde{g}^n_1, \cdots , \tilde{g}^n_N)$, we can then write \eqref{u2.4} as
\begin{equation}\label{u2.2}
\partial_i (\tilde{ \xi}_n  e^{F_i}  ) = \tilde{g}^n_i  e^{F_i} \; a.e. \;\text{ in } (a_i,b_i) \text{ for } \;  i=1, \cdots, N .  
\end{equation}
In the above equation $\partial_i$ stands for pointwise i-th partial derivative. For any $i$ we note that $(x_1, \cdots,x_{i-1}, a_i, x_{i+1},\cdots,  x_N) \not \in supp(\varphi)$.  For any $x \in E$, integrating along $i$-th coordinate direction in \eqref{u2.2},
 \begin{eqnarray}\label{u2.3}
( e^{F_i} \tilde{\xi}_n )(x) &=& ( e^{F_i} \tilde{\xi}_n )(x) -  (e^{F_i} \tilde{\xi}_n)(x_1, \cdots, x_{i-1}, a_i, x_{i+1}, \cdots,  x_N)\notag \\
&= & \int_{a_i}^{x_i} (\tilde{g}^n_i e^{F_i} )(x_1, \cdots , x_{i-1}, s, x_{i+1}, \cdots , x_N) \; ds. 
 \end{eqnarray}
 Since $\tilde{\xi}_n \to 0$ and $\tilde{g}^n_i \to g_i$  in $L^1_{loc}(\Omega)$, we obtain for a.e. $(x_1, \cdots,x_{i-1},  x_{i+1},\cdots,  x_N) \in  \Pi_{ j \neq i} (a_j,b_j)$ :
 $$
 \tilde{\xi}_n \to 0 \;\text{ and }\; \tilde{g}^n_i \to \varphi g_i \;\text{ in }\; L^1((a_i,b_i)).
 $$
 Taking subsequential limit  in \eqref{u2.3},  for a.e. $(x_1, \cdots,x_{i-1},  x_{i+1},\cdots,  x_N) \in  \Pi_{ j \neq i} (a_j,b_j)$ and a.e. $x_i \in (a_i,b_i)$ we get,
 $$
 \int_{a_i}^{x_i} (\varphi g_i  e^{F_i} )(x_1, \cdots , x_{i-1}, s, x_{i+1}, \cdots , x_N) ds =0, \; \;  i=1,2, \cdots, N .
 $$
 Hence $\varphi \vec{{\bf g}} \equiv {\bf 0}$ in $E$.
 \qed
 
 \medskip
 
\begin{corollary}\label{entr}
Let $\vec{{\bf f}} \in (L^2_{loc}(\Omega))^N$ and assume  $V_0:= div \; \vec{{\bf f}} + |\vec{{\bf f}}|^2  \in L^1_{loc}(\Omega)$ 
 is such that  $-\Delta +V_0$ is an $L^1$-subcritical operator in $\Omega$. Then the associated map ${\bf J}_0$ is injective. 
\end{corollary}
\proof Follows from lemmas \ref{tca} and \ref{u2}.  \qed

\medskip

Finally, we characterise balanced potentials in the $L^1$-subcritical operator context.

\begin{lemma}\label{att}
Let $V \in L^1_{loc}(\Omega)$ be such that $-\Delta+V$ is an $L^1$-subcritical operator in $\Omega$. Then the following are equivalent:
\begin{enumerate}
\item[{\rm \bf (i)}] 
$V$ is balanced in $\Omega$,
\item[{\rm \bf (ii)}]
${\bf J}$ is injective and 
$V {\bf J}(u) \in L^1_{loc}(\Omega)$ for any $u \in {\mathcal H}_V(\Omega)$,
\item[{\rm \bf (iii)}]
$V$ is strongly  balanced in $\Omega$.
\end{enumerate}
\end{lemma}
\proof   
{\bf (i)} $\Longrightarrow$ {\bf (ii)}: follows from  definition \ref{tri}. and corollary \ref{bru}. 

\medskip

{\bf (ii)} $\Longrightarrow$ {\bf  (iii)}:   Let ${\mathcal Z}:= {\bf J}({\mathcal H}_V(\Omega)) \subset L^1_{loc}(\Omega)$ with the norm
$$
\| z \| := \|u\|_{V,\Omega} \;\; \text{ where } \;\; z= {\bf J}(u), u \in {\mathcal H}_V(\Omega).
$$
Therefore,  ${\mathcal Z}$ and ${\mathcal H}_V(\Omega)$ are isometric. Given any $z= {\bf J}(u)$, we write $u=u_+-u_-$ as in lemma \ref{lem:BADAMILK}. Then, again as in that lemma,
$$
 {\bf J}(u_\pm) =  ({\bf J}(u))^\pm .
$$
Therefore, $|z| = {\bf J}(u_+ + u_-) \in {\mathcal Z}$ and 
$$
\| |z| \| := \| u_+ + u_-\|_{V,\Omega} \leq \|u\|_{V,\Omega}:= \|z\|.
$$
Thus, $V, {\mathcal Z}$ satisfy assumptions of proposition \ref{inz} and hence $V$ is strongly  balanced in $\Omega$.

\medskip

{\bf (iii)} $\Longrightarrow$ {\bf (i)}: follows from corollary \ref{nes}.
\qed

\medskip

\begin{remark}\label{psyc}
If $V:=V_0$ is as in corollary \ref{entr}, then we have that the statements (i) and (iii) of lemma \ref{att} are equivalent to the following 

 $$(ii)^{\prime} :  \;\; V_0 {\bf J}_0(u) \in L^1_{loc}(\Omega) \;\;\text{ for any }\;\; u \in {\mathcal H}_{V_0}(\Omega).$$
\end{remark}

\subsection{Examples of Strongly Balanced Potentials} \label{s6.5}

We now give examples of large class of strongly balanced potentials.
\subsubsection{Some very general examples}

\begin{remark}\label{gia}
Assume $Q_V \succeq 0$ in $\Omega$. 
\begin{enumerate}
\item[{\rm \bf (i)}] 
If either $V^{-} \in L^{\infty}_{loc} (\Omega)$, or $V^{+} \in L^{\infty}_{loc} (\Omega)$, then $V$ is strongly balanced in $\Omega$. This follows from proposition \ref{trie}.
\item[{\rm \bf (ii)}] 
Let $V^+ \in L^q_{loc}(\Omega)$ for some $q>1$. If ${\mathcal L}_V(\Omega) \hookrightarrow L^{q^\prime}_{loc}( \Omega)$ 
(where $q^{\prime}=\frac{q}{q-1}$), then $V$ is strongly balanced in $\Omega$.
\end{enumerate}
\end{remark}

\medskip

\begin{lemma}\label{rudra}
 Let $V \in L^1_{loc}(\Omega)$  be such that $Q_V \succeq 0$ in $\Omega$.
 \begin{enumerate}
\item[{\rm \bf (i)}]  
$V_\epsilon:= (1+\epsilon)V^+- V^-$ is a strongly balanced potential for any $\epsilon >0 $.
 \item[{\rm \bf (ii)}]
$\tilde{V}_\epsilon:= V^+-(1-\epsilon)V^-$ is a strongly balanced potential for any $\epsilon>0$.
\item[{\rm \bf (iii)}] 
Define $V_n := \min\{V,n\}.$ Then $V$ is strongly balanced in $\Omega$ if $Q_{V_n} \succeq 0$ in $\Omega$ for some $n$.
\end{enumerate}
\end{lemma}
\proof 
{\bf (i)} 
Clearly  $Q_{V_\epsilon} \succeq 0$ in $\Omega$. Now, for any compact set $K \subset \Omega$ and $\xi \in W_c^{1,\infty}(\Omega)$, 
$$
\frac{1}{1+\epsilon}\int_K V_\epsilon^+ |\xi|  = \int_K V^+ |\xi| \leq c_K\Big(\int_K V^+|\xi|^2\Big)^{\frac{1}{2}} \leq \frac{c_K }{\sqrt{\epsilon}}  \Big(\int_K  \epsilon V^+ |\xi|^2\Big)^{\frac{1}{2}} \leq \frac{c_K }{\sqrt{\epsilon}}  \sqrt{Q_{V_\epsilon}(\xi)}.
$$

{\bf (ii)} We note that 
$$
\epsilon \int_\Omega V^- |\xi|^2  \leq Q_{\tilde{V}_\epsilon} (\xi), \; \; \forall \xi \in W^{1,\infty}_c(\Omega).
$$
We may assume $\epsilon <1$ as otherwise the result is obvious by remark \ref{gia}(i). As above,
$$
\frac{1}{1-\epsilon} \int_K \tilde{V}_\epsilon^-  |\xi|  = \int_K V^- |\xi| \leq \frac{c_K }{\sqrt{\epsilon}}  \Big(\int_K  \epsilon V^- |\xi|^2\Big)^{\frac{1}{2}} \leq \frac{c_K }{\sqrt{\epsilon}}  \sqrt{Q_{\tilde{V}_\epsilon}(\xi)}.
$$
That $\tilde{V}_\epsilon$ is strongly balanced in $\Omega$ follows from proposition \ref{trie}. 

\medskip

{\bf (iii)} 
Noting that $V_n$ is a strongly balanced potential (since it is bounded above), the result follows from the identity $$V = \min \{ V, n\}+(V-n)\chi_{\{V \geq n\}}$$  and proposition \ref{mia}(ii).
\qed

\subsubsection{Balanced potentials whose positive and negative singularities are separated}

 We now define a class of potentials whose positive and negative singularities are separated.
 
  \begin{definition}\label{uno}
 Let ${\mathcal V}(\Omega)$ denote the class of functions  $V \in L^1_{loc}(\Omega)$  for which  we can find  a   cover of $\Omega$ by open sets $\{U_j\}$ , $U_j \Subset \Omega$,  such that  either $V^+$ is bounded on $U_j$ or else there exists an open set $O_j $ satisfying $U_j \Subset O_j \Subset \Omega$ and $V^-$ is bounded on $O_j$. 
 \end{definition} 
 
 \medskip
 
\begin{proposition}\label{reshma}
Let $V \in {\mathcal V}(\Omega)$.
\vspace{-2mm}
   \begin{enumerate}
  \item[{\rm \bf (i)}]  
  If $h_1,h_2 \in L^{\infty}_{loc}(\Omega)$ and $h_1 $ nonnegative, then $h_1 V+h_2 \in {\mathcal V}(\Omega)$.
  \item[{\rm \bf (ii)}] 
  $tV^+-sV^- \in {\cal V}(\Omega)$ for all $t,s>0$.
  \end{enumerate}
\end{proposition}
 \proof 
 We take the  cover $\{U_j\}$ for $V$.
  The same cover works for $h_1V+h_2$
and  $tV^+-sV^- $, $t,s>0$.
  \qed

\medskip

\begin{example} \label{duo}
\vspace{-2mm}
\begin{enumerate}
\item[{\rm \bf (i)}] 
If $V^{+} \text{ or } V^- \in L^{\infty}_{loc} (\Omega)$, then clearly $V \in {\mathcal V}(\Omega)$   (we simply take  $\{U_j\}$ to be any countable collection of open balls covering $\Omega$).
\item[{\rm \bf (ii)}]  
If  there exists  a compact set $ K$ and an  open set $O \Subset \Omega$ such that 
$$K \subset O , \; \sup _{\Omega \setminus K} V < +\infty \;\text{  and } \inf _O V > -\infty,$$
then $V \in {\mathcal V}(\Omega)$. In this case we take any open set $U$ such that $K \subset U \Subset O$ to be a member of the open cover   and note that $V$ is bounded from above outside $K$.
\item[{\rm \bf (iii)}]  
Let  $S:=\{x_i\}$ be a discrete set of points in $\Omega$ contained in an open set $U \subset \R^N$ such that $\overline{U} \subset \Omega$. Associated to each $x_i$ we choose $\alpha_i \in (-N,0)$ and  $a_i \geq 0$ such that $\sum_i a_i < \infty$. Then, any $V \in L^1_{loc}(\Omega)$ such that 
$$
0 \leq  V \leq  \sum_i a_i |x-x_i|^{\alpha_i}  \;\text{a.e.  in } U \;\; and \;\; V \leq 0  \; \text{a.e. in } \;\Omega \setminus U$$
 is in ${\mathcal V}(\Omega)$. Indeed, we  choose open balls $B_i \Subset U$ centered at the singularity $x_i$ and excluding  other singularities. Noting that $V$ is locally bounded above outside the union of these balls, it is easy to construct the cover $\{U_j\}$ as required by the definition \ref{uno} by enlarging the collection $\{B_j\}$. 
We can easily generalise this example by replacing the point singularities $\{x_i\}$ by a disjoint sequence of compact sets $\{K_i\}$  containing the positive singularities of $V$. 
\item[{\rm \bf (iv)}]  
Let $U $  and $\{a_i\}$ be as in (iii). Let $S:=\{x_i\}$ be any countable set  such that $\overline{S} \Subset U$ ($S$ maybe even dense in a portion of $U$). Choose  $\alpha \in (-N,0)$. Then any $V \in L^1_{loc}(\Omega)$ satisfying 
$$
0 \leq  V \leq  \sum_i a_i |x-x_i|^{\alpha}  \;\text{a.e.  in } U\;\;and \;\; V \leq 0  \; \text{a.e. in } \;\Omega \setminus U
$$  
   will be in ${\mathcal V}(\Omega)$ as it satisfies conditions in example (ii).
\item[{\rm \bf (v)}] 
Suppose there exist $x_0 \in \Omega$ and a sequence of open balls $B_{\epsilon_n}(x_0), \epsilon_n > 0,$ contained in $\Omega$ such that
$$
  \sup _{B_{\epsilon_n}} V = +\infty  
      \quad \text{ and } \quad 
   \inf_{B_{\epsilon}} V = -\infty \text{ for all } n.
$$
Then $V$ is not  in ${\mathcal V}(\Omega)$.
\end{enumerate}
\end{example}

\medskip

\begin{proposition}\label{paint}
 Let $V \in {\mathcal V}(\Omega)$  be such that  
 $Q_V \succeq 0$ in $\Omega$.
 Then $V$ is strongly balanced in $\Omega$.
\end{proposition}
\begin{proof}
 Let  $\xi \in W_c^{1,\infty}(\Omega)$. 
 Let  the cover $\{ U_j \}$ be as given in the definition \ref{uno}.  If  $V^+$ is bounded on $U_j$ then clearly
\begin{equation}\label{aza}
\int_{U_j} V^+ |\xi| \leq (\sup_{U_j}  V^+) \int_{U _j} |\xi|.
\end{equation} 
If not, we can obtain open set $O_j$ as in definition \ref{uno} on which $V^-$ is bounded. Choose $\varphi \in C_c^2(O_j)$, $0 \leq \varphi \leq 1$ and $\varphi \equiv 1$ on $U_j$. Then,  by \eqref{eq:sloka}, 
\begin{equation}\label{siya}
\int_{U_j} V^+ |\xi| \leq C_{j} \bigg (\int_{O_j} V^- |\xi| \varphi +  \sqrt {Q_V (\xi)} \bigg) \leq \tilde{ C}_{j} \bigg (\int_{O_j}  |\xi| +  \sqrt {Q_V (\xi)} \bigg). 
\end{equation}

From \eqref{aza} and \eqref{siya}  we get that for any $U_j$, 
$$
\int_{U_j}V^+ |\xi|  \leq C_j \bigg (\int_{O_j} |\xi| +  \sqrt {Q_V (\xi)} \bigg). 
$$
Given a compact set $K \subset \Omega$ we can cover $K$ by  finitely many $U_j$'s  and use the above inequality to verify the inequality in \eqref{eq:BalancedCondition} holds. 
\end{proof}


\begin{example}\label{uri}
Let $N \geq 3$, $a_i  \in \Omega, $ $\alpha_i \in (1, N/2)$ and $c_i \in {\mathbb R}$ for $i=1, \cdots, N$. We choose $|c_i| \geq 1$ if $\alpha_i=1$. Define 
$$\vec{{\bf \Gamma}}:= ( c_1 |x-a_1|^{-\alpha_1}, \cdots,  c_N |x-a_N|^{-\alpha_N}), \;\;\; V_0:= div \; \vec{{\bf \Gamma}}+ |\vec{{\bf \Gamma}}|^2.$$  
Using the proposition \ref{paint}, we can check that $V_0$ (and hence any locally integrable $V \geq V_0$) is strongly balanced in $\Omega$.

For more examples of potentials satisfying the conditions in the  proposition \ref{paint}, see corollary \ref{loe} and remark \ref{waz}.
\end{example}
 \subsubsection{ Balanced potentials whose energy space imbeds in $H^1_{loc}$}\label{s6.5.3}
The following result shows that a large class of locally integrable functions that fall in the framework of \cite{JayeMazyaVerb} are balanced potentials. We also remark that this result allows a balanced potential to oscillate infinitely often near a singularity, but not in a ``wild" manner; see remark \ref{sci} in this context. Example \ref{turk} (see again remark \ref{sci}) shows that the assumption on $V^+$ made in the following lemma is sharp.

\begin{lemma}\label{swa}
Let $V \in  L^1_{loc}(\Omega)$  be such that $V^+ \in L^{\frac{2N}{N+2}}_{loc}(\Omega)$  when $N \geq 3$ and $V^+ \in L^{p}_{loc}(\Omega)$ for some $p >1$ when $N= 2$. 
If  $Q_V \succeq 0$ in $\Omega$ and ${\cal L}_V(\Omega) \hookrightarrow H^1_{loc}(\Omega)$, then $V$ is strongly balanced in $\Omega$. Same conclusion holds if  above conditions are imposed on $V^-$  instead of $V^+$.
\end{lemma}
\proof Let $N \geq 3$. We  estimate for any ball $B \Subset \Omega$ and $\xi \in W_c^{1,\infty}(\Omega)$, 
\begin{eqnarray}\label{faz}
\int_{B }V^+ |\xi| 
&\leq& \Big(\int_B (V^+)^{\frac{2N}{N+2}} \Big)^{\frac{N+2}{2N}}\Big(\int_B |\xi|^{\frac{2N}{N-2}}\Big)^{\frac{N-2}{2N}} \notag \\
&\leq&  \tilde{c}_B \| \xi \|_{H^1(B)} \notag \\
&\leq&  c^*_B \sqrt{Q_V ( \xi )}. \notag
\end{eqnarray}
A simple covering argument shows that $V$ is strongly balanced in $\Omega$.
A similar proof holds for $N=1,2$. The case of $V^-$ can be shown using  proposition \ref{trie}(i).
\qed 

Next,  we make the following definition which  plays a main role in showing the equivalence of $L^1$ and feeble $L^2$ criticalities (see proposition \ref{prop:Sonata}):
  \begin{definition}\label{chalka}
We call $V \in L^1_{loc}(\Omega)$ a tame potential in $\Omega$ if $V$ is balanced in $\Omega$ and $V^+ \in L^p_{loc}(\Omega)$ for some $p > \frac{N}{2}$ \;($N \geq 2$).
\end{definition}

\begin{lemma}\label{saud}
Let $V \in L^{p}_{loc}(\Omega)$ for some $p > \frac{N}{2}$ when $N \geq 2$ and $V \in L^{1}_{loc}(\Omega)$ when $N=1$. If $Q_V \succeq 0$  in $\Omega$, then $V$ is strongly balanced (infact, tame) in $\Omega$.
\end{lemma}
\proof We shall use some results from \cite{PinTint} to show the lemma. 
If $-\Delta+V$ is $L^1$-critical in $\Omega$, let $\Phi>0$ denote the ground state solution of this operator. Fix a ball $B \Subset \Omega$ and  a $\psi \in C_c^{\infty}(B)$ such that $\int_\Omega \psi \; \Phi \neq 0$. Consider the quadratic form
$$
\tilde{Q}(\xi):= Q_V(\xi) + \Big|\int_\Omega  \psi \; \xi \Big|^2,\;\; \xi \in W_c^{1,\infty}(\Omega).
$$
Let $\tilde{{\mathcal H}}(\Omega)$ denote the Hilbert space obtained by completing  the space $\tilde{{\mathcal L}}(\Omega):=(W_c^{1,\infty}(\Omega),(\tilde{Q})^{\frac{1}{2}})$. It can be shown (see theorem 1.4 in \cite{PinTint}) that the above quadratic forms generate equivalent norms  on $\tilde{{\mathcal H}}(\Omega)$ as such $\psi \in C_c^{\infty}(\Omega)$ is varied. Then from proposition 3.1 in \cite{PinTint} we obtain that if $-\Delta+V$ is $L^1$-subcritical in $\Omega$, then ${\mathcal L}_V(\Omega) \hookrightarrow H^1_{loc}(\Omega)$ and  if $-\Delta+V$ is $L^1$-critical in $\Omega$, then $\tilde{{\mathcal L}}(\Omega) \hookrightarrow H^1_{loc}(\Omega)$. We can proceed as in the proof of the last  lemma to conclude noting that
$$
\Big| \int_\Omega  \psi \; \xi \Big|  \leq c_B \int_B |\xi|.
$$
\qed

\begin{remark}\label{naw}
Since all our results beginning from  section \ref{s8} 
are  valid for any tame potential, it follows that they generalise (in view of the above  lemma) the corresponding results in Murata \cite{Murata:1986}  and Pinchover-Tintaraev \cite{PinTint}.
\end{remark}
  
  We provide examples of potentials for which  a  strengthened version of assumption (ii) of definition \ref{tri} is enough to ensure that it is strongly balanced.
 
 \begin{proposition} \label{binz}
Let $N \geq 3$ and $V_0 :=  div \; \vec{{\bf f}}+ |\vec{{\bf f}}|^2$ where $ \vec{{\bf f}} :=(f_1,\cdots,f_N) $ satisfies the following conditions (for $i=1, \cdots , N$): 
\vspace{-2mm}
\begin{enumerate}
\item[{\rm (i)}] 
 $f_i \in L^{\frac{4N}{N+2}}_{loc}(\Omega)$ 
\item[{\rm (ii)}]
$(\partial_i f_i ) \; w \in L^1_{loc}(\Omega)$ for any $w \in H^1_{loc}(\Omega)$. 
\end{enumerate}
If   $V \in L^1_{loc}(\Omega)$ is such that $V \geq t V_0$ in $\Omega$ for some $t \in (0,1)$, 
then $V$ is  strongly balanced in $\Omega$. 
 \end{proposition}
\proof 
From theorem~\ref{vat}, it follows that $Q_{V_0} \succeq 0$ in $\Omega$.  Fix any $t \in (0,1)$. Then, ${\cal L}_{tV_0}(\Omega) \hookrightarrow H^1_{loc}(\Omega)$.  We note that  for any  $w \in H^1_{loc}(\Omega)$ and any ball $B \Subset \Omega$,
$$
\int_B |\vec{{\bf f}}|^2 | w| \leq \| |\vec{{\bf f}}|^2 \|_{ L^{\frac{2N}{N+2}}(B)} \| w \|_{ L^{\frac{2N}{N-2}}(B)} < +\infty. $$
Therefore, from Proposition \ref{inz} we obtain that $tV_0$ is strongly balanced in $\Omega$. That $V$ is also strongly balanced follows from proposition \ref{mia} (ii). 
\qed 

\medskip

Let $\partial_i$ and $\partial^p_i$ stand respectively for the $i$-th distributional and pointwise derivatives. In what follows we fix a generic  cube  $E := \Pi_{i=1}^N [a_i,b_i] \Subset \Omega$. First we state the following standard result : 

 \begin{proposition}\label{yad}
 Let $f \in L^1(E)$. If $\partial_1 f  \in L^1(E)$, then for a.e. $\tilde{x} \in \Pi_{i=2}^N[a_i,b_i] $ the function $f(\cdot ,\tilde{x})$ is absolutely continuous on $[a_1,b_1]$ and $\partial_1 f( \cdot, \tilde{x}) = \partial^p_1 f( \cdot, \tilde{x})$ a.e.  in $ [a_1,b_1]$. Corresponding statement holds for any  $i=2, \cdots ,N$.
 \end{proposition}
 
 \medskip

We now provide easier to verify sufficient conditions  ensuring assumption (ii) of proposition \ref{binz} holds. 

\begin{proposition}\label{miz}
\begin{enumerate}
\item[{\rm \bf (i)}]
Let   $f_i \in L^2_{loc}(\Omega)$. Suppose that for any cube $E \Subset \Omega$, we can write  $\partial_i f_i $  as
\begin{equation}\label{ret}
\partial_i f_i = \partial_i g_+ - \partial_i g_-  \;\; \text{ where } \;\; g_{\pm} \in L^2(E), \; \partial_i g_{\pm} \in L^{1}(E) \;\;\text{ and } \;\; \partial_i g_{\pm} \geq 0.
\end{equation}
Then  $\partial_i f_i$ satisfies condition (ii) in proposition \ref{binz}. 
\item[{\rm \bf (ii)}]
 Let $f_1 \in L^2_{loc}(\Omega)$ and $\partial_1f _1 \in L^1_{loc}(\Omega)$. Suppose that for any cube $E \Subset \Omega$,
 $$
 \int_{\Pi_{i=2}^N[a_i,b_i]} \Big( \int_{a_1}^{b_1} |\partial_1 f_1 | dx_1 \Big)^2 d \tilde{x}  < \infty.
 $$
 Then $\eqref{ret}$ holds for $f_1$. A corresponding statement is true for any $i= 2, \cdots N$.
\item[{\rm \bf (iii)}] 
  If  $f_i \in L^{2}_{loc}(\Omega), \; \partial_i f_i \in L^{1}_{loc}(\Omega)$ and $\partial_i f_i$ is  locally bounded above (or below), then $f_i$ satisfies condition (ii) in 
proposition \ref{binz}. 
\end{enumerate}
\end{proposition}
\proof 
{\bf (i)}
For convenience, we may take $i=1$. For $w \in H^1_{loc}(\Omega)$ 
and   any  $\varphi \in C^1_c(E)$ nonnegative, by Tonelli's theorem and proposition \ref{yad},
 \begin{eqnarray}
\int_E (\partial_1 g_\pm) |w|\varphi &=& \int_{\Pi_{i=2}^N[a_i,b_i]} dx_2 \;  \cdot \cdot \cdot dx_N \int_{a_1}^{b_1}( \partial_1 g_\pm) |w| \varphi \; dx_1  \notag \\
&=& - \int_{\Pi_{i =2}^N[a_i,b_i]} dx_2 \;  \cdot \cdot \cdot dx_N \int_{a_1}^{b_1} g_\pm \partial_1 (|w| \varphi)  \; dx_1\notag \\
&=&- \int_{E}  g_\pm \partial_1 (|w| \varphi) < +\infty. \notag
\end{eqnarray}

\medskip

{\bf (ii)}
We define
 $$
 g_{\pm}(x_1,\tilde{x}) := \int_{a_1}^{x_1} (\partial_1 f_1)^\pm(t, \tilde{x}) \; dt, \;\; x_1 \in [a_1,b_1] \;\;\; \text{ for a.e. } \; \tilde{x} \in \Pi_{i=2}^N [a_i,b_i] .
$$
We also have $g_\pm \in L^2(E)$, $\partial_1 g_\pm = (\partial_1 f
_1)^\pm \in L^1(E)$  
and hence $\partial_1f_1= \partial_1 g_+ - \partial_1 g_- $. 

\medskip

{\bf (iii)}
We note that on any cube $E \Subset \Omega$, we have $\partial_i (Cx_i \pm f_i ) \geq 0$ for some constant $C_E>0$ and can apply Tonelli's theorem as above.
\qed

\medskip

\begin{corollary}\label{yid}
 If $f_1(x)= \rho(x_1)\eta(x_2,x_3,\cdot \cdot \cdot , x_N)$ for some $\rho \in W^{1,1}_{loc}({\mathbb R})$ and $\eta \in L^2_{loc}({\mathbb R}^{N-1})$, then $f_1$ satisfies the assumptions of proposition \ref{miz} (ii) and hence \eqref{ret} holds. 

\end{corollary}
See also example \ref{ma} and theorem 4.1 in \cite{JayeMazyaVerb} in the context of the above two results.
\begin{remark}\label{jod}
From the construction of $g_\pm$ in proposition \ref{miz} (ii) it is clear that \eqref{ret} always holds if we relax the requirement $g_\pm \in L^2_{loc}(\Omega)$ to $g_\pm \in L^1_{loc}(\Omega)$.
\end{remark}

\subsubsection{Some well-known  potentials are balanced}

We provide some examples of strongly balanced potentials not in $L^p_{loc}$ for some $p>\frac{N}{2}$ and as well as not in the class ${\mathcal V}$.

\begin{example} \label{swam}
\vspace{-2mm}
\begin{enumerate}
\item[{\rm \bf (i)}] 
Let  $N \geq 3$ and $1/S_N$ denote the norm of the imbedding ${\cal D}_0^{1,2}(\R^N) \hookrightarrow L^{\frac{2N}{N-2}}(\R^N)$. Assume $V \in L^1_{loc}(\Omega)$ satisfies
$
\|V^-\|_{L^{\frac{N}{2}}(\Omega) } \leq S^2_N.
$
Then
$Q_{-V^-} \succeq 0$ in $\Omega$ and hence $-V^-$ is strongly balanced in $\Omega$ by  remark \ref{gia} (i). Thus $V$ is strongly balanced in $\Omega$ by proposition \ref{mia} (ii). Using example \ref{duo} (v), we can choose $V^+,V^-$  such that $V$  is not in  ${\cal V}(\Omega)$.
\item[{\rm \bf (ii)}] 
Let $V,W$ be as in proposition \ref{mia}(ii). Suppose  that 
$$\tilde{c}_B := \int_B \frac{1}{W}< +\infty \;\text{ for any ball }\; B \Subset \Omega.$$  Then
$$\frac{1}{\sqrt{\tilde{c}_B}} \int_B |\xi| \leq \frac{1}{\sqrt{\tilde{c}_B}} \Big(\int_B \frac{1}{W} \Big)^{\frac{1}{2}} \Big(\int_B W |\xi|^2\Big)^{\frac{1}{2}}   \leq  \sqrt{Q_{V+W} (\xi)}.
$$
That is, we additionally have that $-\Delta+V+W$ is  an  $L^1-$ subcritical operator in $\Omega$.   
\end{enumerate}
\end{example}

\medskip 

We now give some examples that involve Hardy type potentials.

\begin{example} \label{swamBis}
Let $N \geq 3$ and consider the Hardy constant $H_N:= \frac{(N-2)^2}{4}$. 
\vspace{-2mm}
\begin{enumerate} 
\item[{\rm \bf (i)}]
Let $\{x_i\}_{i=1}^k \subset \R^N$ and choose $\{a_i \}_{i=1}^k \subset \R \setminus\{0\}$ such that 
$$\alpha:=\sum_{i=1}^k \max\{a_i,0\} \leq  H_N.$$
 Define
$$
V(x) := - \sum_1^k a_i |x-x_i|^{-2}, x \not\in \{x_1,x_2,\cdot \cdot \cdot, x_k\}.
$$
Then, we can write $V=V_0+V_1$  by summing all the  poles with $a_i<0$ to get $V_0$ and those with $a_i>0$  to get $V_1$. It can be shown that $Q_{V_1} \succeq 0$ in $\R^N$ (and ${\mathcal L}_V(\R^N) \hookrightarrow {\mathcal D}_{loc}^{1,2}(\R^N)$ if $\alpha<H_N$); see proposition 1.2 in \cite{FT}. Since $V_1 \leq 0$, it is strongly balanced in $\Omega$. Thus, by proposition \ref{mia}(ii), $V$ is a strongly balanced potential in $\R^N$.  
\item[{\rm \bf (ii)}]
If $\{a_i\}_{i=1}^k$ in the previous example are chosen so that 
$$
a_i < H_N \;\text{ and }\; \sum_1^k a_i < H_N,
$$
then we can find atleast one configuration of poles $\{x_i\}_{i=1}^k \subset \R^N$ so that $Q_V \succeq 0$ in $\R^N$ and ${\mathcal L}_V(\R^N) \hookrightarrow {\mathcal D}_{0}^{1,2}(\R^N)$ (see theorem 1.1 in \cite{Ter}). We see that $V \in {\mathcal V}(\R^N)$  from example \ref{duo} (iii) and hence is strongly balanced in $\R^N$.
\item[{\rm \bf (iii)}]
We take $S:=\{x_i\} \subset \R^N$ to be a countable dense set and $\{a_i\}$ be a sequence of positive numbers whose sum is not bigger than $H_N$. Consider the Hardy potential with poles on the dense set $S$:
$$
V(x) := -\sum_i a_i |x-x_i|^{-2}.
$$
Recalling that $-\Delta -H_N |x|^{-2}$ is a non-negative operator in $\R^N$, we obtain for any $ \xi \in C_c^{\infty}(\R^N)$,
\begin{eqnarray}
\int_{\R^N} |\nabla \xi|^2 \geq \frac{1}{H_N} \sum_i  a_i \int_{\R^N} |\nabla \xi(\cdot+x_i)|^2 \geq \sum_i a_i \int_{\R^N} |x|^{-2} \xi^2(\cdot+x_i) =- \int_{\R^N} V \xi^2. \notag
\end{eqnarray}
Thus $Q_V \succeq 0$ in $\R^N$ and since $V \leq 0$ it is a strongly balanced (infact a tame) potential in $\R^N$. Indeed, $V \not \in L^{\frac{N}{2}}(B)$ for any ball $B$.
\item[{\rm \bf (iv)}]
Consider the critical Hardy operator $-\Delta +V$, $V:=- H_N |x|^{-2},$  on any bounded domain $U \subset \R^N, N \geq 3$. Then it is well known that the operator is $L^1$-subcritical in $U$ and ${\mathcal L}_V(U) \not \hookrightarrow H^1_0(U).$ Clearly, $V$ is tame in $U$.
\end{enumerate}
\end{example}

\medskip

For more examples of strongly balanced potentials see corollaries \ref{lee} and \ref{loe} in \S 11.2. These examples require some further results than presented so far.

\subsection{Examples of non balanced potentials}
\label{s6.6}

The previous discussion shows that if the negative and positive parts of the potential $V$ do not interact too much, then the potential is balanced. 
We describe below a non balanced potential that  oscillates infinitely often near its point singularity. We also point out that these oscillations are ``wild" in the sense of remark \ref{sci} below.

\begin{example}\label{turk}
Denote $B:=B_1(0) \subset {\mathbb R}^N, N \geq 3$ and fix any $t \in (0,1)$. Let $H_N:= (\frac{N-2}{2})^2$ denote the critical Hardy constant. We take any $c \in (0, H_N/4\;], \; \alpha<0$ and consider the potential
$$
V_\alpha (x):= - c \; \sin^2(|x|^\alpha) |x|^{-2}, \;\; x \in B.
$$
Define then the vectorfield $\vec{{\bf \Gamma}}_\alpha := (\sqrt{-V_\alpha},0, \cdots, 0)$.
Thanks to Hardy's inequality we have $$Q_{-4|\vec{{\bf \Gamma}}_\alpha|^2} = Q_{4V_\alpha} \succeq 0 \;\; \text{ in } \; B.$$
By setting  $\sigma_\alpha := div \; \vec{{\bf \Gamma}}_\alpha$, 
an explicit computation gives,
$$
    \sigma_\alpha(x)= \sqrt{c} \; x_1 \Big( \alpha |x|^{\alpha-3}\cos(|x|^{\alpha}) - \sin(|x|^{\alpha})|x|^{-3} \Big),
$$
We choose $\alpha> 2-N$ and  check that $\sigma_\alpha \in L^1(B)$. Now, by theorem 4.1 (iii)  in \cite{JayeMazyaVerb}, 
$$
 \Big | \int_B \sigma_\alpha \xi^2 \Big| \leq \int_B |\nabla \xi|^2 \;\;\quad  \forall \; \xi \in W_c^{1,\infty}(B).
$$
In particular, $\sigma_\alpha \in H^{-1}_{loc}(B)$. It is  easy to see from the above inequality that we can identify ${\mathcal H}_{t \sigma_\alpha}(B)$ with $H^1_0(B)$. If we choose $\beta \in ( \frac{2-N}{2},0)$, then 

$$w_\beta(x):=|x|^{\beta}-1 \in H^1_0(B).$$

A straightforward consideration involving radial-angular coordinates  shows that
\begin{equation}\label{wb}
\sigma_\alpha w_\beta \not \in L^1(B_\epsilon(0)) \;\; \text{ for any } \;\; \epsilon \in (0,1]\;\; \text{ if }\;\;  \alpha + \beta \leq 2 -N.
\end{equation}
Therefore, for any $\alpha \in (2-N, \frac{2-N}{2})$ we may find $\beta \in ( \frac{2-N}{2},0)$ such that \eqref{wb} holds. Hence 
 $t\sigma_\alpha$ is not balanced in $B$ for any $\alpha \in (2-N, \frac{2-N}{2})$.
\qed

\end{example}

\begin{remark}\label{sci}
We note that $\sigma_\alpha^+ \not \in L^{\frac{2N}{N+2}}_{loc}(B)$ if $\alpha \leq \frac{2-N}{2}$ and hence for such $\alpha$ the potential $t \sigma_\alpha$ for $t \in (0,1)$ does not satisfy the assumptions of lemma \ref{swa}. Conversely, if $\alpha > \frac{2-N}{2}$, then $\sigma_\alpha  \in L^{\frac{2N}{N+2}}(B)$ and lemma \ref{swa} ensures that $t \sigma_\alpha$ is balanced in $B$ for $t \in (0,1)$. Thus, balanced potentials are allowed to oscillate infinitely often near a singularity as long as these oscillations are not ``wild". 
\end{remark}

\medskip

\section{AAP principle in the space of Distributions for general nonnegative operators} \label{s7}

In theorems \ref{AAP-distr} and \ref{thm:converse3} below, we will show that  given a non-negative quadratic form $Q_V$ in $\Omega$ with $V \in L^p_{loc}(\Omega)$ for some $p>\frac{N}{2}$  \;($N \geq 2$), the set ${\cal C}_V(\Omega)$ is  non-empty and the converse for any $V \in L^1_{loc}(\Omega)$.

\subsection{${\cal C}_V (\Omega)\neq \emptyset$ implies $Q_V \succeq 0$ in $\Omega$}\label{s7.1}
We recall the following well-known result (whose proof is given in the appendix 12.2):
\begin{proposition}\label{prop:shave}
Let $\Omega_*$ be a bounded open set, $W \in L^1_{loc}(\Omega_*), W \geq 0$, $f \in L^{\infty}(\Omega_*), f \geq 0$. If $u\in H^1_0(\Omega_*)$ is a non-negative nonzero weak solution of $-\Delta u + W u=f$ in $\Omega_*$, then
$$
\int_{\Omega_*} |\nabla \xi |^2\,dx
+
\int_{\Omega_*} W \xi^2
\ge
\int_{\Omega_*} \left(\frac{f}{u}\right)\xi^2, \;\; \forall \xi \in C_c^{\infty}(\Omega_*) .
$$
In particular, $f/u \in L^1_{loc}(\Omega_*)$.
\end{proposition}
We extend the above result to the distributional framework for any $V \in L^1_{loc}(\Omega)$.
\begin{theorem}\label{AAP-distr}
Let $V \in  L^1_{loc}( \Omega)$ and  $\Phi \in {\cal C}_V (\Omega)$ (i.e.,  $\Phi \not \equiv 0$, $\Phi \geq 0$ satisfies  $-\Delta \Phi +V\Phi  \geq h$  in ${\cal D^\prime}(\Omega)$ for some non-negative $h \in L^1_{loc} (\Omega)$). Then, 
\begin{equation}\label{eq:ground-distr-General}
   Q_V(\xi)  \ge
    \int_{\Omega}\left( \frac{h}{\Phi}\right) \,\xi^2\,dx,
\qquad\forall  \xi \in W^{1,\infty}_c(\Omega).
\end{equation}
In particular,  $h/\Phi \in L^1_{loc} (\Omega)$.
\end{theorem}
\begin{proof}
From theorem \ref{thm:StrongMax} we obtain that $\Phi>0$ a.e. in $\Omega$. Now if both $h \equiv 0$ and $V^- \equiv 0$,  then the result follows immediately. Therefore, we may assume $h + V^- \Phi \not \equiv 0$ in $\Omega$. Define the truncation function
 $$
    T_k (s) :=
    \left\{ \begin{array} {ll}
      k &\hbox{ if } \,  s > k, \\
      s &\hbox{ if } \,  0 < s \leq k \\
      0 &\hbox{ if } \,  s \leq  0.
    \end{array}\right.
 $$
 For each $k \in \mathbb N$ define
$$
   H_k :=
   T_k(h+V^- \, \Phi) \in L^\infty(\Omega).
$$

Note that
$$
  0\le H_{k}\le H_{k+1}\le\dots
  \le h + V^{-} \Phi \in L^1_{loc}( \Omega).
$$ 

 Fix a $\xi \in   W^{1,\infty}_c(\Omega)$ and choose an open bounded set $\Omega_*  \Subset \Omega$ with smooth boundary which contains the support of $\xi$ as well as so that $h + V^- \Phi \not \equiv 0$ in $\Omega_*$. In particular, we have that $H_k \not \equiv 0$ in $\Omega_*$ for all large $k$.

Consider then the unique nonnegative (distributional) solution $w_k$ of the Dirichlet problem
\begin{equation}\label{lin-kBis}
    -\Delta w_k + V^{+} w_k = H_k,   \quad w_k\in H^1_0(\Omega_*),
\end{equation}
whose existence is ensured by corollary~\ref{cor:OdeOde}. It is also the case that $w_k >0$ a.e. in $\Omega_*$ for all large $k$ since we have $H_k \not \equiv 0$ in $\Omega_*$ for all large $k$.

{\bf Claim:} {\it $w_k \leq \Phi$ a.e. in $\Omega_*$.}\\
{\it Proof of claim:} 
Let $\Omega_{**}$ be a bounded open subset of $\Omega$ with $\Omega_* \Subset \Omega_{**} \Subset \Omega$. Since $V \Phi \in L^1_{loc}(\Omega)$,  by proposition \ref{prop:RanRan},  we get $\Phi \in W_{loc}^{1,1}(\Omega)$.
We use the following version of Kato's inequality (see proof of proposition B.5 in \cite{BMP}):
$$
\text{{\bf Kato:}}\quad \hbox{Let } u, f \in L^1(\Omega_*)
\hbox{ solve } -\Delta u \leq f \hbox{ in } {\cal D^\prime}(\Omega_*). 
\hbox{ Then, }  -\Delta (u^+) \leq \chi_{\{u \geq 0\}}f  \hbox{ in } {\cal D^\prime}(\Omega_*).
$$
We note that 
$$-\Delta (w_k -\Phi) \leq -V^+(w_k-\Phi) \text{ in } {\cal D^\prime}(\Omega_*).$$ 
Therefore, by the above Kato's inequality, 
$$-\Delta (w_k-\Phi)^+ \leq -\chi_{\{w_k-\Phi \geq 0\}} V^+(w_k-\Phi) \leq 0 \hbox{ in } {\cal D^\prime}(\Omega_*).$$
Further, since $\Phi \geq 0$ and $\Phi \in  W_{loc}^{1,1}(\Omega)$, we get that $(w_k-\Phi)^+=0$ on $\partial \Omega_*$.  From lemma \ref{lem:power}, we get $ (w_k-\Phi)^+ \leq 0$ a.e.  This proves the claim.

By similar arguments as in the claim above, but applied to $w_{k}-w_{k+1}$, we conclude that
$$
  0 \leq w_{k}\le w_{k+1} \;\hbox{ in } \;\Omega_*. 
$$

Let  $w_k\to w_\infty$ a.e. Note that  $0 \leq w_\infty \leq \Phi$ and $w_\infty \not\equiv 0$. By dominated convergence theorem, $w_k \to w_\infty$  in $L^1_{loc}(\Omega_*)$ and hence $w_\infty$ is a distributional solution of $-\Delta w_\infty+V^+ w_\infty = h+V^- \Phi$ in $\Omega_*$. From theorem \ref{thm:StrongMax} we conclude that $w_\infty >0$ a.e. in $\Omega_*$.

From \eqref{lin-kBis} and proposition \ref{prop:shave} we conclude that for all large $k$,
\begin{equation}\label{eq:scot}
\int_{\Omega_*} |\nabla \xi |^2\,dx
+
\int_{\Omega_*} V^{+} \xi^2
\ge
\int_{\Omega_*} \left(\frac{H_k}{w_k}\right)\xi^2 .
\end{equation}

Noting that pointwise we have
$$
   0 \le
   \frac{H_k}{w_k}
   =
   \frac{ T_k(h + V^- \Phi)}{w_k}
   \, \, \to \, \,
 \frac{ h+ V^{-} \Phi}{w_\infty}
   \quad\text{a.e. in } \Omega_*,
$$
and applying Fatou Lemma, we conclude
$$
  \liminf_{k \to \infty}
  \int_{\Omega_*} \left(\frac{H_k}{w_k}\right) \xi^2
  \ge
  \int_{\Omega_*} \left(\frac{ h+  V^- \Phi}{w_{\infty}}\right)\xi^2 .
$$
Finally, since $w_\infty \leq \Phi$, from \eqref{eq:scot} and the above inequality we obtain,
$$
\int_{\Omega_*} |\nabla \xi|^2 + \int_{\Omega_*}
V \xi^2 \geq \int_{\Omega_*} |\nabla \xi|^2 + \int_{\Omega_*}
V^+\xi^2 -\int_{\Omega_*} V^-\left(\frac{\Phi}{w_\infty}\right)\xi^2
\ge \int_{\Omega_*} \left(\frac{h}{\Phi}\right)\xi^2.$$ 
Since $\xi$ was any test function with support in $\Omega$, the theorem follows.
\end{proof}

\begin{remark}\label{nbh}
If in the above theorem 
$$ \inf_B \frac{h}{\Phi} >0  \;\;\text{  for any }\;\; B \Subset \Omega$$
 then in fact we obtain that  ${\mathcal L}_V(\Omega) \hookrightarrow L^2_{loc}(\Omega)$. This is a situation one encounters when $V$ is regular (say in $L^p_{loc}(\Omega)$ for some $p> \frac{N}{2}$). 
\end{remark}

We can now state a version of the above theorem for the energy supersolution.

\begin{proposition}\label{AAP-Relaxed}
Assume $V \in L^1_{loc}(\Omega), V$ is balanced  and $-\Delta + V$ is $L^1-$subcritical (i.e., ${\cal L}_V(\Omega) \hookrightarrow L^1_{loc}(\Omega)$).   
Given  a non-negative  non-zero function $ f \in L^{1}_{loc} (\Omega)$, let $u \in {\cal H}_V(\Omega)$  be an energy supersolution of $-\Delta u + Vu=f$. That is,
$$
     a_{V,\Omega} (u, \xi) \geq \int_{\Omega} f \xi ,
     \quad \forall \,  0 \leq \xi \in W_c^{1,\infty} (\Omega) .
   $$

Then,
\begin{equation}\label{ground-RElaxed}
   Q_V ( \xi)\,
   \ge
   \,
   \int_\Omega  \, \left(\frac{f}{{\bf J}(u)} \right) \xi^2dx
     \quad\forall   \xi \in W_c^{1,\infty} (\Omega).
\end{equation}
\end{proposition}
\begin{proof}
It follows from lemma \ref{lem:BADAMILK2} and proposition \ref{prop:EnergyVsDist} that ${\bf J}(u)$ is a non-negative non-zero distributional supersolution. In particular, from theorem \ref{thm:StrongMax}, ${\bf J}(u)>0$ a.e. in $\Omega$. The result now follows from theorem \ref{AAP-distr}.
\end{proof}

\subsection{$V^+ \in L^p_{loc}(\Omega)$ for some $p>\frac{N}{2}$  and $Q_V \succeq 0$ in $\Omega$  imply ${\cal C}_V (\Omega)\neq \emptyset$} \label{s7.2}

From standard existence results for Green's function (see, for instance, \cite{Zhao}), we have :
\begin{proposition}\label{prop:Green}
Let  $V^+ \in L^{p}_{loc}(\Omega)$ for some $p>\frac{N}{2}$  \;($N \geq 2$) and $\Omega_* \Subset \Omega$ be a domain with smooth boundary. Then, there exists a Green's function $G$  for $-\Delta + V^+$ in $\Omega_*$ satisfying the following properties:
\begin{itemize}
\item[{\bf (i)}] 
 $G : \Omega _* \times \Omega_* \to \R$, $G>0$ in $\Omega_*$, $G(x,y)=G(y,x)$,
\item[{\bf (ii)}]
$G(x,y) \sim \delta(x)$ (the distance to the boundary function) uniformly for $x$ near $\partial \Omega_*$ and $y$ restricted to a  compact subset of $\Omega_*$ and 
\item[{\bf (iii)}]
 for a given $f \in L^{\infty}(\Omega_*)$, the very weak solution of the boundary value problem
  $$(-\Delta+ V^+)u=f \;\text{ in } \;\Omega_*, \; u=0  \;\text{ on }\; \partial \Omega_*$$  is given by the Green's representation  formula $u(x)= \int_{\Omega_*}G(x,y)f(y) dy$.
\end{itemize}
\end{proposition}
\begin{remark}\label{ackn}
The existence of Green's function, Green's representation formula as well as Harnack inequality can be ensured under weaker assumptions on $V^+$, for example, membership in the local Morrey spaces or the local Kato class. Consequently, the analysis done here can be modified to handle these more general situations without changing the final results. 
\end{remark}

\medskip

\begin{theorem}\label{thm:converse3}
Let $N \geq 2$. Assume $V \in L^1_{loc}(\Omega)$, $V^+ \in L^{p}_{loc}(\Omega)$ for some $p>\frac{N}{2}$  \;($N \geq 2$) and $Q_V \succeq 0$ in $\Omega$. 
Then  there exists a non-negative Radon measure $\mu$ and an a.e. positive  function $u \in W_{loc}^{1,q}(\Omega)$ for any $1\leq q <\frac{N}{N-1}$ which is a distributional solution  of $-\Delta u + Vu = \mu$ in $\Omega$. In particular, ${\cal C}_V (\Omega)\neq \emptyset$.
\end{theorem}
\begin{proof}
Consider an exhaustion $\{\Omega_m\}_{m=1}^{\infty}$ of $\Omega$ by open bounded sets $\Omega_m$ with smooth boundary  such that $\Omega_m \Subset \Omega_{m+1} \Subset \Omega$.
Let $V_n^- = \min\{V^-,  n\}, n \in \mathbb N$. Define
\begin{equation}
\lambda_{n,m} := \inf_{\phi \in H^1_0(\Omega_m)\setminus \{0\}}
\frac{\int_{\Omega_m}   |\nabla \phi|^2 + (V^+- V_n^-) \phi^2}{\int_{\Omega_m}
\phi^2}.
\end{equation}
Using the assumption $Q_V \succeq 0$ in $\Omega$, we obtain that $
\lambda_{n,m}  \geq 0.$
Since $V^+  \in L^{\frac{N}{2}}(\Omega_m)$, there exists a non-negative function $u_{n,m} \in
H^1_0(\Omega_m)$ that achieves the above infimum. Clearly $u_{n,m} $ solves the following equation  in the $H^1-$ weak sense:
\begin{equation}\label{u_neqn4}
\left\{
\begin{array}{cl}
&   -\Delta u_{n,m}  + V^+ u_{n,m} =( V_n^- +\lambda_{n,m}) u_{n,m}  \;\hbox{ in } \Omega_m,\vspace{2mm}\\
&  u_{n,m} \geq 0\;\text{in\;} \Omega_m ;  \vspace{2mm}\\
& u_{n,m} = 0 \; \text{on\;} \partial\Omega_m. 
\end{array}
\right.
\end{equation}
Again by the assumption on $V^+,$ we note that  $u_{n,m}$ is continuous in $\overline{\Omega}_m$.
Fix a ball $B_0 \Subset  \Omega_1$. Let $x_{n,m} \in \overline{B}_0$ denote a point at which $u_{n,m}$ achieves 
its minimum in $\overline{B}_0$, that is:
\begin{equation}\label{eq:33.1}
u_{n,m}(x_{n,m}) \leq u_{n,m}(x) \text{ for all } x \in \overline{B}_0.
\end{equation}

Let $G_m$ be the (positive) Green's function for $-\Delta + V^+$ on $\Omega_m$. We now normalise $u_{n,m}$ as:
\begin{eqnarray}\label{normalise3}
w_{n,m} &:=& u_{n,m}/ \rho_{n,m}  \;\; \text{ where }\nonumber \\
\rho_{n,m}&:= &\int_{\Omega_m} ( V_n^-(y) + \lambda_{n,m}) u_{n,m}(y) G_m(x_{n,m},y)dy >0.
\end{eqnarray}
Define 
\begin{equation}\label{eq:38}
f_{n,m}:=  (V_n^- + \lambda_{n,m}) w_{n,m}.
\end{equation}

Clearly $w_{n,m} \in H^1_0(\Omega_m)$ satisfies in the $H^1-$weak sense:
\begin{equation}\label{Pin-Tin eqn3}
\left\{
\begin{array}{cl}
 &  -\Delta w_{n,m}  + V^+ w_{n,m} =f_{n,m} \;  \hbox{ in } \Omega_m,\; w_{n,m} \geq 0 \hbox{ in } \Omega_m,\vspace{2mm}\\
& \int_{\Omega_m} f_{n,m} (y) G_m(x_{n,m},y) dy = 1.
\end{array}
\right.
\end{equation}

From Green's function representation, 
$$ w_{n,m}(x) = \int_{\Omega_m} G_m(x,y) f_{n,m}(y) dy. $$
We note from the above expression and the normalisation in \eqref{Pin-Tin eqn3} that $w_{n,m}(x_{n,m})= 1$ for all $n,m$. Clearly, from \eqref{eq:33.1} and the above normalisation, 
\begin{equation}\label{eq:32}
1=w_{n,m}(x_{n,m}) \leq w_{n,m}(x) \text{ for all } x \in \overline{B}_0.
\end{equation}
Fix $m_0 \geq 1$.
We note that $G_{m} \geq G_{m_0}$ on $\Omega_{m_0}$ for all $m \geq m_0$ (proof similar to that of the claim in theorem \ref{AAP-distr}) and hence,
$$
\gamma(m_0) :=  \min_{x \in \overline{B}_0, y \in \Omega_{m_0}} G_{m_0+1} (x ,y) >0.
$$
Therefore, 
$$
 G_m(x_{n,m}, y)\geq \gamma(m_0)\;\;\; \forall  n \geq 1,m \geq  m_0+1, y \in  \Omega_{m_0}.
$$
Consequently, 
\begin{equation} \label{eq:37}
\forall n \geq 1, m \geq m_0+1: \;\; \gamma(m_0) \int_{\Omega_{m_0}} f_{n,m} \leq \int_{\Omega_{m_0}}f_{n,m}(y) G_m(x_{n,m},y)dy \leq 1.
\end{equation}
Thus, we get that,
\begin{equation}\label{eq:28}
\|f_{n,m} \|_{L^1(\Omega_{ m_0})} \leq \frac{1}{\gamma(m_0)} \;  \text{ for all } \; n \geq 1, m \geq m_0+1 \;\text{ and }\; m_0 \geq 1.
\end{equation}
Consider the very weak solution $\psi_{n,m}$ of the following problem for $n \geq 1, m \geq m_0+2$:
$$
\left\{
\begin{array}{cl}
&   -\Delta \psi_{n,m}  + V^+ \psi_{n,m} =f_{n,m} \;\hbox{ in } \Omega_{m_0+2} ,\vspace{2mm}\\
&  \psi_{n,m} \geq 0 \;\text{ in\;} \Omega_{m_0+2} ;  \vspace{2mm}\\
& \psi_{n,m} = 0 \; \text{ on\;} \partial \Omega_{m_0+2}. 
\end{array}
\right.
$$
In what follows $q$ will denote any number in the interval $[1,\frac{N}{N-1})$.
By the estimate in \eqref{eq:diwali} in proposition \ref{prop:train} and \eqref{eq:37}, 
\begin{equation}\label{eq:33}
\begin{array}{cl}
\|\psi_{n,m}\|_{W_0^{1,q}(\Omega_{m_0+2})}  &\leq C(q,m_0)  \|f_{n,m}\|_{L^1(\Omega_{m_0+2})} \\
& \leq \frac{C(q,m_0)}{\gamma(m_0+2)}, \; \forall n \geq 1, m \geq m_0+3.
\end{array}
\end{equation}
We note that $h_{n,m}:=w_{n,m}-\psi_{n,m}$ solves the following equation for all $n \geq 1, m \geq m_0+2$ :

$$
\left\{
\begin{array}{cl}
& -\Delta h_{n,m}  + V^+ h_{n,m} =0  \hbox{ in } \Omega_{m_0+2} ,\vspace{2mm}\\
& h_{n,m} \geq 0 \;\text{ on \;} \partial \Omega_{m_0+2} . \\
\end{array}
\right.
$$
By the weak comparison principle, $h_{n,m} \geq 0$ in $\Omega_{m_0+2}$. By  using Harnack's inequality  we obtain that 
\begin{eqnarray}\label{madh}
\sup_{\Omega_{m_0+1}} h_{n,m} &\leq& C( m_0)  \inf_{\Omega_{m_0+1}} h_{n,m} \leq  C( m_0) \inf_{ \overline{B}_0} h_{n,m}  \leq C( m_0) h_{n,m}(x_{n,m})\notag \\  
&\leq&  C( m_0) w_{n,m}(x_{n,m}) =  C( m_0), \; \forall n \geq 1, m \geq m_0+2.
\end{eqnarray}

Therefore, 
$\{ -\Delta h_{n,m}\}$ is  a bounded sequence in $L^p(\Omega_{m_0+1})$. Therefore, by standard elliptic regularity and \eqref{madh} 
we obtain  that (for different constants $C(m_0)$),
\begin{eqnarray}\label{gop}
\|h_{n,m}\|_{W^{2,p}(\Omega_{m_0})} &\leq& C(m_0) \Big ( \|h_{n,m}\|_{L^p(\Omega_{m_0+1})}+ \| V^+ h_{n,m}\|_{L^p(\Omega_{m_0+1})}  \Big) \notag \\
&\leq& C(m_0) \Big (1+ \| V^+\|_{L^p(\Omega_{m_0+1})} \Big)  \|h_{n,m}\|_{L^{\infty}(\Omega_{m_0+1})} \notag \\
&\leq& C(m_0) \Big (1+ \| V^+\|_{L^p(\Omega_{m_0+1})} \Big)  \;\; \forall n \geq 1, m \geq m_0+2.
\end{eqnarray}
From the above estimates \eqref{eq:33} and \eqref{gop}  we obtain
$$
 \|w_{n,m}\|_{W^{1,q}( \Omega_{m_0})} \leq C(q, m_0),  \;  \forall n \geq 1, m \geq m_0+3.   
$$
In what follows, we fix $r$ to be the H\"older conjugate of $p$, and we note that  $r \in [1,\frac{N}{N-2})$.
The last inequality together with Sobolev imbedding means that 
\begin{equation}\label{eq:41}
\big \{w_{n,m}: n \geq 1, m \geq m_0+3 \big \} \hbox{ is relatively compact in } L^r(\Omega_{m_0}), \; \forall  m_0 \geq 1.
\end{equation}

Consider the diagonal sequence $\{w_{n,n}\}$ which we will denote by $\{v_n\}$.
By \eqref{eq:41},  there exists a subsequence of $\{v_n\}$, which we will denote by $\{v_{n_j(1)}\}$ and a non-negative function $u_1 \in L^r(\Omega_1)$ such that
$$
v_{n_j(1)}  \to u_1 \hbox{ in } L^r(\Omega_1) \hbox{ as } j \to \infty.
$$
Considering now the sequence $\{v_{n_j(1)}\}$ and again applying \eqref{eq:41} to it, we obtain a subsequence of $\{v_{n_j(1)}\}$, which we will denote by $\{v_{n_j(2)}\}$, and a non-negative function $u_2 \in L^r(\Omega_2)$ such that
$$
v_{n_j(2)}  \to u_2 \hbox{ in } L^r(\Omega_2) \hbox{ as } j \to \infty.
$$
It is easy to see that $u_1 \equiv u_2$ in $\Omega_1$. Proceeding inductively, we obtain subsequences $\{v_{n_j(m)}\}$ and non-negative functions $u_m \in L^r(\Omega_m)$ such that
\begin{equation}\label{eq:29.2}
\left\{
\begin{array}{l}
   \{v_{n_j(m+1)}\}_j \subset \{v_{n_j(m)}\}_j \hbox{ for all } m; \\
   v_{n_j(m)} \to u_m \hbox{ in } L^r(\Omega_m) \hbox{ as } j \to \infty, \; \hbox{ for each }\;  m \in \mathbb N ;\\
    u_m \equiv u_{m+1} \hbox{ in } \Omega_m.
\end{array}
\right.
\end{equation}
Define
\begin{equation}\label{eq:29.3}
u \in L^r_{loc}(\Omega) \hbox{ as: } u \equiv u_m \hbox{ on } \Omega_m, m=1,2,3,\cdot \cdot \cdot.
\end{equation}
We note that each $v_{n_j(m)}$  solves :
\begin{equation}\label{eq:29.4}
\int_\Omega v_{n_j(m)}(-\Delta \xi + V^+ \xi) = \int_\Omega (V_{n_j(m)}^- +\lambda_{n_j(m),n_j(m)})v_{n_j(m)}\xi , \; \forall \xi \in C_c^{\infty}(\Omega_m).
\end{equation}
Take any $\xi \in C_c^{\infty}(\Omega)$, $\xi \geq 0$. Choose $m$ such that support of $\xi$ is contained in $\Omega_m$. Letting $j \to \infty$ in the above equation with such a $\xi$, using \eqref{eq:29.2}-\eqref{eq:29.3}, the assumption $V^+ \in L^{p}_{loc}(\Omega)$ and Fatou's lemma on the right hand side, we get,
$$ 
\int_\Omega u(-\Delta \xi + V^+ \xi) \geq \int_\Omega V^- u \xi .
$$
Since $\xi \geq 0$ was arbitrary, we obtain that  $-\Delta u + Vu \geq 0$ in the sense of distributions. Define the non-negative Radon measure $\mu:= -\Delta u + Vu $.
Then, we get that $u$ is a distribution solution of $-\Delta u + Vu = \mu$ in $\Omega$.  From \eqref{eq:29.3} and \eqref{eq:32} we get that $u$ is non-negative in $\Omega$ and $\inf_{B_0} u \geq 1$. Appealing to  theorem \ref{thm:StrongMax} we obtain that $u>0$ a.e. in $\Omega$. That $u \in W^{1,q}_{loc}(\Omega)$ for $q \in [1, \frac{N}{N-1})$ follows from proposition \ref{prop:RanRan}.
\end{proof}

\begin{theorem}[AAP Principle in ${\cal D^\prime}(\Omega)$]\label{thm:chitra}
Assume $V \in L^1_{loc}(\Omega)$. 
\vspace{-2mm}
\begin{enumerate}
\item[{\bf (i)}]
If ${\cal C}_V (\Omega) \neq \emptyset$ then $Q_V \succeq 0$ in $\Omega$. 
\item[{\bf (ii)}]
Conversely, if  $N \geq 2$, $V^+ \in L^{p}_{loc}(\Omega)$ for some $p>\frac{N}{2}$ and $Q_V\succeq 0$ in $\Omega$, then   ${\cal C}_V (\Omega) \neq \emptyset$.
\item[{\bf (iii)}]
When  $N=1$, we have  $Q_V \succeq 0$ in $\Omega$ iff  ${\cal C}_V (\Omega) \neq \emptyset$.
\end{enumerate}
\end{theorem}
\begin{proof}
The case $N \geq 2$ follows from theorems \ref{AAP-distr} and \ref{thm:converse3}. When $N=1$ we refer to theorem 3.1 in \cite{GeZhao}.
\end{proof}
\medskip
\section{AAP principle  for $L^1$-subcritical operators with balanced potential}\label{s8}
Our assumptions on $V$ are rather weak and  may prevent the existence of a Green's function for the operator $-\Delta +V$ in $\Omega$. Nevertheless, if the operator is $L^1$-subcritical, within the energy space ${\cal H}_V(\Omega)$ it continues to retain the unique solvability and monotonicity properties that are usually implied by the existence of a nonnegative Green's function. 

\begin{lemma}\label{lem:jug}
Let $V \in L^1_{loc}(\Omega)$ and $Q_V \succeq 0$ in $\Omega$.  Then, $-\Delta + V$ is $L^1$-subcritical in $\Omega$
if and only if there exists a (unique) energy solution (see definition \ref{def:EnergySolution}) $u \in {\cal H}_V(\Omega) $  solving $-\Delta u + V u =f$ for any $f \in  L^{\infty}_{c} (\Omega)$. Indeed, if $f$ is nonnegative, then so  is ${\bf J}(u)$.\end{lemma}
\begin{proof}
Let $-\Delta + V$ be $L^1$-subcritical, i.e., ${\cal L}_V(\Omega) \hookrightarrow L^1_{loc} (\Omega)$. Take any $f \in  L^{\infty}_{c} (\Omega)$. 
Define the continuous linear map
$$
   \Phi_f:
    L^1_{loc} (\Omega) \to \mathbb R
    \quad \hbox{ by } \quad
    \Phi_f(h):= \int_{\Omega} f h.
$$ 
Then $\Phi_f$ is continuous on 
${\cal L}_V(\Omega)$, and can therefore be uniquely extended to a continuous linear map 
on ${\cal H}_V(\Omega)$. Using Riesz representation theorem,  we can find a unique 
$u \in {\cal H}_V(\Omega) $ satisfying
\begin{equation} \label{eq:Casar}
    a_{V,\Omega}(u,h) =  \Phi_f(h), 
    \quad
    \forall h \in W^{1,\infty}_c(\Omega).
\end{equation}
Conversely, assume that the operator equation is  solvable in  ${\cal H}_V(\Omega)$  (in the energy sense)  for each given $f \in  L^{\infty}_{c} (\Omega)$. By density of  $W^{1,\infty}_c(\Omega)$ in  ${\cal H}_V(\Omega)$ such a solution is unique. Suppose that ${\cal L}_V(\Omega) \not \hookrightarrow  L^1_{loc}(\Omega)$. Then, we may find a ball $B \Subset \Omega$ and a sequence $\{\xi_n \} \subset W^{1,\infty}_c(\Omega)$ such that $\xi_n \to 0$ in  ${\cal H}_V(\Omega)$, but $\inf_n \int_B |\xi_n| > 0$. By considering $|\xi_n|$ in place of $\xi_n$ and noting that $Q_V(\xi_n)=Q_V(|\xi_n|)$, 
we may assume without loss of generality that $\xi_n \geq 0$. Let $u \in {\cal H}_V(\Omega)$ be  the energy solution of $-\Delta u+ V u = \chi_B$.
It then follows that $ \int_B \xi_n = a_{V,\Omega}(u,\xi_n) \to 0 \hbox{ as } n \to \infty.$
This gives a contradiction, thereby showing the imbedding ${\cal L}_V(\Omega)  \hookrightarrow  L^1_{loc}(\Omega)$.

The assertion that ${\bf J}(u) \geq 0$ whenever $f \geq  0$ follows from lemma  \ref{lem:BADAMILK2}.
\end{proof}

\medskip

The following corollary shows that $L^1-$subcritical operators with balanced potentials  possess a hidden nonnegative Green's function.
\begin{corollary}\label{pomo}
Let $V \in L^1_{loc}(\Omega)$ be such that $Q_V \succeq 0$ in $\Omega$.  Then, 
\begin{enumerate}
\vspace{-2mm}
\item[{\rm \bf (i)}]
If $V$ is balanced in $\Omega$ and  $-\Delta + V$ is $L^1$-subcritical in $\Omega$, then
 given any $f \in  L^{\infty}_{c} (\Omega)$, there exists a unique distributional solution  $u_* \in L^1_{loc}(\Omega) $ obtained as $u_*={\bf J}(u)$ for some $u \in {\cal H}_V(\Omega)$,  solving $-\Delta u_* + V u_* =f$. Indeed, if $f$ is nonnegative, then so is the corresponding solution $u_*$. 
 \item[{\rm \bf (ii)}]
Conversely, if for any   compact subset  $K$ of $\Omega$ with positive measure,  the problem $-\Delta u + V u =\chi_K$ admits a nonnegative distributional solution $u$, then $-\Delta + V$ is $L^1$-subcritical in $\Omega$.
\end{enumerate}
\end{corollary}
\proof  (i) Follows from the above lemma and proposition  \ref{prop:EnergyVsDist}. 

\medskip

(ii) We take any such compact $K$ and the coresponding solution $u$. From  theorem \ref{thm:StrongMax} we have $u>0$ a.e. in $\Omega$ and from theorem \ref{AAP-distr} we obtain that $Q_V \succeq w:=\chi_K/u$ in $\Omega$. Therefore,
$$
\Big( \int_K 1/w \Big)^{\frac{1}{2}}  \sqrt{Q_V(\xi)} \geq  \Big( \int_K 1/w \Big)^{\frac{1}{2}} \Big (\int_K w |\xi|^2 \Big )^{\frac{1}{2}} \geq  \int_K |\xi|, \quad \forall \xi \in W_c^{1,\infty}(\Omega).
$$
 \qed

\begin{remark}\label{11k}
From the above results, it is clear that  if $V$ is balanced and $-\Delta + V$ is $L^1-$subcritical in $\Omega$, there exists no nonzero $u \in {\cal H}_V(\Omega)$ such that  $u_*={\bf J}(u)$ solves in the distributional sense $-\Delta u_* + Vu_*=0$ in $\Omega$. In other words, in this situation, the only distributional solution $u_* \in {\bf J}({\cal H}_V(\Omega))$ to $-\Delta u + V u=0$ in $\Omega$  is the trivial solution.
\end{remark}

\begin{corollary}\label{hasn}
Let  $V \in L^1_{loc}(\Omega)$ and $Q_V \succeq 0$ in $\Omega$. Suppose for some $\epsilon_0 $ chosen to be positive if $-\Delta+V$ is $L^1$-critical in $\Omega$ and zero otherwise, there exists   $f_0 \in L^\infty_c(\Omega)$ such  that the problem $-\Delta u + (V+\epsilon_0)u =f_0$ has no distribution solution in $\Omega$. Then $V$ is not a balanced potential in $\Omega$.
\end{corollary}
\proof Suppose $V$ is balanced in $\Omega$. Then by proposition \ref{mea}(ii), $V+\epsilon_0$ is balanced in $\Omega$. We note that the operator $-\Delta + V+\epsilon_0$  is $L^1$-subcritical in $\Omega$ and apply corollary \ref{pomo}   to conclude that  for any $f \in L^\infty_c(\Omega)$ we obtain a  distributional solution of $-\Delta u + (V+\epsilon_0) u =f$ in $\Omega$. 
\qed
\begin{corollary}\label{husn}
Let  $V \in L^1_{loc}(\Omega)$ and $Q_V \succeq 0$ in $\Omega$. Suppose for some $\epsilon_0 $ chosen to be positive if $-\Delta+V$ is $L^1$-critical in $\Omega$ and zero otherwise, we have ${\mathcal C}_{V+\epsilon_0}(\Omega) = \emptyset$ (see definition \ref{infinitea}). Then, $V$ is not balanced in $\Omega$.
\end{corollary}
\proof Follows by the contrapositive argument as above and noting that if $f \in L^\infty_c(\Omega)$, $f \not \equiv 0$ is chosen nonnegative, the corresponding distributional solution will be positive a.e. in $\Omega$  from corollary \ref{pomo} and theorem \ref{thm:StrongMax}. \qed

\medskip

But, we can show that the  above results imply the existence of a positive distributional solution to the equation $-\Delta u +Vu=0$ on any ``strict" sub-domain of $\Omega$.

\begin{proposition}\label{cor:run}
Assume  $V \in L^1_{loc}(\Omega), V$ is balanced in $\Omega$ and $-\Delta+V$ is  $L^1-$subcritical  in $\Omega$. Then, given any compact set $K \Subset \Omega$ with positive measure, there exists an a.e. positive distribution solution $u$ satisfying $(-\Delta+V)u = 0 $ in $\Omega \setminus K$ .
\end{proposition}
\begin{proof}
 Since $-\Delta+V$ is $L^1$-subcritical, by corollary \ref{pomo} we can solve $(-\Delta+V) u = \chi_K$ in ${\cal D}^\prime(\Omega)$  with $u \geq 0$. From  proposition  \ref{prop:RanRan} we can find a quasicontinuous representative of $u$, which solves the same PDE as $u$ in the distributional sense, and from theorem \ref{thm:StrongMax}, we get $u>0$ a.e. in $\Omega$.
\end{proof}
\begin{corollary}\label{bindu}
Assume  $V \in L^1_{loc}(\Omega), V$ is balanced in $\Omega$ and $-\Delta+V$ is  $L^1-$subcritical  in $\Omega$. Then, ${\cal C}_V(\Omega)$ contains an infinite set of linearly independent functions.
\end{corollary}
\begin{proposition}\label{roger}
Assume that  $V \in L^1_{loc}(\Omega)$ is a balanced potential  and $-\Delta+V$ is  $L^1-$ subcritical  in $\Omega$. 
Suppose additionally that 
for some  compact submanifold $\Sigma$ contained in $\Omega$, the restriction map
$$
    {\mathcal L}_V(\Omega) \to L^{1}(\Sigma),
    \quad
    \xi \mapsto \xi |_{\Sigma}
$$ 
is continuous.
 Then there exists an a.e. positive   distributional solution $u \in L^1_{loc}(\Omega)$ solving $-\Delta u +Vu=0$ in $\Omega \setminus \Sigma$.
\end{proposition}
\begin{proof}
 By Riesz representation  argument, for any $h \in L^{\infty}(\Sigma)$, $h\geq 0, h \not \equiv 0$, we obtain an Energy solution $u_* \in {\mathcal H}_V(\Omega) $ such that
$$
a_V(u_*, \xi)=\int_{\Sigma }h \xi d \nu, \quad \forall \xi \in W^{1,\infty}_c(\Omega).
$$
For simplicity we let $h \equiv 1$ on $\Sigma$. We can then proceed as in proposition \ref{prop:EnergyVsDist} to show that $u:={\bf J(u_*)}$ is a nonnegative  distribution solution of $-\Delta u +Vu= \nu $ in $\Omega$.  Therefore, $u \not \equiv 0$ and $-\Delta u +Vu= 0$ in $\Omega \setminus \Sigma$.
\end{proof}

\begin{remark}\label{hungry}
The main point of the above results is that Harnack inequality may fail for the operator  $-\Delta + V$. 
It seems natural to push  the  arguments in proposition \ref{cor:run} further by taking a sequence $\{K_n\}$ of compact sets approaching  the boundary (or point at infinity) of $\Omega$ and show that the corresponding sequence of solutions $\{u_n\}$ converges (up to a subsequence) locally to an a.e. positive function $u$ solving $(-\Delta+V)u=0$ in the distributional sense in the full domain $\Omega$.  Such an argument seems to require Harnack-type inequality for its success (see  \cite{AS}, and  \cite{PT} for further generalisations). 

 \end{remark}
 
 This motivates the following 
 
 {\bf Open problem :}  Assume  $V \in L^1_{loc}(\Omega), V$ is balanced and $-\Delta+V$ is $L^1-$subcritical  in $\Omega$. Does $0 \in {\cal M}_V(\Omega)$ ?
 
\begin{corollary}\label{cor:bun}
Assume  $V \in L^1_{loc}(\Omega), V$ is balanced and $-\Delta+V$ is $L^1-$subcritical  in $\Omega$. Then, there exists an open dense subset $U$ of $\Omega$ and  an a.e. positive distribution solution $u$ satisfying $(-\Delta+V)u = 0 $ in $U$.
\end{corollary}
\begin{proof}
Take a compact $K \subset \Omega$ to be any ``Cantor-like" set with positive measure and empty interior. Then $U:= \Omega \setminus K$ is an open dense subset of $\Omega$ and then we apply proposition \ref{cor:run}.
\end{proof}

\begin{proposition}[AAP Principle in ${\cal H}_V(\Omega)$ for $L^1-$subcritical operator] \label{vicol} 
Let $V \in L^1_{loc}(\Omega)$ be such that $Q_V \succeq 0$ in $\Omega$.  
Then  $-\Delta + V$ is $L^1$-subcritical in $\Omega$ iff there exists $u \in {\cal H}_V(\Omega)$ and $f \in  L^1_{loc} (\Omega)$ with $\inf_B f > 0$ for any ball $B \Subset  \Omega$ such that $-\Delta u+ Vu =f$ in the energy sense in $\Omega$. 
\end{proposition}
\begin{proof}
Let $-\Delta + V$ be $L^1$-subcritical.  
Consider  an exhaustion $\{\Omega_n\}$ of $\Omega$ by open sets such that
$$
    \Omega_n \Subset \Omega, 
    \quad
    \overline{\Omega}_{n} \subset \Omega_{n+1}.
   $$  
Set $K_1 := \overline{\Omega}_1$ and $K_{n+1} := \overline{\Omega}_{n+1} \setminus \Omega_n$ (for $n \geq 1$).
Note that given a compact  set $K \subset \Omega$,  
\begin{equation} \label{eq:FinIntersection1}
   \{ n \in \mathbb N \, : \, K \cap K_n \not = \emptyset \} \hbox{ is a finite set. }
\end{equation}
By lemma  \ref{lem:jug}, there exists  $u_n \in {\cal H}_V(\Omega),  u_n \neq 0$ such that
$$  
     a_{V,\Omega}(u_n ,\xi)= \int_\Omega \chi_{K_n} \xi,
     \quad \;\forall \xi \in W_c^{1, \infty}(\Omega).
       $$

 Setting
$$
   \rho_n :=  \sum_{j=1}^n \frac{1}{j^2} \frac{u_j}{ \|  u_j \|_{V,\Omega}} ,
    \quad
    f_n :=   \sum_{j=1}^n \frac{1}{j^2} \frac{ \chi_{K_j} }{\| u_j \|_{V,\Omega}} ,
$$
we have also
$$
  a_{V,\Omega}( \rho_n , \xi)   = \int_{\Omega} f_n \xi,
    \quad
    \forall \xi \in W_c^{1,\infty} (\Omega) .
$$
We easily see that 

$$\rho_n \to u:= \sum_{j=1}^\infty \frac{1}{j^2} \frac{u_j}{ \|  u_j \|_{V,\Omega}} \hbox{ in } {\cal H}_V(\Omega).$$

Concerning the sequence $\{f_n\}$, from the observation~\eqref{eq:FinIntersection1} we see that
$$
     \sum_{j=1}^n \frac{1}{j^2} \frac{ \chi_{K_j} }{\| u_j \|_{V}} \to 
    f: = \sum_{j=1}^{\infty} \frac{1}{j^2} \frac{ \chi_{K_j} }{\| u_j \|_{V}} 
    \quad
    \hbox{ in } L^1_{loc} (\Omega) .
$$ 
Note that  $\inf_B f>0$ for any ball $B \Subset \Omega$. Therefore, for each $\xi \in W_c^{1,\infty} (\Omega)$ we have $a_{V,\Omega}(\rho_n, \xi) \to a_{V,\Omega}(u,\xi)$ and 
$\int_{\Omega} f_n \xi \to \int_{\Omega} f \xi$. Hence, we get in the limit: 
$$
   a_{V,\Omega}(u, \xi) = \int_\Omega f \xi \quad \forall \xi \in W_c^{1,\infty} (\Omega).
$$

\medskip

Conversely, let $u \in {\cal H}_V(\Omega)$ be  an energy solution of $(-\Delta+V)u=f$  for some  $f\in L^1_{loc}(\Omega)$ with $ \inf_B f>0$ for any ball $B \Subset \Omega$ . Suppose that ${\cal L}_V(\Omega) \not \hookrightarrow  L^1_{loc}(\Omega)$. Then, 
we may find a ball $B \Subset \Omega$ and a sequence $\{\xi_n \} \subset W^{1,\infty}_c(\Omega)$ such that $\xi_n \to 0$ in  ${\cal H}_V(\Omega)$, but $\inf_n \int_B |\xi_n| > 0$. By considering $|\xi_n|$ in place of $\xi_n$ and  noting that $Q_V(\xi_n)=Q_V(|\xi_n|)$, we may assume without loss of generality that $\xi_n \geq 0$. It then follows that 
$$\inf_n \int_\Omega f \xi_n \geq (\inf_B f ) \inf_n \int_B \xi_n>0 \;\hbox{ and }\; a_{V,\Omega}(u,\xi_n) \to 0 \;\hbox{ as }\; n \to \infty.$$
This gives a contradiction to the assumption that $u$ is the energy solution of $-\Delta u+ V u =f$, thereby showing the imbedding ${\cal L}_V(\Omega)  \hookrightarrow  L^1_{loc}(\Omega)$.
\end{proof}

\begin{theorem}[AAP Principle in ${\cal D}^{\prime}(\Omega)$ for $L^1-$subcritical operator] \label{prop:Sona} 
Let $V \in L^1_{loc}(\Omega)$ be such that $Q_V \succeq 0$ in $\Omega$. 
\begin{enumerate}
\item[{\rm \bf (i)}]
Let $V$ be balanced  in $\Omega$.  If  $-\Delta+V$ is $L^1$-subcritical in $\Omega$, then 
 there exists an a.e. positive function $u_{*} \in L^1_{loc}(\Omega)$ (obtained as  $u_*={\bf J}(u)$ for some $u \in  {\cal H}_V(\Omega)$) and $f \in  L^1_{loc} (\Omega)$ with $\inf_B f > 0$ for any ball $B \Subset  \Omega$, solving  $-\Delta u_* + V u_* =f$ in the sense of distributions in $\Omega$.
 \item[{\rm \bf (ii)}]
 Conversely, if there exists a nonnegative function $u \in L^1_{loc}(\Omega)$  and $f \in  L^1_{loc} (\Omega)$ with $\inf_B f > 0$ for any ball $B \Subset  \Omega$, solving  $-\Delta u + V u =f$ in the sense of distributions in $\Omega$, then $-\Delta+V$ is $L^1$-subcritical in $\Omega$.
\end{enumerate}
\end{theorem}
\proof 
{\bf (i)} 
follows from proposition \ref{vicol},  proposition \ref{prop:EnergyVsDist}, lemma \ref{lem:BADAMILK2} and theorem \ref{thm:StrongMax}.

\medskip

{\bf (ii)} 
follows in the same way as the proof of corollary \ref{pomo} (ii) by noting that $w:= f/u$ satisfies $1/w \in L^1_{loc}(\Omega)$.  \qed

\section{Equivalence of $L^1$ and $L^2$-subcriticality  notions} \label{s9}
\subsection{Equivalence of $L^1$- and global $L^2$-subcriticality when $V$ is balanced} \label{s9.1}
When $V$ is balanced, it is a rather straightforward affair to prove the equivalence of the notions of $L^1$- and global $L^2$- subcriticalities as the following result shows: 
\begin{proposition}\label{basa}
Assume $V \in L^1_{loc}(\Omega)$ is such that $Q_V \succeq 0$ in $\Omega$. 
\vspace{-2mm}
\begin{enumerate}
\item[{\rm \bf (i)}] 
Let $V$ be balanced in $\Omega$. Then $-\Delta + V$   is  globally $L^2$-subcritical in $\Omega$ if it is $L^1$-subcritical in $\Omega$.
\item[{\rm \bf (ii)}] 
$-\Delta + V$   is $L^1$-subcritical in $\Omega$ if  it is globally $L^2$-subcritical in $\Omega$.
\end{enumerate}
\end{proposition}
\proof
{\bf (i)} 
By theorem \ref{prop:Sona}, there exists  $u_* \in L^1_{loc}(\Omega) $ obtained as $u_*={\bf J}(u)$ for some $u \in {\cal H}_V(\Omega)$ and   $f \in L^1_{loc}(\Omega)$ with $\inf_B f>0$ for any ball $B \Subset \Omega$,  solving $-\Delta u_* + V u_* =f$ in the distribution sense. Indeed, $u_{*} > 0$ a.e.. By theorem \ref{AAP-distr}, we get $Q_V \succeq f/u_*$ in $\Omega$. This in particular shows that $w:= f/u_* \in L^1_{loc}(\Omega)$. It is easy to see that $1/w \in  L^1_{loc}(\Omega)$.  Therefore,   $-\Delta + V$ is globally $L^2$-subcritical.

\medskip

{\bf (ii)} 
Conversely, let $-\Delta + V$ be globally $L^2$-subcritical with the weight $w$. Then given any compact set $K \subset \Omega$, by an application of H\"older's inequality,
$$
\Big( \int_K 1/w \Big)^{\frac{1}{2}}  \sqrt{Q_V(\xi)} \geq  \Big( \int_K 1/w \Big)^{\frac{1}{2}} \Big (\int_K w |\xi|^2 \Big )^{\frac{1}{2}} \geq  \int_K |\xi|, \quad \forall \xi \in W_c^{1,\infty}(\Omega).
$$
This shows that $-\Delta + V$ is $L^1$-subcritical. \qed
\subsection{Equivalence of $L^1$ and  feeble $L^2$-subcriticality when $V$ is tame in $\Omega$ } \label{s9.2}
In this subsection we assume $V^+ \in L^p_{loc}(\Omega)$ for some $p>\frac{N}{2}$  \;($N \geq 2$). Mainly, we show that the validity of the imbedding ${\cal L}_V(\Omega) \hookrightarrow L^1_{loc}(\Omega)$ is determined (somewhat surprisingly)  by whether or not ${\cal L}_V(\Omega) \hookrightarrow L^1(K)$ for some  compact set $K \subset \Omega$ with positive measure. From this, we deduce the equivalence between the notions of $L^1$- and  feeble $L^2$- subcriticality.

\begin{proposition} \label{prop:AntiMaximum}
Let $\Omega_*$ be a bounded domain with smooth boundary and let $K_1$, $K_2$ be two disjoint  compact sets with positive measure contained in $\Omega_*$. 
Assume that $V^+ \in L^p(\Omega_*)$ for some $p>\frac{N}{2}$  \;($N \geq 2$). Then,  the unique weak solution $\psi_\alpha$ to the problem
\begin{equation}\label{eq:GowdaBis}
     -\Delta u + V^+ u =  \alpha \chi_{K_1}  - \chi_{K_2} , 
     \quad
     u \in H^1_0 (\Omega_*),\; \; \alpha \in \R,
\end{equation}
is  in the space $W^{2,p}(\Omega_*)$ for all $\alpha$ and is positive a.e. in $\Omega$ for all $\alpha >0$ large enough. In particular, $\psi_\alpha  \in C_0(\overline{\Omega}_*)$ and satisfies the equation \eqref{eq:GowdaBis} pointwise a.e. in $\Omega_*$.
\end{proposition}
\proof 
Let $G, G_0$ be the Green's functions of the operators $- \Delta + V^{+}$ and $-\Delta$ respectively  in $\Omega_*$.  It is well known that there exists a constant $c >0$ such that 
\begin{equation}
      c^{-1} \; G_0(x,y) \leq G(x,y) \leq c \; G_0(x,y) \text{ for all } x,y  \in \Omega_*. \notag
\end{equation} The solution 
$\psi_{\alpha}$ to~\eqref{eq:GowdaBis}
can be represented as follows
\begin{equation} 
      \psi_{\alpha} (x) = \alpha \int_{K_1} G(x,y) dy - \int_{K_2} G(x,y) dy . \notag
\end{equation}

Hence, we deduce 
\begin{equation}
\Big( \frac{\alpha}{c} \Big) \int_{K_1} G_0(x,y) d y
    -  c \int_{K_2} G_0(x,y) d y \leq     \psi_{\alpha} (x) \leq    c \; \alpha \int_{K_1} G_0(x,y) d y
         . \notag
\end{equation}

Therefore,
\begin{equation} \label{divya}
\Big( \frac{\alpha}{c} \Big) u_1(x)
    -  c  \; u_2(x)  \leq     \psi_{\alpha} (x) \leq  \alpha c \; u_1(x)
         \text{ in } \Omega_* 
\end{equation}

where $u_1, u_2 \in H^1_0(\Omega_*)$ solve the following problems:
$$
-\Delta u_1 =  \chi_{K_1} \quad  \text{ and } \quad -\Delta u_2 = -\chi_{K_2}.
$$
We note that  $u_1, u_2 \in C^1(\overline{\Omega}_*)$ and that by Hopf lemma, $\frac{\partial u_1}{\partial \nu}<0$ on $\partial \Omega_*$. Therefore,  we obtain from \eqref{divya} that $\psi_\alpha$ is bounded  in $\Omega_*$  for all $\alpha$ and can be made positive by choosing $\alpha>0$ large enough. By elliptic regularity, $\psi_\alpha \in W^{2,p}(\Omega_*) \subset  C(\overline{\Omega}_*)$ for all $\alpha$.
\qed
\begin{lemma} \label{lem:H2Extension}
Assume that  $p>\frac{N}{2}\;(N \geq 2)$ and $\Omega_* \Subset \Omega$ be a bounded domain with  smooth boundary,  
$u \in W^{2,p} (\Omega_*) \cap H_0^1 (\Omega_*) \cap C_0(\overline{ \Omega}_*)$ be a nonnegative function. Consider the extension
$$
     \tilde u  = 
     \left\{
     \begin{array} {ll}
        u   &\hbox{ in } \Omega_*, \\
        0       &\hbox{ in } \Omega \setminus \Omega_* .
     \end{array}
     \right.
$$
Then, $\gamma (\tilde u)  \in W^{2,1}_c (\Omega)$ for any function  $\gamma \in C^2 (\mathbb R)$ with $\gamma(t) = 0$ for all $t \leq 0$. 
\end{lemma}
\proof
 Note that  $\gamma(u), \gamma^{\;\prime}(u) \in H^1_0(\Omega_*)$ and hence $\gamma(\tilde{u}) \in H^1_c(\Omega)$. Fix $1 \leq i,j \leq N$. Choose now a sequence $\{\phi_n\} \subset C^{\infty}(\overline{\Omega}_*)$ such that
$$
\phi_n \to u \text{ in } W^{2,p}(\Omega_*).
$$
Since $W^{2,p}(\Omega_*) \hookrightarrow H^1(\Omega_*)$ and $W^{2,p}(\Omega_*) \hookrightarrow  C(\overline{\Omega}_*)$, we additionally  have 
$$
\phi_n \to u \text{ in } H^1(\Omega_*) \cap C(\overline{\Omega}_*).
$$
It is a simple calculation to check that 
\begin{eqnarray}\label{mum}
\gamma(\phi_n)  &\to & \gamma(u)  \text{ in }  H^1(\Omega_*) \cap C(\overline{\Omega}_*), \notag \\
\gamma^{\;\prime }(\phi_n)  &\to & \gamma^{\;\prime}(u)  \text{ in }  H^1(\Omega_*) \cap C(\overline{\Omega}_*) \text{ and } \notag \\
\gamma^{\;\prime}(\phi_n) \partial_i \phi_n  &\to&  \gamma^{\;\prime}(u) \partial_i u \text{ in }  W^{1,1}(\Omega_*). 
\end{eqnarray}
From  the facts that $\phi_n \to 0$ uniformly on $\partial \Omega_*$ and $\partial_i\phi_n \to \partial_i u$ in $W^{1,p}(\Omega_*)$ we obtain using the Trace embedding that
$$
\int_{\partial \Omega_*} |\gamma^{\;\prime}(\phi_n) \partial_i \phi_n | \leq \max_{\partial \Omega_*}|\gamma^{\;\prime}(\phi_n) | \int_{\partial \Omega_*} |\partial_i \phi_n | = o_n(1).
$$
From \eqref{mum}, the above estimate and the Trace embedding, we obtain that $\gamma^{\;\prime}(u) \partial_i u =0$ on $\partial \Omega_*$. That is,
$$ \gamma(u) \in H^1_0(\Omega) \text{ and } \partial_i\gamma(u)= \gamma^{\;\prime} (u) \partial_i u \in W_0^{1,1}(\Omega_*).$$  
Therefore, for  a test function $\xi \in C_c^{\infty} (\Omega)$, 
 $$
    \int_{\Omega} \gamma (\tilde u) \, \partial_{ij} \xi =  \int_{\Omega_*} \gamma (u) \, \partial_{ij} \xi
    = - \int_{\Omega_*} \gamma^{\;\prime} (u) \partial_i u \, \partial_j \xi 
    = \int_{\Omega_*}  \partial_j (\gamma^{\;\prime} (u) \partial_i u) \, \xi .
$$
 This shows that the second distributional derivative of   $\gamma (\tilde u)$ is clearly an $L^1(\Omega)$ function with support in $\Omega_*$. \qed

\begin{proposition}
Assume $V^+ \in L^p_{loc}(\Omega)$ for some $p>\frac{N}{2}$  \;($N \geq 2$). Let $K_1, K_2$ be two disjoint  compact subsets of $\Omega$. Then, there exists 
a nonnegative function $ \Psi \in W^{2,1}_c(\Omega) \cap H_c^1 (\Omega) \cap C_c(\Omega)$   and a constant $C>0$ such that the following inequality is satisfied pointwise a.e. in $\Omega$:
$$
    -\Delta \Psi +V^+\Psi \leq C \chi_{K_1} - \chi_{K_2} .
$$
\end{proposition}
\proof
Choose a bounded domain $\Omega_* \Subset \Omega$ with smooth boundary such that $K_1$, $K_2$ 
are compactly contained in $\Omega_*$. 
Applying proposition~\ref{prop:AntiMaximum}, we choose a positive function
$\psi := \psi_{\alpha} \in W^{2,p}(\Omega_*) \cap H^1_0 (\Omega_*) \cap C_0(\overline{\Omega}_*)$ solving problem~\eqref{eq:GowdaBis}, and extend it by zero on 
$\Omega \setminus \Omega_*$ (while still referring to this extension as $\psi$). 

Let $\gamma(s) = \chi_{(0,\infty)} s^3$. Then $\gamma \in C^2 (\mathbb R)$ and 
$$
   \gamma\;'' \geq 0 \text{ for all } s \in \R;
   \quad
   \gamma\;'  (s) >0 \hbox{ and } \gamma(s) - \gamma^{\;\prime}(s)s < 0 \hbox{ for } s > 0.
$$
Define $ \overline{\psi} := \gamma(\psi)$. Clearly $\overline \psi \geq 0$ and by Lemma~\ref{lem:H2Extension},   
$\overline \psi \in W^{2,1}_c (\Omega)$. Hence, for a.e. in $\Omega_*$,
\begin{eqnarray}   
  - \Delta \overline{\psi} +  V^+ \overline{\psi}   &= &
  - \gamma^{\;\prime\prime}(\psi)|\nabla \psi |^2 - \gamma^{\; \prime}(\psi ) \Delta \psi+ V^+ \gamma(\psi) \notag \\  
& = &- \gamma^{\; \prime\prime}(\psi)|\nabla \psi|^2 + 
  \gamma^{\; \prime}(\psi) \big(- \Delta \psi + V^+ \psi \big) 
  + V^+ \big(\gamma(\psi)-\gamma^{\;\prime} (\psi)\psi \big) \notag \\
  &\leq&  \gamma^{\;\prime}(\psi) \big(- \Delta \psi+ V^+ \psi \big). \label{kej}
\end{eqnarray}

Since $\psi$ is continuous in $\overline{ \Omega}_*$ and positive in $\Omega_*$, we have for some positive constant $c$,
$$
    \gamma \;'(\psi) \leq c \, \hbox{ in } K_1
    \quad \hbox{ and }
    \gamma\;' (\psi)  \geq \, \frac{1}{c} \hbox{ in } K_2.
$$
Since $\psi$ solves~\eqref{eq:GowdaBis}, we obtain from \eqref{kej},
$$
     - \Delta \overline{\psi} +  V^+ \;\overline{\psi}   \leq \alpha c \; \chi_ {K_1}- \Big(\frac{1}{c}\Big) \chi_ {K_2} \text{ a.e. in } \Omega_*.
$$
Hence, $\Psi:=c \; \overline{\psi}$ is the required function. 
\qed

\medskip

\begin{lemma} [Non-Vanishing of null sequences]\label{prop:Vanishing} 
Assume  $V^{+}  \in L^p_{loc}(\Omega)$ for some $p>\frac{N}{2}$  \;($N \geq 2$), $V^- \in L^1_{loc}(\Omega)$ and $Q_V (\Omega) \succeq 0$. Then, for any sequence  $\{\xi_n \} \subset W_c^{1,\infty} (\Omega)$ such that  $Q_V (\xi_n) \to 0$ we have the following alternative
\vspace{-2mm}
\begin{enumerate}
\item[{\rm \bf (i)}] 
either
$\displaystyle \, 
      \int_{K} |\xi_n| \to 0 \hbox{ for any compact set } K \hbox{ with positive measure}, 
$
\item[{\rm \bf (ii)}] 
or,
$\displaystyle \, 
       \inf_{n } \int_{K} |\xi_n| >0  \hbox{ for any compact set } K \hbox{ with positive measure} .
$
\end{enumerate}
\end{lemma}
\begin{proof} 
We argue by contradiction.  Let $K_1, K_2$ be two disjoint compact subsets of $\Omega$ with positive measure such that we can find a sequence $\{\xi_n\}$ satisfying
$$
Q_V (\xi_n) \to 0,\; \int_{K_1} |\xi_n | \to 0 \quad  \hbox{ and }   \quad  \inf_{n } \int_{K_2} |\xi_n| >0.
$$
Since $Q_V(\xi_n)=Q_V(|\xi_n|)$, without loss of generality we may assume that $\xi_n$ is nonnegative for all $n$.
Consider the function $\Psi$ given by the previous proposition for the pair $K_1,K_2$. Let $\Omega_n \Subset \Omega$ contain $K_1,K_2$ and the supports of $\Psi$ and $\xi_n$. Then, by corollary \ref{aba} we can expand for any $t \in \R$,

\begin{eqnarray}
\|T(\xi_n + t \Psi) \|_{V,\Omega_n} &=& \int_{\Omega} \Big\{ |\nabla(\xi_n + t \Psi)|^2 + V (\xi_n + t \Psi)^2 \Big\} \notag \\
     &=& Q_V(\xi_n) + t^2 \|T(\Psi) \|_{V,\Omega_n} + 2 t \int_{\Omega} \xi_n ( - \Delta \Psi + V \Psi)  \nonumber\\
     &\leq& 
     o_n(1) + t^2  \|T(\Psi) \|_{V,\Omega_n} + 2 t \int_{\Omega_n} \xi_n ( - \Delta \Psi + V^+ \Psi)  \label{eq:Piscataway}
\end{eqnarray}
By the property of $\Psi$, we note that:
$$
   \int_{K_1} \xi_n ( - \Delta \Psi + V^+ \Psi)   \leq  o_n(1),
\quad
\sup_n   \int_{K_2} \xi_n ( - \Delta \Psi + V^+ \Psi)   < 0,
$$
and
$$
  \int_{\Omega_*  \setminus (K_1 \cup K_2)} \xi_n ( - \Delta \Psi + V^+ \Psi)   \leq 0 .
$$
Hence, for $t>0$ small enough and a fixed large  $n$, we deduce that the right-hand side of ~\eqref{eq:Piscataway} is strictly negative, a contradiction.
\end{proof}

\medskip

The above proposition implies the following alternative.
\begin{corollary}\label{knock}
 Let  $V^+  \in L^p_{loc}(\Omega)$ for some $p>\frac{N}{2}$  \;($N \geq 2$), $V^- \in L^1_{loc}(\Omega)$ and $Q_V \succeq 0$ in $\Omega$. 
Then, 
\begin{enumerate}
\item[{\rm \bf (i)}]
either  ${\cal L}_V(\Omega)  \hookrightarrow L^1_{loc} (\Omega)$,
 \item[{\rm \bf (ii)}]
or, the restriction map ${\cal L}_V(\Omega)  \to L^1 (K)$ is discontinuous for any compact set $K$ with positive measure. 
\end{enumerate}
\end{corollary}

\medskip

Finally, we can show the equivalence of the two notions of criticality when $V$ is tame: 

\begin{proposition} \label{prop:Sonata}
Assume $V \in L^1_{loc}(\Omega)$ and $Q_V \succeq 0$ in $\Omega$.
\vspace{-2mm}
\begin{enumerate}
 \item[{\rm \bf (i)}] 
 Let $V^+  \in L^p_{loc}(\Omega)$ for some $p>\frac{N}{2}$  \;($N \geq 2$).  Then,  $-\Delta + V$ is $L^1$-subcritical if it is feebly $L^2$-subcritical.
 \item[{\rm \bf (ii)}] 
 Suppose $V $ is tame in $\Omega$. Then,  $-\Delta + V$ is globally $L^2$-subcritical in $\Omega$  if  and only if 
the restriction map ${\cal L}_V(\Omega)   \to L^1 (K)$ is continuous for some compact set $K$ with positive measure.
\item[{\rm \bf (iii)}] 
Suppose $V $ is tame in $\Omega$. Then,  $-\Delta + V$ is globally $L^2$-subcritical in $\Omega$  if and only if $-\Delta + V$ is feebly $L^2$-subcritical in $\Omega$. 
\end{enumerate}
\end{proposition}
\begin{proof}
{\rm \bf (i)} 
Let  $Q_V \succeq c \chi_K$  in $\Omega$ for some constant $c>0$ and some compact set $K \subset \Omega$ with positive measure.   By an application of H\"older's inequality,
$$
 c^{-\frac{1}{2}}|K|^{\frac{1}{2}}\sqrt{Q_V(\xi)} \geq |K|^{\frac{1}{2}}\Big (\int_K  \xi^2 \Big )^{\frac{1}{2}}  \geq \int_K |\xi|, \quad \forall \xi \in W_c^{1,\infty}(\Omega).
$$
Appealing to corollary \ref{knock}, we obtain that ${\cal L}_V(\Omega) \hookrightarrow L^1_{loc}(\Omega)$.

\medskip

{\rm \bf (ii)} 
follows from  corollary \ref{knock} and proposition \ref{basa}.

\medskip

{\rm \bf (iii)} 
follows from (i) and (ii) above.
\end{proof}

\section{AAP Principle for subcritical and critical operators with tame potentials  } \label{s10}

\begin{remark}\label{host}
In view of the  proposition \ref{prop:Sonata} we will simply use the terminology ``subcritical" (``critical") instead of $L^1$, globally $L^2$ or feebly $L^2$-subcritical (critical) whenever   $V$ is tame.  
\end{remark}

\medskip

We can make more precise the distributional AAP Principle given  in theorem \ref{thm:chitra}  by considering separately the two subclasses of non-negative operators, namely the class of sub-critical and the critical operators. Indeed, we  formulate this in lemmas \ref{lem:adiga1} and \ref{lem:char crit} in terms of the Riesz measures. For tame potentials, a neat version is given in theorems \ref{fen} and \ref{fenu}. Since we will use the equivalence between the two criticalities stated in proposition \ref{prop:Sonata}, we will often  restrict $V$ to be tame in $\Omega$.

\subsection{Distributional AAP Principle for subcritical operators in terms of Riesz measures} \label{s10.1}

 In the light of  proposition \ref{prop:Sonata}, we first restate the results in section \ref{s8} when $V$ is tame in $\Omega$:
\begin{theorem}\label{thm:AAP-sub}
Assume $V \in L^1_{loc}(\Omega), V$ is tame in $\Omega$. Then $-\Delta + V$ is subcritical in $\Omega$  if and only if there exists a non-negative distribution solution $u_*$ of the equation $-\Delta u+ Vu=f$ where $f$ is some nonnegative locally integrable function satisfying  ${\it \inf_K f > 0}$ for some compact set $K \subset  \Omega$ with positive measure.
\end{theorem}

\begin{proof}
The ``only if" part  follows from propositions \ref{basa}, \ref{prop:Sonata} and theorem \ref{prop:Sona}.  

Conversely, if such  $u_*$ and $f$ exist, by theorem \ref{AAP-distr}, we get $Q_V \succeq f/ u_*$  in $\Omega$. In particular 
$f/ u_*\in L^1_{loc}(\Omega)$. Since $f/u_* \not \equiv 0$ in $K$, for some $c>0$ we also have that 
$\{f/ u_*\geq c \}\cap K $
 has positive measure. Then, $Q_V \succeq c \chi_{\tilde{K}}$ in $\Omega$ for any compact  $\tilde{K} \subset \{f/ u_*\geq c\} \cap K$ with positive measure.  Therefore, $-\Delta + V$ is subcritical by proposition \ref{prop:Sonata}.
\end{proof}

Next, we show the following ``solvability" property for a subcritical operator based on proposition \ref{prop:Sonata}.
\begin{lemma}\label{madura}
Let $V \in L^1_{loc}(\Omega)$ and $V$ tame in $\Omega$.  If the problem $-\Delta u +Vu = \chi_K$ admits a nonnegative distribution solution  for some compact $K \subset \Omega$ with positive measure, then the problem $-\Delta u +Vu=f$ admits a distribution solution for any $f \in L^{\infty}_c(\Omega)$.
\end{lemma}
\begin{proof}
By theorem  \ref{thm:AAP-sub}, $-\Delta+V$ is subcritical in $\Omega$ and the solvability for any $f \in L^{\infty}_c(\Omega)$ follows from  corollary \ref{pomo}.
\end{proof}

Finally, we show an AAP Principle for subcritical operators in terms of the Riesz measures. We note that, in contrast with results in section \ref{s8},  the solutions need not lie in the space ${\cal H}_V(\Omega)$. 

\begin{lemma}\label{lem:adiga1}
Let $V \in L^1_{loc}(\Omega)$ and $Q_V \succeq 0$ in $\Omega$. 
\begin{enumerate}
\item[{\rm \bf (i)}] 
Assume $V^+ \in L^p_{loc}(\Omega)$ for some $p>\frac{N}{2}$  \;($N \geq 2$). Then, $-\Delta+V$ is  feebly $L^2$-subcritical in $\Omega$ if  ${\cal M}_V (\Omega) \neq \{0\}$ (see definition \ref{reisz}).
\item[{\rm \bf (ii)}]
Conversely, assume $V$ is balanced. Then ${\cal M}_V (\Omega) \neq \{0\}$ if $-\Delta+V$ is   $L^1$-subcritical in $\Omega$.
\end{enumerate}
\end{lemma}

\begin{proof}  
{\bf (i)} 
We recall that ${\cal M}_V (\Omega) \neq \emptyset$  by theorem \ref{thm:converse3}. 
 Let $\mu \in {\cal M}_V (\Omega)$, $\mu \not \equiv 0$ and $u \in L^1_{loc}(\Omega)$ be a nonnegative function solving $-\Delta u + Vu = \mu$ in ${\cal D}'(\Omega)$. From proposition \ref{prop:RanRan}, $u \in W^{1,1}_{loc}(\Omega)$ and from theorem  \ref{thm:StrongMax}, $u>0$ a.e. in $\Omega$. Take an exhaustion $\{\Omega_n\}_{n=1}^{\infty}$ of $\Omega$ as follows:
$$ 
    \Omega_n \hbox{ an open bounded set, } \Omega_1 \Subset \Omega_n \Subset \Omega_{n+1} \Subset \Omega \hbox{ and  } \mu(\Omega_1)>0.
$$

Choose  balls $B \Subset B_1 \Subset \Omega_1$ such that $\mu(\Omega_1\setminus B_1)>0$. Let $\tilde{\mu}= \mu \llcorner (\Omega_1\setminus B_1) $.  Let $\theta_n \in W^{1,1}_0(\Omega_n)$ be a very weak solution (refer proposition \ref{prop:train})  of 

$$-\Delta \theta_n + (V^+ + \chi_B)\theta_n = \tilde{\mu} \hbox{ in } \Omega_n, \; \theta_n=0 \hbox{ on } \partial \Omega_n.$$ 
Then, by remark \ref{phat},   $\theta_n$ solves the above equation in the distributional sense.
  By using the distributional Kato's inequality as in the claim  of theorem \ref{AAP-distr}, we obtain that $\theta_n \geq 0$ in $\Omega_n $ for all $n$ as well as
$$ \theta_n \leq \theta_{n+1} \leq u \; \hbox{ in } \Omega_n \hbox{ for all  } n \geq 1. $$ Since $\theta_n \in W^{1,1}_0(\Omega_n)$ is nontrivial, we may find a quasi-continuous representative  of $\theta_n$ (which we denote again by $\theta_n$) that is positive q.a.e. by theorem \ref{thm:StrongMax}.
By elliptic regularity and Harnack principle, $\theta_n$ is continuous and positive outside the support of $\tilde{\mu}$ and in particular on $B$.
Define $z_n:=u-\theta_n$. Then,

$$ -\Delta z_n + Vz_n \geq (V^- + \chi_B) \theta_n  \hbox{ in  }  {\cal D}^{\prime}(\Omega_n), \quad 0 \leq z_{n+1} \leq z_{n} \hbox{ in } \Omega_n.$$ 
Clearly $z_n$ is nontrivial and from  theorem \ref{thm:StrongMax} again, we get $z_n>0$ a.e. in $\Omega_n, n \geq 1.$
Appealing to theorem \ref{AAP-distr}, we get 
\begin{equation} \label{eat}
Q_V \succeq (V^- +\chi_B)(\theta_n/z_n) \hbox{ in } \Omega_n.
\end{equation} 
We see that   $w:= (V^- + \chi_B)(\theta_1/z_1 )$ is an a.e. positive function on $B$.  Therefore,  for some $c>0$, we  have that  the set $\{w \geq c \}\cap B $
 has positive measure.  We  choose a compact set  $\tilde{K} \subset \{w \geq c\} \cap B$ with positive measure.   

We note from the monotonicity of $\theta_n$ and $z_n$  that  
\begin{equation}\label{et}
 (V^- + \chi_B)(\theta_n/z_n) \geq (V^- + \chi_B)(\theta_1/z_1 )\chi_{\Omega_1} \hbox{ in }\Omega_n,  \; \forall n \; \geq 1. 
\end{equation}
 Then, from \eqref{eat}, \eqref{et}  and the choice of $\tilde{K}$ we get that, 
 $$Q_V \succeq c \chi_{\tilde{K}} \text{ in } \Omega_n,  \; \forall \; n \geq 1.$$ 

Therefore,  $-\Delta+V$ is feebly $L^2$-subcritical in $\Omega$. \\
(ii)  follows from corollary \ref{pomo}. 
\end{proof}
As a consequence of the above lemma, we have the following result for tame potentials.
\begin{theorem}[AAP principle in ${\cal D^\prime}(\Omega)$ for subcritical operators]\label{fen}
Suppose that $V \in L^1_{loc}(\Omega)$  is tame in $\Omega$.  Then $-\Delta+V$ is subcritical in $\Omega$ if and only if  ${\cal M}_V (\Omega) \neq \{0\}$.
\end{theorem}

\subsection{Distributional AAP Principle for Critical operators in terms of Riesz measures} \label{s10.2}
 
 We state the following AAP Principle for critical operators, which follows directly from lemma \ref{lem:adiga1} and using the fact  
 ${\cal M}_V (\Omega)\neq \emptyset$ (see theorem \ref{thm:converse3}):
 
\begin{lemma}[AAP Principle in ${\cal D^\prime}(\Omega)$ for Critical Operators]\label{lem:char crit}
Let $V \in L^1_{loc}(\Omega)$ such that $Q_V \succeq 0$ in $\Omega$. 
\begin{enumerate}
 \item[{\rm \bf (i)}] 
 Assume $V^+ \in L^p_{loc}(\Omega)$ for some $p>\frac{N}{2}$  \;($N \geq 2$). Then,  ${\cal M}_V (\Omega) = \{0\}$ if  $-\Delta+V$ is  feebly $L^2$-critical in $\Omega$.
\item[{\rm \bf (ii)}] 
Conversely, assume $V$ is balanced. Then  $-\Delta+V$ is   $L^1$-critical in $\Omega$ if ${\cal M}_V (\Omega) = \{0\}$.
\end{enumerate}
\end{lemma}

\medskip

The above lemma implies the following result for tame potentials.
\begin{theorem}[AAP principle in ${\cal D^\prime}(\Omega)$ for critical operators]\label{fenu}
Suppose that $V \in L^1_{loc}(\Omega)$  is tame in $\Omega$.  Then $-\Delta+V$ is critical in $\Omega$ if and only if  ${\cal M}_V (\Omega) = \{0\}$. That is, any nonnegative super solution of $-\Delta + V$ in $\Omega$ is infact a solution.
\end{theorem}

By applying theorems \ref{thm:converse3},  \ref{fenu} and \ref{thm:StrongMax}, we obtain
\begin{corollary}\label{cor:cash}
Assume  $V \in L^1_{loc}(\Omega)$ is tame in $\Omega$ and $-\Delta +V$ is critical in $\Omega$. 
Then there exists an a.e.  positive distribution solution $u$ satisfying $(-\Delta+V)u = 0 $ in $\Omega$. 
\end{corollary}

\medskip

The above corollary should be viewed together with the proposition \ref{cor:run}. We have then  the following result for those nonnegative operators which extend as nonnegative operators to a bigger domain:
\begin{proposition}\label{4cr}
 Assume that  $\overline{\Omega} \neq \R^N$. Let  $V \in L^1_{loc}(\Omega)$  be tame in $\Omega$.
If $-\Delta + V$ is subcritical in $\Omega$ we assume additionally that there exist an open set $\tilde{\Omega}$ and $\tilde{V} \in L^1_{loc}(\tilde{\Omega})$ such that
\begin{enumerate}
\item[(i)] 
$\tilde{\Omega} \setminus \Omega$ contains a compact set $K$ with positive measure
\item[(ii)]  
$\tilde{V}$ extends $V$ and $\tilde{V}$ is tame in $\tilde{\Omega}$ (in particular $Q_{\tilde{V}} \succeq 0$ in $\tilde{\Omega}$).
\end{enumerate}
Then  there exists an a.e. positive distribution solution $u$ satisfying $(-\Delta+V)u = 0 $ in $\Omega$.
\end{proposition}
\begin{proof} If $-\Delta + V$ is critical in $\Omega$, by the above corollary, we are done. Otherwise, we consider the extended operator $-\Delta + \tilde{V}$.  If $-\Delta + \tilde{V}$ is subcritical in $\tilde{\Omega}$  we  apply proposition \ref{cor:run} with the compact set $K$. If it is critical, then the result follows again  from the above corollary.
\end{proof}

\begin{remark} \label{owa}
 There are well-known examples where the above assumptions don't hold, for example  the boundary Hardy operator on a ball can't be extended to a bigger domain as any such extension  will not be locally integrable in the bigger domain.
 \end{remark}

\begin{remark}\label{sore}
To obtain  the above conclusion for any nonnegative operator under the assumptions that $V$ is locally bounded or $V \in L^p_{loc}(\Omega)$ for $p>\frac{N}{2}\; (N \geq 2)$,  see for example  theorem 2.3 in \cite{PT} and \cite{PinTint}.
\end{remark}

The following is a restatement of the results given in proposition \ref{cor:run}, corollary \ref{cor:cash} and theorem \ref{AAP-distr}.
\begin{corollary}\label{5cr}
Assume  $V \in L^1_{loc}(\Omega)$ and $V$ tame in $\Omega$. Then $Q_V \succeq 0$ in $\Omega$  iff  for any compact $K \subset \Omega$ with positive measure, there exists an a.e. positive distribution solution $u$ to the  problem $-\Delta u + Vu=0$  in $\Omega \setminus K$.
\end{corollary}
\proof The ``only if" part follows from corollaries \ref{pomo} and \ref{cor:cash}. To show the converse, let $\xi  \in W_c^{1,\infty}(\Omega)$. Choose a compact set $K \subset \Omega$ with positive measure disjoint from the support of $\xi$. Then, by assumption, there exists an a.e. positive distribution solution $u$ to the  problem $-\Delta u + Vu=0$  in $\Omega \setminus K$. Hence by theorem \ref{AAP-distr}, $Q_V \succeq 0$ on $\Omega \setminus K$. In particular, $Q_V(\xi) \geq 0$.  \qed\\

We recall here the definition of a ``null" sequence:
\begin{definition}\label{null}
Let $Q_V \succeq 0$ in $\Omega$. A sequence $\{\xi_n\} \subset W^{1,\infty}_c(\Omega)$ is called a null sequence for $Q_V$ if $Q_V(\xi_n) \to 0$ and there exists a compact set $K \subset \Omega$ with positive measure such that $\int_K |\xi_n| =1$ for all $n$.
\end{definition}

The following proposition just restates the result in proposition \ref{prop:Sonata}.
\begin{proposition}\label{om}
Let $V \in L^1_{loc}(\Omega)$ be tame in $\Omega$ and  $Q_V \succeq 0$ in $\Omega$. Then $-\Delta+V$ is critical  in $\Omega$ iff $Q_V$ has a null sequence.
\end{proposition}

The following  results show that being a critical operator is an ``unstable" state:

\begin{lemma}\label{nin}
Let $ U \subset \Omega \subset \tilde{\Omega}$ be  open sets  such that $\tilde{\Omega}\setminus \Omega$ and $\Omega \setminus U$  contain a compact set  with positive measure. Let $V \in L^1_{loc}(\tilde{\Omega})$ and  $V^+ \in L^p_{loc}(\tilde{\Omega})$ for some $p>\frac{N}{2}  \;(N \geq 2)$. If  $Q_V \succeq 0$ in $\Omega$, then $-\Delta+V$ is $L^1$-subcritical in $U$. If $ -\Delta+V$ is $L^1$-critical in $\Omega$ then it is supercritical in $\tilde{\Omega}$.
\end{lemma}
\begin{proof} 
 Fix a compact set $K \subset \Omega \setminus U$ of positive measure. We obviously have  $Q_{V+ \chi_K} \succeq \chi_K$ in $\Omega$. 
 By proposition \ref{prop:Sonata}(i) we have that $-\Delta+V+\chi_K$ is  $L^1$-subcritical in $\Omega$ and hence  we obtain that $-\Delta+V$ is  $L^1$-subcritical in $U$. 
 
 Suppose $Q_V \succeq 0$ in $\tilde{\Omega}$. If $-\Delta + V$ is $L^1$-critical in $\tilde{\Omega}$, by the same arguments as above we obtain that $-\Delta + V$ is $L^1$-subcritical in $\Omega$, a contradiction.  If $-\Delta + V$ is $L^1$-subcritical in $\tilde{\Omega}$ so will  it  be in $\Omega$, again the same contradiction. Thus,  $-\Delta+V$ is   supercritical in $\tilde{\Omega}$. 
\end{proof}

\medskip

\begin{corollary}\label{cor:fun}
Let $V, U,\Omega,\tilde{\Omega}$ be as in the above lemma. 
If $V$ is balanced in $\Omega$, then $-\Delta+V$ is subcritical in $U$. 
If additionally $ -\Delta+V$ is critical in $\Omega$ then it is supercritical in $\tilde{\Omega}$.
\end{corollary}
\proof Follows from the above lemma and proposition \ref{prop:Sonata}.  \qed

\medskip

\begin{proposition}\label{but}
Assume $V \in L^{1}_{loc}(\Omega)$ is such that  $Q_V \succeq 0$  in $\Omega$.
\begin{enumerate}
\item[{\rm \bf (i)}] 
Take any nonnegative $w \in L^1_{loc}(\Omega)$ such that $w \not \equiv 0$. Then $- \Delta + V+w$ is a feebly $L^2$-subcritical operator in $\Omega$. If $- \Delta + V$ is a feebly $L^2$-critical operator in $\Omega$, then  $- \Delta + V-w$ is supercritical in $\Omega$.
\item[{\rm \bf (ii)}] 
If  $V$ is tame in $\Omega$ and $w \in L^{\infty}_{loc}(\Omega), w \not \equiv 0$ and is nonnegative, in the statement (i) above we can replace ``feebly $L^2$-critical (subcritical)" by critical (subcritical).
\end{enumerate}
\end{proposition}
\proof
{\bf (i)} We note that  we may find a  compact set $K$ of positive measure and a $c>0$ such that $K \subset \{w \geq c\}$. We then see that,
$$Q_{V+w} (\xi) = Q_{V}(\xi)+ \int_\Omega w \xi^2 \geq c \int_\Omega \chi_K  \xi^2\; \;\forall \; \xi \in W_c^{1,\infty}(\Omega).
$$
Hence, $- \Delta + V$ is feebly $L^2$-subcritical.

 If the operator $-\Delta + V$ is feebly $L^2$-critical, there exists $\xi_{0} \in W_c^{1,\infty} (\Omega)$ such that
$$
     \int_{\Omega}  |\nabla \xi_{0} |^2 + V \xi_{0}^2   < c \int_\Omega \chi_K  \xi_0^2 \leq \int_{\Omega} w \xi_{0}^2.
$$
That is, 
$Q_{V-w}(\xi_0)<0.$

\medskip

{\bf (ii)} 
follows from the fact that $V+w$ is tame in $\Omega$ (by proposition \ref{mea}(ii)) and proposition \ref{prop:Sonata}. 
\qed

\subsection{Characterisation of subcritical/critical operators in terms of dimension of ${\cal C}_V(\Omega)$ when $V$ is tame in  $\Omega$}\label{s10.3}
In this subsection we improve the results in  section 10.3 of \cite{MPJDE} to the class of tame potentials. 
\begin{theorem} \label{thm:Simplicity}
Let $V$ be tame in $\Omega$ and assume that  $ -\Delta+V$ is critical in $\Omega$. 
If $\Phi_1 $, $\Phi_2$ are two non-negative supersolutions of $-\Delta+V$ in $\Omega$, then $\Phi_2 = t \Phi_1$ for some $t \geq 0$.
\end{theorem}

\proof The result follows easily  if one of $\Phi_1,\Phi_2$ is the trivial solution. Let us then assume that both are nontrivial supersolutions. 
By the well-known result (cf. proposition 5.9 in \cite{PoCo}), we get that $\Phi:= \min\{\Phi_1, \Phi_2 \}$ is also a supersolution
of  $-\Delta+V$  and hence by criticality of the operator (theorem \ref{fenu}), 
$$
( -\Delta+V)   \Phi_1 = 0 \quad \hbox{ and } \quad  (-\Delta+V) \Phi = 0 \quad \text{ in } \Omega.
$$
Therefore, the non-negative function $\Phi_1 - \Phi$ satisfies $(-\Delta+V)(\Phi_1 - \Phi) = 0$ and 
the strong maximum principle in Theorem \ref{thm:StrongMax}  gives
that either, 
$$
    \Phi_1 >   \min\{ \Phi_1, \Phi_2 \}  \, \,  a.e.  \quad \hbox{ or }   \quad 
    \Phi_1 \equiv \min\{ \Phi_1, \Phi_2 \} \, \, a.e.
$$  
That is, the set of all non-negative supersolutions is  linearly ordered with respect to the pointwise  order ``$\leq$" on $\Omega$. 
Without loss of generality,  we can assume $\Phi_1 \leq  \Phi_2$ a.e. in $\Omega$. 
Consider  the set
$$
   T := \{ t \in \mathbb R  \, : \, \Phi_1 + t \Phi_2 \geq 0  \;\;\text{a.e. in} \;\; \Omega\}.
$$
We note that $T $ is an interval which is not empty (since $0 \in T$) and bounded below (since $\Phi_1 , \Phi_2 \not\equiv 0$, we have $-2 \not \in T$). Hence, $t_0 := \inf T$ is finite and by definition 
\begin{equation} \label{eq:Sole}
       \Phi_1 + t_0 \Phi_2  \geq 0,\quad  \text{i.e.,} \quad  t_0 \in T.
\end{equation}
Now, for any $\varepsilon >0$ 
the nonnegative supersolutions $\varepsilon \Phi_2$ and  $\Phi_1 + t_0 \Phi_2$ are comparable and  hence 
 either  
$   \Phi_1 + t_0 \Phi_2  \geq \varepsilon \Phi_2 $\,  a.e.,
    or   
  $ \Phi_1 + t_0 \Phi_2  \leq  \varepsilon \Phi_2$  \, a.e..
The first alternative   contradicts the minimal property of $t_0$.  Hence, we deduce that 
$
    \Phi_1 + t_0 \Phi_2  \leq  \varepsilon \Phi_2 \quad  a.e. 
    \quad
    \forall \varepsilon >0.
$
Letting $\varepsilon \to 0$  we deduce that $ \Phi_1 + t_0 \Phi_2  \leq 0$ in $\Omega$, which combined with \eqref{eq:Sole} shows that $ \Phi_1 + t_0 \Phi_2  \equiv 0$ in $\Omega$. \qed

We can now state the following result which characterises critical operators in terms of the dimension of the cone of non-negative supersolutions ${\cal C}_V(\Omega)$:

\begin{theorem} 
Let $V \in L^{1}_{loc}(\Omega)$ be tame in $\Omega$. Then 
\vspace{-2mm}
\begin{enumerate}
\item[{\bf (i)}]
$-\Delta + V$ is subcritical in $\Omega$ iff
${\cal C}_V(\Omega)$  contains an infinite set of linearly independent functions.
\item[{\bf (ii)}]
$-\Delta + V$ is critical in $\Omega$ iff
${\cal C}_V(\Omega)= \{tu_0: t >0 \} $
for some nonnegative  nontrivial   $ u_0 \in L^1_{loc}(\Omega)$ solving $-\Delta u_0 +Vu_0=0$ in ${\mathcal D}^{\prime}(\Omega)$.
\end{enumerate}
\end{theorem}
\proof 
{\bf (i)} 
The ``only if" part follows from  corollary \ref{bindu}. The converse follows from  theorem \ref{fenu}  and theorem \ref{thm:Simplicity} above. 

\smallskip

{\bf (ii)}
The  ``if" part follows from (i). For the converse, we note that ${\cal C}_V(\Omega) \neq \emptyset$ by theorem \ref{thm:converse3}. By theorem  \ref{fenu} any element of ${\cal C}_V(\Omega)$ is a distributional solution of $-\Delta u +Vu=0$ in $\Omega$. 
Conclusion follows  from  theorem \ref{thm:Simplicity} above.  
\qed

\section{Examples and applications} \label{s11}
\subsection{Non-negative operators} \label{s11.1}

\begin{lemma}\label{levy}
Let $u \in W^{2,1}_{loc}(\Omega)$ be such that $\inf_B u >0$ for any ball $B \Subset \Omega$ and $f \in L^1_{loc}(\Omega)$ be a nonnegative function. Let $V:= (\Delta u+f)/u$. Then, $-\Delta + V$ is a non-negative operator.
\end{lemma}
\proof Clearly, $V $ is locally integrable in $\Omega$ and $u$ solves $-\Delta u +Vu =f$ in ${\mathcal D}^\prime(\Omega)$. Thus, the associated quadratic form satisfies $Q_V \succeq 0$ in $\Omega$ by theorem \ref{AAP-distr}.
\qed

\medskip

In the following example we show that it is possible to choose in this manner $V \in L^1(\Omega)$ but $V \not \in L^p(\Omega)$ for any $p>1$.
\begin{example}\label{bump}
Denote by $I$ the unit interval $[0,1]$ and  by $I^N$ the cube $[0,1]^N$. Let $f \in L^\infty(I^N)$ be a nonnegative function. Choose a  function $\rho$ on $I$ such that $\rho \in L^1(I)$. Define $w: I^N \to \R$ by 
$$
w(x)=w(x_1,x_2,\cdot \cdot \cdot,x_N) := \int_0^{x_1} \Big(\int_0^t \rho(s)ds \Big)  dt.
$$
Clearly $w \in W^{2,1}(I^N)\cap C(I^N)$  and 
$$
\Delta w (x) = \frac{\partial^2 w}{\partial x_1^2}(x)= \rho(x_1) \; \text{ a.e. in } \; I^N.
$$
Choosing $k>0$ large enough so that $w+k>0$ in $I^N$ and letting $u:=w +  k$, we see that $u \in W^{2,1}(I^N)$ is a positive continuous function  which solves $-\Delta u + Vu=f$ in ${\mathcal D}^\prime(I^N)$ where 
\begin{equation}\label{fol}
V(x_1,x_2,\cdot \cdot \cdot,x_N):= (\rho(x_1)+f(x))/(w(x)+k).
\end{equation}
Clearly, if we further choose $\rho \not \in L^p(I)$ for any $p>1$, we get that $V  \in L^1(I^N)$ but $V \not \in L^p(I^N)$ for any $p>1$. By the above lemma $Q_V \succeq 0$ in $\Omega$.
\end{example}

\medskip

\begin{example}\label{ma}
Let $\Omega:= I^N, N \geq 3$.  Choose $f \equiv 1$ and $\rho \in W^{1,1}(I)$ such that 
\vspace{-2mm}
\begin{enumerate}
 \item[(i)] $\rho + 1 <0$, 
  \item[(ii)] 
  $\rho^{\; \prime} \not \in L^p(I)$  for any $p>1$.
\end{enumerate}
Let $V$  be the  negative singular potential  as given in \eqref{fol}. Thus, $Q_V \succeq 0$ in $\Omega$. We choose the vector field
$$\vec {\bf {\Gamma}} := \frac{1}{2\sqrt{N}} (\sqrt{-V},\sqrt{-V}, \cdot \cdot \cdot, \sqrt{-V}).$$
Let $\sigma := div \; \vec {\bf {\Gamma}}$.  We note that
 \begin{equation}\label{agm}
 \sigma (x_1,x_2, \cdot \cdot \cdot, x_N)  = \frac{\sqrt{-V(x)}}{4 \sqrt{N} (\rho(x_1)+1)} \Big( V \int_0^{x_1} \rho(t) dt - \rho^{\;\prime}(x_1) \Big). \notag
 \end{equation}
 From the above equation and the properties of $\rho$, we see that $\sigma$ has the same integrability as $\rho^{\;\prime}$ i.e., $\sigma \in L^p(I^N)$ only for $p=1$.
We check that $Q_{-4|\vec {\bf {\Gamma}} |^2} = Q_V \succeq 0$ in $\Omega$. 
Therefore,  from theorem 4.1 (iii) in \cite{JayeMazyaVerb} we obtain that 
\begin{equation}\label{dar}
\Big|\int_\Omega \sigma \xi^2 \Big|  \leq \int_\Omega |\nabla \xi|^2, \;\;\forall \xi \in W^{1,\infty}_c(\Omega).
\end{equation}
That is,  $\sigma$ is form bounded,
$Q_\sigma \succeq 0$ in $\Omega$ and the identity map\\
$$id : \Big(W^{1,\infty}_c(\Omega), \| \cdot \|_{H^1_0(\Omega)}\Big)  \to  {\cal L}_\sigma(\Omega) \;\;\text{ is continuous. }\;\;$$
Fix $t \in (0,1)$. If we consider the potential $t \sigma$, we obtain from \eqref{dar} that the above map is a homeomorphism onto ${\cal L}_{t\sigma}(\Omega)$. That is, the corresponding energy space ${\cal H}_{t \sigma}(\Omega)$ can be identified with $H^1_0(\Omega)$. By corollary \ref{yid}, we obtain that $\rho^{\;\prime} w \in L^1_{loc}(\Omega)$ for any $w \in H^1_{loc}(\Omega)$ and hence from proposition \ref{inz}  we obtain that $t \sigma$ is a balanced potential in $\Omega$.
 \end{example}

\subsection{Subcritical operators} \label{s11.2}
\begin{lemma}\label{lev}
Let $\Phi \in W^{2,1}_{loc}(\Omega)$ be such that $\inf_B \Phi >0$ for any ball $B \Subset \Omega$ and $f \in L^1_{loc}(\Omega)$ be a nonnegative function such that $\inf_{K}f>0$ for some compact   $K \Subset \Omega$ with positive measure. Let $V:= (\Delta \Phi+f)/\Phi$. Then, $-\Delta + V$ is a feebly $L^2-$ subcritical 
 operator.
\end{lemma}
\proof We note that $-\Delta \Phi + V\Phi =f $ in ${\cal D^\prime}(\Omega)$. Since $\Phi,f$ are nonnegative, by theorem \ref{AAP-distr} we obtain that $Q_V \succeq f/\Phi$ in $\Omega$. The conclusion follows by noting that $f/\Phi \not \equiv 0$ in $K$.
\qed

\medskip

\begin{corollary}\label{lee}
Let $\Phi,f$ be as in the above  lemma with the additional assumptions: $(\Delta \Phi)^+$ and $f$ belong to $L^{\infty}_{loc}(\Omega)$. Take $V$ as in the above lemma.  Then, $V$ is a tame  potential  in $\Omega$ and given any $g \in L^{\infty}_c(\Omega)$, there exists a  unique $u \in {\mathcal H}_V(\Omega)$ solving $-\Delta u + Vu=g$ in ${\cal D^\prime}(\Omega)$.
\end{corollary}
\proof
We note that $V$ is tame in $\Omega$ by the above lemma and remark \ref{gia} (i). The proof of rest of the assertions follows from the above lemma, proposition \ref{prop:Sonata} and corollary \ref{pomo}.
\qed

\medskip

\begin{corollary}\label{loe}
 Take any $\Phi \in W^{2,1}_{loc}(\Omega)$  and $f \in L^{\infty}_{loc}(\Omega)$ satisfying
$$
\inf_B \Phi ,\; \inf_B f >0\;\text{ for any ball }\; B \Subset \Omega \;\text{ and } \; \Delta \Phi \in {\mathcal V}(\Omega) \; \text{(see definition \ref{uno})}.
$$ 

Define $V:= (\Delta \Phi+f)/\Phi$.  Then,  $V$ is balanced in $\Omega$ and  given any $g \in L^{\infty}_c(\Omega)$, there exists a  unique $u \in {\mathcal H}_V(\Omega)$ solving $-\Delta u + Vu=g$ in ${\cal D^\prime}(\Omega)$.
\end{corollary}
\proof
We note that $-\Delta+V$ is globally $L^2$-subcritical in $\Omega$ by following the proof of lemma \ref{lev}. Thus, the operator is $L^1$-subcritical in $\Omega$. Since  $\Delta \Phi  \in {\cal V}(\Omega)$ we get that $V \in {\cal V}(\Omega)$   by proposition \ref{reshma}(i). Thus $V$ is strongly balanced in $\Omega$ by proposition  \ref{paint}. The proof of solvability follows from  corollary \ref{pomo}.
\qed

\medskip

\begin{remark}\label{waz}
Let $\Omega$ be a bounded domain in $\R^N, \;N \geq 3$ and $\Gamma_N$ denote the fundamental solution of $-\Delta$ in $\R^N$. Assume $g \in L^1(\Omega)\cap L^r_{loc}(\Omega)\cap {\cal V}(\Omega)$ for some $r>1$ and satisfies
$$
\Psi(x):= -\int_{\R^N}  \Gamma_N (y-x) g(y) \chi_\Omega(y) dy \;\text{ is  bounded below in } \Omega.
$$
This can be achieved for example, by taking $g^+ \in L^p (\Omega)$ for some $p>\frac{N}{2}$. Then letting $\Phi = \Psi- \inf_\Omega \Psi +1$ we can satisfy the conditions in the last corollary.
\end{remark}

\medskip

\begin{lemma}\label{obv}
Let $W \in L^{1}_{loc}(\R^N)$ be a nonnegative function.
\begin{enumerate}
\item[{\rm \bf (i)}] 
Let $K \subset \R^N$ be a compact set with positive measure. Then, $-\Delta + W$ 
is  subcritical in $\R^N \setminus K$, $N\geq 1$.
\item[{\rm \bf (ii)}] 
$-\Delta + W$ is  subcritical in $\R^N$ if $ N \geq 3.$
\end{enumerate}
\end{lemma}
\begin{proof}
We remark that $Q_W \succeq Q_0$ in any $\Omega \subset \R^N$. Therefore, both $L^1$ and $L^2$ subcriticality of $-\Delta +W$ in $\Omega$ follow from the subcriticality of 
$-\Delta$ there.

\medskip

{\bf (i)} 
Since  $Q_0 \succeq 0$ in $\R^N$,  from corollary \ref{cor:fun} we obtain that $-\Delta$ is subcritical in $\R^N \setminus K$, $N \geq 1$.

\medskip

{\bf (ii)} 
It is enough to show $-\Delta$ is $L^2-$ subcritical in $\R^N$ if $N \geq 3$. Choose 
$$\Phi_N(x)=(N(N-2))^{\frac{N-2}{4}}(1+|x|^2)^{\frac{2-N}{2}}.$$ 
Then $\Phi_N$ is a non-negative locally integrable function solving $-\Delta \Phi_N= (\Phi_N)^{\frac{N+2}{N-2}} $ in ${\cal D^\prime}(\R^N)$. The assertion now follows from  theorem \ref{thm:AAP-sub}.\\
\end{proof}
\begin{corollary}\label{nandi}
Let $N \geq 1$ and $W \in L^{1}_{loc}(\R^N)$ be a nonnegative function. Then, given any  compact set $K \subset \R^N$ of positve measure, there exists an a.e. positive  distribution solution $u$ of $-\Delta u + W u =\chi_K$ in $\R^N$.
\end{corollary}
\begin{proof}
Follows from the above lemma,  remark \ref{gia}(i) and corollary \ref{cor:run}.
\end{proof}
In the following proposition, we generalise the example \ref{swamBis} (iii).

\begin{proposition} \label{sandy}
Let $V_0 $ be  a potential in $\Omega$  such that
$$ - \; C|x|^{-\beta}  \leq V_0(x) \leq  C \;\text{ for some } \; C>0, \beta \in (0, N).$$
Assume  $Q_{V_0} \succeq 0$ in $\Omega$.
 Let $S:=\{x_i\}  \subset \Omega$ be a countable   set and  $ \alpha_i $ be nonnegative numbers  such that  $$ 0 < \alpha := \sum_i \alpha_i  \leq 1.$$  Define   
 $$V(x)= \sum_i \alpha_i V_0(x-x_i).
 $$ 
 \begin{enumerate}
\item[{\rm \bf (i)}] 
$V$ is tame in $\Omega$ and there exists an a.e. positive distributional super solution to $-\Delta+V$ in $\Omega$.
\item[{\rm \bf (ii)}]  
Let $V_1 $  be a tame potential in $\Omega$ such that 
$$
V  \leq V_1,  \; V_1 \not \equiv V.
$$
Then, given any $f \in L^{\infty}_{c}(\Omega)$, there exists a  distribution solution $u $ of 
$$-\Delta u +  V_1 u =f \hbox{ in }\Omega.$$
  In fact, given any compact set $K$ in $\Omega$ with positive measure, there exists an a.e. positive  distribution solution $u$ of $$-\Delta u +  V_1 u =0 \; \text{ in } \; \Omega \setminus K. $$
\item[{\rm \bf (iii)}] 
If  $ \alpha <1$, both the assertions in (ii)  above hold for $V$.
\item[{\rm \bf (iv)}]  
If $\alpha=1$, given any compact set $K$ in $\Omega$ of positive measure and any
$f \in L^{\infty}_{c}(\Omega \setminus K)$,
 there exists a  distribution solution $u $ of 
$$-\Delta u + V u =f \hbox{ in }\Omega\setminus K .$$ 
As before, there exists an a.e. positive  distribution solution $u$ of $-\Delta u + V u =0$ in $\Omega \setminus K$. 
\end{enumerate}
\end{proposition}
\begin{proof}  
{\bf (i)} 
We note that $V$ is bounded above  in $\Omega$ and by the pointwise estimate on $V_0$, it is also locally integrable in $\Omega$.  Let $\{ \beta_i \}$ be any sequence of nonnegative numbers  such that  $  \sum_i \beta_i  \leq 1.$  Since $Q_{V_0} \succeq 0$ in $\Omega$, we obtain forall $\xi \in C_c^{\infty}(\Omega)$,
\begin{eqnarray}\label{kab}
\int_{\Omega} |\nabla \xi|^2 \geq \sum_i  \beta_i \int_{\Omega} |\nabla \xi(\cdot+x_i)|^2 \geq -\sum_i \beta_i \int_{\Omega} V_0 \; \xi^2(\cdot+x_i) . 
\end{eqnarray}

Taking $\beta_i=\alpha_i$ in the above inequality,  we get that $Q_V \succeq 0$ in $\Omega$ and hence by remark \ref{gia}(i), $V$ is tame in $\Omega$. Existence of a positive super solution follows from theorem \ref{thm:converse3}. 

\medskip

{\bf (ii)} 
From proposition \ref{but}(i), we get that $-\Delta+V_1$ is feebly $L^2$-subcritical in $\Omega$ and since $V_1$ is tame in $\Omega$, it is subcritical in $\Omega$. We  can  appeal to corollary \ref{pomo} and proposition \ref{cor:run}.

\medskip

{\bf (iii)} 
Let $\alpha <1$. Now, by the equation \eqref{kab} above,
\begin{eqnarray}\label{lea}
 \int_{\Omega} |\nabla \xi|^2 + V \xi^2 &=& (1-\alpha) \int_{\Omega} |\nabla \xi|^2 + \alpha \Big(\int_{\Omega} |\nabla \xi|^2 + \frac{1}{\alpha} \sum_i \alpha_i V_0(\cdot -x_i) \xi^2 \Big)  \notag \\
 &\geq &   (1-\alpha) \int_{\Omega} |\nabla \xi|^2 . \notag
\end{eqnarray}
Thus, ${\cal L}_V(\Omega) \hookrightarrow {\cal D}_{0}^{1,2} (\Omega)$.  By Sobolev imbedding,  $-\Delta+V$ is subcritical in $\Omega$. The assertions about the existence of distribution solutions  follow again from corollary \ref{pomo} and proposition \ref{cor:run}. 

\medskip

{\bf  (iv)}  follows from  corollaries \ref{cor:fun}, \ref{pomo} and proposition \ref{cor:run}. 
\end{proof}

\medskip

\begin{remark}
\begin{enumerate}
\item[{\rm \bf (i)}]
By taking $N \geq 3, \; \Omega=\R^N$ and $V_0 = -\frac{(N-2)^2}{4} |x|^{-2}$\; (or any of its radial improvements in bounded domains), we see that the above proposition gives corresponding statements about the multi-polar Hardy operator with poles in the set $S$. These poles are allowed to be dense in some portion or all of $\Omega$.
\item[{\rm \bf (ii)}]
For the multi-polar Hardy potential having only  finitely many poles but where $\alpha_i$  can be any real number less than $\frac{(N-2)^2}{4}$ , the problem of whether ${\cal L}_V(\Omega) \hookrightarrow {\cal D}_{0}^{1,2} (\R^N)$ is more involved and is treated in \cite{Ter}.
\item[{\rm \bf (iii)}]
A similar construction can be made when the ``mother potential" $V_0$ has a higher dimensional  singular set on a plane by a weighted sum of countable translations of $V_0$  perpendicular to the plane.
\end{enumerate}
\end{remark}
\medskip

\subsection{Critical operators}\label{s11.3}
Let $\Omega$ be a bounded domain and $V \in L^{1}_{loc}(\Omega)$.  Define the subspace and the functional
$$ H(V,\Omega):= \{u \in H^1_0(\Omega) : Vu^2 \in L^1(\Omega) \} ; $$  $$ I_V: H(V,\Omega) \to \R \hbox{ by } I_V (u) := \int_{\Omega} \Big\{|\nabla u|^2+Vu^2\Big\}.$$
Consider the corresponding ``first eigenvalue"
$$ \lambda_1 :=\inf\Big\{ I_V(u): u \in H(V,\Omega) \hbox{ and } \int_{\Omega} u^2=1 \Big \}.$$

The following are well-known examples of critical Schr\"odinger operators:
\begin{lemma}\label{fall}
Let  $V^+ \in L^{\frac{N}{2}}(\Omega)$  \;($N \geq 3$).  Assume that $\lambda_1>-\infty$ and is achieved by some element in $H(V,\Omega)$. Then, $- \Delta +V - \lambda_1$ is both feebly $L^2-$critical  and $L^1-$ critical in $\Omega$. 
\end{lemma}
\begin{proof} Let $\Phi \in H(V,\Omega)$ attain the infimum for $\lambda_1$. Considering $|\Phi|$ instead of $\Phi$, we may assume $\Phi $ is nonnegative. By the assumptions, we obtain that 
\begin{equation}\label{chad}
 I_V(u) -\lambda_1 \int_{\Omega} u^2 \geq 0 \; \; \forall u \in H(V,\Omega) \hbox{ and }  I_V(\Phi) -\lambda_1 \int_{\Omega} \Phi^2 = 0.
 \end{equation}
Therefore, $\Phi$ solves
$$
-\Delta \Phi +V \Phi = \lambda_1 \Phi \hbox{ in } {\cal D^\prime}(\Omega).
$$
From theorem \ref{thm:StrongMax} we obtain that $\Phi >0$ a.e. in $\Omega$. Suppose there exists a 
compact set $K \subset \Omega$ with positive measure and $c_K>0$ such that
$$
\int_{\Omega} \Big\{|\nabla \xi|^2+(V-\lambda_1)\xi^2\Big\} \geq c_K \int_{K} \xi^2, \; \forall \xi \in C_c^1(\Omega).
$$
We now choose a sequence $\{\xi_n\} \subset C_c^1(\Omega)$ such that $ \xi_n \to \Phi$ in $H^1_0(\Omega)$. Then, from the above inequality, using Sobolev imbedding and Fatou's lemma we get
$$
\int_{\Omega} \Big\{|\nabla \Phi|^2+(V^+-\lambda_1)\Phi^2\Big\} \geq \int_{\Omega} V^-\Phi^2 + c_K\int_{K}  \Phi^2.
$$
From \eqref{chad} we obtain that $\Phi \equiv 0$ in $K$ a contradiction to its positivity a.e.  This shows that $- \Delta +V - \lambda_1$ is feebly $L^2-$critical in $\Omega$. A similar argument shows that the operator  is  $L^1-$critical in $\Omega$ as well. 
\end{proof}

\medskip

\begin{lemma}\label{dumb}
The operator $- \Delta$ is critical in $\mathbb R^N$ when $N=1,2$.
\end{lemma}
\begin{proof}
Let $u \in L^1_{loc}(\Omega)$ be a non-negative distribution solution of $-\Delta u = f$ for some nonnegative locally integrable function $f$. If $N=2$, by the Louville property of nonnegative super harmonic functions we get $u$ is identically constant and hence $f \equiv 0$. If $N=1$, we note that $u$ becomes a nonnegative concave function in $\R$ and hence must be constant. Therefore, again $f \equiv 0$. Criticality follows from theorem \ref{thm:AAP-sub}.
\end{proof}

\subsection{Supercritical operators} \label{s11.4}

The following results provide a way of constructing supercritical operators using a wide class of singular potentials. 
To state the result in full generality, given a continous function $\phi: [0,\infty) \to [0,\infty)$ with $\phi(0)=0$, we introduce the Orlicz class
$$
    L_{\phi}(\Omega)  :=  \left\{ u:\Omega \to \R \, \hbox{ measurable} \, :  \, \int_{\Omega} \phi(|u|) < \infty \right\}  .
$$

We denote by $2^*$ the critical Sobolev exponent $\frac{2N}{N-2}$.

\begin{proposition}\label{non-existence-distr}
Let $\Omega$ be a bounded domain in $\R^N$, $N \geq 3$.
Let $\alpha: [0,\infty)\to [0, \infty)$ be a continuous function such that
\begin{enumerate}
 \item[{\rm \bf (i)}] 
 $\beta(t) := t^{-2}\alpha(t)$ is a continuous function in $[0,\infty)$ with $\beta(0)=0$
 \item[{\rm \bf (ii)}] 
 $H^1_0(\Omega) \setminus L_{\alpha}(\Omega) \neq \emptyset$ and $L^{2^*}(\Omega) \subseteqq L_{\beta}(\Omega)$.
 \end{enumerate}
For any $w \in H^1_0(\Omega) \setminus L_{\alpha}(\Omega)$ define the potential $V_w := -\beta(w)$.
Then, the operator $- \Delta+ V_w$ is supercritical.
\end{proposition}
\proof
Let $w \in H^1_0(\Omega) \setminus L_{\alpha}(\Omega)$. 
Since $w \in L^{2^*}(\Omega)$ and by assumption $L^{2^*}(\Omega) \subseteq L_{\beta}(\Omega)$, 
we obtain that $ V_w:= -\beta (w) \in L^1(\Omega)$. 
Therefore $w \in H^1_0(\Omega)$, but 
$$ \int_{\Omega} V_w w^2=
- \int_{\Omega} \alpha(w) = -\infty.$$  
Let now $\{\phi_n\} \subset C_c^{\infty}(\Omega) $ be a
sequence converging in $H^1_0(\Omega)$ to $w$.  From Fatou's lemma we obtain
that $\displaystyle \lim_{n \to \infty} Q_{V_w}(\phi_n) = -\infty$. 
\qed

\medskip

\begin{corollary}\label{carter}
Let $\Omega \subset \mathbb R^N, N \geq 3$ be a bounded domain and $p \in (2^{*}, 2^{*} + 2)$. If for each
$w \in H^1_0(\Omega)\setminus L^{p} (\Omega) $, we define the potential $V_w = - |w|^{p-2}$, then
 the operator $- \Delta+V_w$ is supercritical. Note that $V_w \not \in L^q(\Omega)$ for some $2(N-1)/N<q<N/2$.
\end{corollary}
\proof
Consider $\alpha(t) := |t|^{p}$ and $\beta (t) := |t|^{p-2}$ which is continuous (note that $p > 2$). Since 
$H^1_0 (\Omega) \not \subset L^{p} (\Omega)$ and $L^{2^*} (\Omega) \subset L^{p-2} (\Omega)$,
the conclusion follows from the above proposition.
\qed

\medskip

\begin{corollary}\label{gates}
Let $V$ be any of the potentials considered in   proposition \ref{non-existence-distr}. Then there exists no non-trivial non-negative 
distributional supersolution to the operator $-\Delta + V$.
\end{corollary}
\begin{proof}
For $V$ as in proposition \ref{non-existence-distr}, we have that $Q_V \not \succeq 0$ in $\Omega$. Therefore, in view of theorem \ref{AAP-distr}, there cannot exist a non-trivial non-negative 
distributional supersolution to the operator $-\Delta + V$ in $\Omega$.
\end{proof}

\medskip

 \section{Appendix } 
\subsection{ Weak compactness in $L^1_{loc}$}

The space $L^1_{loc} (\Omega)$ is a locally convex topological space with a countable family of seminorms defined by 
$|\cdot|_n:=\int_{\Omega_n} |u|$ where  $\{\Omega_n\}$ is an exhaustion of $\Omega$ by bounded open sets. Its dual can be identified with $L^{\infty}_{c} (\Omega)$ (see \cite{Horvath}, Thm. 4, p.182). 
Weakly compact sets  in $L^{1}_{loc} (\Omega)$ are characterized as follows:
\begin{lemma}  \label{lem:L1Compact}
Let ${\cal K} \subset L^1_{loc} (\Omega)$ and define ${\cal K}_B := \{ f |_B \, \colon \, f \in {\cal K} \}$. 
Then the following statements are equivalent.
\begin{enumerate} 
\item[{\rm \bf (i)}]
${\cal K}$ is  compact in the weak topology $\sigma(L^{1}_{loc}(\Omega) , L^{\infty}_c(\Omega))$;
\item[{\rm \bf (ii)}]
for each ball $B \Subset \Omega$, ${\cal K}_B \subset L^1 (B)$ is  compact in the weak topology 
$\sigma(L^{1} (B) , L^{\infty} (B))$;
\item[{\rm \bf (iii)}]
${\cal K} $ is weakly closed and for each ball $B \Subset \Omega$, the set ${\cal K}_B$ is 
equi-integrable in $L^1 (B)$; 
\item[{\rm \bf (iv)}]
${\cal K}$ is   sequentially compact in the weak topology $\sigma(L^{1}_{loc}(\Omega) , L^{\infty}_c(\Omega))$;
\item[{\rm \bf (v)}]
for each ball $B \Subset \Omega$, ${\cal K}_B$ is  sequentially compact in the weak topology 
$\sigma(L^{1} (B) , L^{\infty} (B))$.
\end{enumerate}
\end{lemma}
\proof
{\bf (i)} $\Longrightarrow$ {\bf (iv)} Since $L^1_{loc}(\Omega)$ with weak topology is a locally convex topological vector space with a countable family of semi-norms that separate points, its weak topology is metrisable. The equivalence follows immediately.
\medskip

{\bf (ii)} $\Longleftrightarrow$ {\bf (iii)}
This equivalence is the Dunford-Pettis Theorem.

\medskip

{\bf (ii)} $\Longleftrightarrow$ {\bf (v)}
This is the Eberlein-Smulian Theorem. 

\medskip

{\bf (iv)} $\Longleftrightarrow$ {\bf (v)}
Since the dual of $L^{1}_{loc} (\Omega)$ is $L^{\infty}_c (\Omega)$, the conclusion 
follows from the following equivalences:
\vspace{-2mm}
\begin{eqnarray*}
  f_n \to f \hbox{ weakly in } L^1_{loc}(\Omega)  &\Longleftrightarrow&
  \int_{\Omega} (f_n - f) \xi \to 0 \quad \forall \xi \in L^{\infty}_c (\Omega) \\ 
  &\Longleftrightarrow&
  \int_{\Omega} (f_n - f) (\chi_B \xi) \to 0 \quad \forall \xi \in L^{\infty} (\Omega) , 
  \quad \forall B \Subset \Omega \\
  &\Longleftrightarrow& 
  \int_{B} (f_n - f) \xi \to 0 \quad \forall \xi \in L^{\infty} (B) ,
  \quad \forall B \Subset \Omega \\
  &\Longleftrightarrow&
    f_n \to f \hbox{ weakly in }   L^1 (B) ,
   \quad \forall B \Subset \Omega.
\end{eqnarray*}
\qed

 \subsection{ Proof of proposition \ref{prop:shave}}
Fix $\varepsilon >0$. 
Let $\xi \in C_c^{\infty} (\Omega_*)$ and take
$\frac{\xi^2}{u + \epsilon }$ as test function in the weak formulation for $u$. We then get
$$
  \int_{\Omega_*}
 \nabla u \cdot \nabla \Big(\frac{\xi^2}{u + \varepsilon} \Big)
+ \int_{\Omega_*} W u \Big(\frac{\xi^2}{u + \epsilon } \Big)
  = \int_{\Omega_*} f \Big(\frac{\xi^2}{u + \epsilon }\Big) .
$$
So,
$$
 - \int_{\Omega_*}
  \left| \nabla \xi - \Big(\frac{\xi}{u + \varepsilon}\Big) \nabla u \right|^2
 + \int_{\Omega_*} \left\{
     |\nabla \xi|^2 +W u \Big(\frac{\xi^2}{u + \epsilon } \Big)
   \right\} 
  = \int_{\Omega_*} f \Big(\frac{\xi^2}{u + \epsilon } \Big).
$$
That is,
$$
\int_{\Omega_*} \left\{
     |\nabla \xi|^2 +W u \Big(\frac{\xi^2}{u + \epsilon } \Big)
   \right\} 
  =
\int_{\Omega_*} f \Big(\frac{\xi^2}{u + \epsilon } \Big)  +
\int_{\Omega_*}
  \left| \nabla \xi - \Big(\frac{\xi}{u + \varepsilon}\Big) \nabla u \right|^2.
$$
Since the set $\{ u = 0 \} $ has measure zero, Lebesgue theorem and Fatou's Lemma yield
$$
  \int_{\Omega_* } \left\{
     |\nabla \xi|^2 + W \xi^2 
   \right\} 
  \geq
\int_{\Omega_*} \Big( \frac{f}{u}\Big) \xi^2  + \int_{\Omega_*}
  \left| \nabla \xi - \Big(\frac{\xi}{u}\Big) \nabla u \right|^2.
$$

\subsection{Proof of Lemma~\ref{brb}}

Choose the subsequence $\{z_{n_k}\}$  such that 
$$
\| z_{n_{k+1}}- z_{n_k}\| \leq \frac{1}{2^k}, \;\; k=1,2, \cdot \cdot \cdot
$$
Let 
$$
\hat{z}_m := \Sigma_{k=1} ^m | z_{n_{k+1}} - z_{n_k} |.
$$
Then $\|\hat{z}_m \| \leq 1$   and  hence $\{\hat{z}_m\}$ converges weakly (up to a subsequence, again denoted by $\{\hat{z}_m\}$) to some $\hat{z} \in {\mathcal Z} $.  Since ${\mathcal Z} \hookrightarrow L^1(K)$, we have that  $\{\hat{z}_m\}$ is also bounded in $L^1(K)$. We also note that, by monotone convergence theorem, $\{\hat{z}_m\}$ converges in $L^1(K)$.  Hence necessarily $\hat{z}_m \leq \hat{z}$. By a telescoping series,
$$
|z_{n_{m+1}}| \leq \hat{z}_m  + |z_{n_1} | \leq \hat{z} + |z_{n_1}| := z^* \;\; \forall m.
$$
\qed

{\bf Acknowledgement:} 
The authors thank Prof.~V. Moroz for some helpful discussions. 
We also thank Prof.~Y. Pinchover for  for encouraging this work.


\end{document}